\documentclass[pdflatex,sn-mathphys-num]{sn-jnl}

\usepackage{amssymb}
\usepackage{amsmath}
\usepackage{amsthm}

\usepackage{lineno}

\usepackage{mathtools}

\usepackage{amsfonts,amssymb,amsmath,amsthm,stmaryrd}

\usepackage{xcolor}

\usepackage{ulem}

\usepackage{graphicx}%
\usepackage{multirow}%
\usepackage{amsmath,amssymb,amsfonts}%
\usepackage{amsthm}%
\usepackage{mathrsfs}%
\usepackage[title]{appendix}%
\usepackage{xcolor}%
\usepackage{textcomp}%
\usepackage{manyfoot}%
\usepackage{booktabs}%
\usepackage{algorithm}%
\usepackage{algorithmicx}%
\usepackage{algpseudocode}%
\usepackage{listings}%

\numberwithin{equation}{section}

\theoremstyle{plain}
\newtheorem{theorem}{Theorem}[section]
\newtheorem{proposition}{Proposition}[section]
\newtheorem{lemma}{Lemma}[section]
\newtheorem{corollary}{Corollary}[section]

\theoremstyle{definition}
\newtheorem{definition}{Definition}[section]

\theoremstyle{remark}
\newtheorem{remark}{Remark}[section]

\let\coloneq\relax
\DeclareRobustCommand{\coloneq}{\mathrel{\mathop{:}\! =}}

\makeatletter
\def\polhk#1{\setbox0=\hbox{#1}{\ooalign{\hidewidth\lower1.5ex\hbox{`}\hidewidth\crcr\unhbox0}}}
\makeatother

\allowdisplaybreaks
\begin{document}

\title{Viscosity and minimax solutions for path-dependent Hamilton--Jacobi equations in infinite dimensions and related differential games}


\author[1]{\fnm{Erhan} \sur{Bayraktar}}\email{erhan@umich.edu}

\author[2,3]{\fnm{Mikhail} \sur{Gomoyunov}}\email{m.i.gomoyunov@gmail.com}

\author*[4]{\fnm{Christian} \sur{Keller}}\email{christian.keller@ucf.edu}

\affil[1]{Department of Mathematics, University of Michigan, Ann Arbor, MI 48109, United States}

\affil[2]{N.N. Krasovskii Institute of Mathematics and Mechanics of the Ural Branch of the Russian Academy of Sciences,
            16 S. Kovalevskaya Str.,
            Ekaterinburg,
            620077,
            Russia}

\affil[3]{Ural Federal University,
            19 Mira Str.,
            Ekaterinburg,
           620002,
           Russia}

\affil*[4]{School of Data, Mathematical, and Statistical Sciences, University of Central Florida,
           Orlando,
           FL
          32816,
           United States}


\abstract{We establish new results for path-dependent Hamilton--Jacobi equations with nonlinear monotone, and coercive operators on Hilbert space, which were initially studied in Bayraktar and Keller [\textit{J. Funct. Anal.}, 275~(8) (2018), pp.~2096--2161].
    Under more general assumptions than in the cited paper (and more general than in the finite-dimensional case as well), we prove the uniqueness of a minimax solution of a terminal-value problem for the equation under consideration and the existence of such a solution on the whole path space.
    We introduce a new notion of a viscosity solution for this problem and show the equivalence of this notion to the notion of a minimax solution, which implies the corresponding existence and uniqueness theorem for viscosity solutions.
    In addition, we obtain a stability result for viscosity solutions using the half-relaxed limits method.
    As applications, we prove two theorems on the existence and characterization of value of a zero-sum differential game for a time-delay (path-dependent) evolution equation.
    The first theorem pertains to the case of non-anticipative (Elliott--Kalton) strategies and is related to the results on viscosity solutions, and the second theorem deals with the case of feedback (Krasovskii--Subbotin) strategies and is based on the results on minimax solutions.}




\maketitle

\section{Introduction}

    We consider a path-dependent Hamilton--Jacobi equation with a nonlinear, monotone, and coercive operator on Hilbert space.
    We study generalized (minimax, viscosity) solutions of a terminal value problem for this equation and give applications to the study of differential games for time-delay (path-dependent) evolution equations.
    The goal of the paper is to strengthen and develop the main results of \cite{Bayraktar_Keller_2018}.

    For a detailed literature review and context, the reader is referred to \cite{Bayraktar_Keller_2018}.
    We limit ourselves here to mentioning that the theory of minimax solutions of Hamilton--Jacobi equations originated in \cite{Subbotin_1980}, see also \cite{Subbotin_1993,Subbotin_1995}.
    For the case of path-dependent Hamilton--Jacobi equations, the development of the theory of minimax solutions began with \cite{Lukoyanov_2001,Lukoyanov_2003a}.
    For the current state of research, see \cite{Gomoyunov_Lukoyanov_RMS_2024}.
    The infinite-dimensional case was considered in \cite{Bayraktar_Keller_2018}.
    The theory of viscosity solutions of Hamilton--Jacobi equations originated in \cite{Crandall_Lions_1983,Crandall_Evans_Lions_1984}.
    For the case of path-dependent equations, the development of the theory of viscosity solutions began with \cite{Lukoyanov_2007}.
    Note that there are quite a large number of works and approaches in this direction \cite{Cosso_Gozzi_Rosestolato_Russo_2021,Ekren_Keller_Touzi_Zhang_2014,Ekren_Touzi_Zhang_2016_1,Plaksin_2021_SIAM,Ren_Touzi_Zhang_2014,Zhou_2018_COCV}.
    The infinite-dimensional case was addressed in \cite{Bayraktar_Keller_2018,Zhou_2022}.

    We work with more general assumptions compared to \cite{Bayraktar_Keller_2018} under which the existence and uniqueness of a minimax solution is proved.
    More specifically, we significantly relax the assumption about the nature of the dependence of the Hamiltonian of the considered Hamilton--Jacobi equation on the path variable.
    Our generalization allows for general delays and thus seems natural and important from the point of view of applications to optimal control problems and differential games for time-delay evolution equations.
    Note also that our assumptions on the Hamiltonian of the equation are more general than those considered in the finite-dimensional case (cf. \cite{Gomoyunov_Lukoyanov_RMS_2024}).
    An additional contribution compared to \cite{Bayraktar_Keller_2018} is that we prove the existence of a minimax solution on the entire path space under consideration, and not on a locally compact subset of it.

    The proof of the theorem on existence and uniqueness of the minimax solution follows the same ``five-step-scheme'' as in \cite{Bayraktar_Keller_2018}.
    The main difference is the use of a different penalty function in the proof of the comparison result.
    As a basis, a special function $\Upsilon$ introduced in the works of J. Zhou is chosen
    (see \cite{Zhou_2020} and also \cite{Zhou_2022,Zhou_2023}).
    This function $\Upsilon$ was also used to prove close results in the finite-dimensional case
    (see, e.g., \cite{Zhou_2020,Gomoyunov_Plaksin_2023} for viscosity solutions of path-dependent Hamilton--Jacobi--Bellman equations and \cite{Gomoyunov_Lukoyanov_Plaksin_2021} for minimax solutions of path-dependent Hamilton--Jacobi equations).
    Since our case is infinite-dimensional, or more precisely, since we work in the variational framework involving Gelfand triples, and the conditions are quite general, additional difficulties arose in justifying suitable smoothness properties of $\Upsilon$, allowing the use of the functional chain rule for this function.
    This difficulty was overcome by using and developing the technique of coinvariant derivatives.
    The proposed approach is quite general and can be useful in other situations when working with other functions.
    For a different approach in the infinite-dimensional case, or more precisely, in the semigroup framework, we refer the reader to \cite{Zhou_2022}, where  path-dependent Hamilton--Jacobi--Bellman equations with unbounded linear operators on Hilbert space are studied.

    We propose a new definition of a viscosity solution of the considered terminal value problem.
    This notion goes back to the constructions from previous works (see, e.g., \cite{Lukoyanov_2006}     and \cite{Bayraktar_Keller_2018}).
    In connection with the infinite-dimensional case under consideration, a singular component is introduced into the definition of test functions, which requires careful processing in the proofs.
    We prove the equivalence of the definitions of minimax and viscosity solutions and thereby obtain a result on the existence and uniqueness of the viscosity solution.
    In addition, we develop the half-relaxed limits method and prove a stability result for viscosity solutions.
    Note that our method of proving this stability result goes back to the theory of minimax solutions.

    We apply the main results to study a zero-sum differential game for a time-delay evolution equation.
    As an illustrative example, one can consider a differential game for a time-delay quasi-linear parabolic partial differential equation.
    We first consider a formulation of the game in the classes of non-anticipative strategies (Elliott--Kalton strategies \cite{Elliott_Kalton_1972}, see also, e.g., \cite{Fleming_Soner_2006,Yong_2015}).
    We introduce lower and upper game values and prove that they coincide with unique viscosity (and, hence, minimax) solutions of the corresponding terminal-value problems for lower and upper Isaacs equations respectively.
    In general, the proof follows the same lines as in the classical (non-path-dependent and finite-dimensional) case, which also shows the naturalness of the viscosity solution concept introduced in the paper and its convenience when working with applications.
    Note also that rather general assumptions about the nature of the players' control actions are considered: it is only assumed that their values belong to compact topological spaces.
    In this connection, it was necessary to obtain a special version of the measurable selection theorem.
    In the case where the Isaacs condition holds, the lower and upper game values coincide, which means that game has value in the classes of non-anticipative strategies.

    In addition, we consider a formulation of the game in the classes of feedback strategies (Krasovskii--Subbotin strategies \cite{Krasovskii_Subbotin_1988}, see also \cite{Krasovskii_Krasovskii_1995,Subbotin_1995}).
    Under the Isaacs condition, we prove that the game has value in this formulation and that this value coincides with the unique minimax (and, hence, viscosity) solution of the corresponding terminal-value problem for the Isaacs equation.
    The proof follows the same lines as in \cite{Bayraktar_Keller_2018} but with another choice of the Lyapunov--Krasovskii function, which is now based on the function $\Upsilon$.
    In the finite-dimensional setting, close results were obtained in \cite[Section 3]{Gomoyunov_Lukoyanov_RMS_2024} (see also \cite{Gomoyunov_Lukoyanov_2024}).

\section{Setting}

    Let $V \subset H \subset V^\ast$ be a Gelfand triple in the sense of \cite[Chapter 1, Definition 1.3]{Hu_Papageorgiou_2000} (see also \cite[Definition 23.11]{ZeidlerIIA}), i.e., $H$ is a separable Hilbert space, $V$ is a reflexive, separable Banach space, and $V$ is continuously and densely embedded into $H$.
    Moreover, it is assumed that the embedding from $V$ into $H$ is compact.
    Denote by $\|\cdot\|$, $|\cdot|$, and $\|\cdot\|_\ast$ the norms on $V$, $H$, and $V^\ast$.
    We also use $|\cdot|$ for norms on Euclidean spaces.
    We write $(\cdot, \cdot)$ for the inner product in $H$ and $\langle \cdot, \cdot \rangle$ for the duality pairing between $V$ and $V^\ast$.
    We also have $\langle h, v \rangle=(h, v)$ for all $h \in H$ and $v \in V$ (see \cite[Section 23.4]{ZeidlerIIA}).

    Fix $T > 0$.
    The space $C([0, T], H)$ is equipped with the supremum norm $\|\cdot\|_\infty$ and the space $[0, T] \times C([0, T], H)$ is equipped with the pseudometric
    \begin{equation}\label{pseudometric}
        \mathbf{d}_\infty((t_1, x_1), (t_2, x_2))
        \coloneq |t_1 - t_2| + \|x_1(\cdot \wedge t_1) - x_2(\cdot \wedge t_2)\|_\infty,
    \end{equation}
    where $(t_1, x_1)$, $(t_2, x_2) \in [0, T] \times C([0, T], H)$, $a \wedge b \coloneq \min\{a, b\}$ for all $a$, $b \in \mathbb{R}$ and, respectively, $x(s \wedge t) = x(s)$ if $s \in [0, t]$ and $x(s \wedge t) = x(t)$ if $s \in [t, T]$ for every $(t, x) \in [0, T] \times C([0, T], H)$.

    Fix $p \geq 2$ and $q > 1$ with $1 / p + 1 / q = 1$.
    Let $A \colon [0, T] \times V \to V^\ast$ be given and satisfy the following hypotheses (see \cite[Section 2.2]{Bayraktar_Keller_2018}).

    $\mathbf{H}(A)$:
    (i)
        For every $x$, $v \in V$, the map $t \mapsto \langle A(t, x), v \rangle$, $[0, T] \to \mathbb{R}$, is (Lebesgue) measurable.

    (ii)
        Monotonicity:
        For almost every (a.e.) $t \in [0, T]$ and every $x$, $y \in V$,
        \begin{equation*}
            \langle A(t, x) - A(t, y), x - y \rangle
            \geq 0.
        \end{equation*}

    (iii)
        Hemicontinuity:
        For a.e. $t \in [0, T]$ and every $x$, $y$, $v \in V$, the map $s \mapsto \langle A(t, x + s y), v \rangle$, $[0, 1] \to \mathbb{R}$, is continuous.

    (iv)
        Boundedness:
        There exist a function $a_1 \in L^q(0, T; \mathbb{R}_+)$ and a number $c_1 \geq 0$ such that, for a.e. $t \in [0, T]$ and every $x \in V$,
        \begin{equation*}
            \|A(t, x)\|_\ast
            \leq a_1(t) + c_1 \|x\|^{p - 1}.
        \end{equation*}

    (v)
        Coercivity:
        There exists a number $c_2 > 0$ such that, for a.e. $t \in [0, T]$ and every $x \in V$,
        \begin{equation*}
            \langle A(t, x), x \rangle
            \geq c_2 \|x\|^p.
        \end{equation*}

    \begin{definition}\label{definitionXL}
        Given $L \geq 0$ and $(t_0, x_0) \in [0, T) \times C([0, T], H)$, denote by $\mathcal{X}^L(t_0, x_0)$ the set of all $x \in C([0, T], H)$ such that $x(t) = x_0(t)$ for every {\color{black} $t \in [0, t_0]$}, $x|_{(t_0, T)} \in W_{p q}(t_0, T)$ and there exists a function $f^x \in L^2(t_0, T; H)$ such that, for a.e. $t \in [t_0, T]$,
        \begin{equation*}
           x^\prime(t) + A(t, x(t)) = f^x(t),
           \quad |f^x(t)|
           \leq L (1 + \|x(\cdot \wedge t)\|_\infty).
        \end{equation*}
        Here, $W_{p q}(t_0, T)$ denotes the space of all $x \in L^p(t_0, T; V)$ possessing a generalized derivative $x^\prime \in L^q(t_0, T; V^\ast)$.
        This space is equipped with the norm $\|\cdot\|_{W_{pq}(t_0, T)}$ defined by
        \begin{equation*}
            {\color{black} \|x\|_{W_{p q}(t_0, T)}}
            \coloneq \|x\|_{L^p(t_0, T; V)} + \|x^\prime\|_{L^q(t_0, T; V^\ast)}.
        \end{equation*}
        See \cite[Definition 2.6]{Bayraktar_Keller_2018} and \cite[Proposition 23.23]{ZeidlerIIA} for more details.
    \end{definition}

    \begin{remark}
        Given a differentiable function $y \colon [t_0, T] \to V$, we will denote its derivative by $\dot{y}$.
        If there is an $x \in W_{p q}(t_0, T)$ such that $x(t) = y(t)$ for a.e. $t\in (t_0, T)$, then $x^\prime(t) = \dot{y}(t)$ for a.e. $t\in (t_0, T)$.
    \end{remark}

    By \cite[Proposition 2.10]{Bayraktar_Keller_2018}, the set $\mathcal{X}^L(t_0, x_0)$ is non-empty and compact in $C([0, T], H)$.

    \begin{definition}
        Given $t_\ast \in [0, T)$, denote by $\mathcal{C}_V^{1, 1}([t_\ast, T] \times C([0, T], H))$ the set of all continuous functions $\varphi \colon [t_\ast, T] \times C([0, T], H) \to \mathbb{R}$ for each of which there exist continuous functions $\partial_t \varphi \colon [t_\ast, T] \times C([0, T], H) \to \mathbb{R}$ and $\partial_x \varphi \colon [t_\ast, T] \times C([0, T], H) \to H$, called path derivatives of $\varphi$, such that:

        (i)
            for every $x \in C([0, T], H)$ with $x(t) \in V$ for a.e.~$t \in [t_\ast, T]$, it holds that $\partial_x \varphi(t, x) \in V$ for a.e. $t \in [t_\ast, T]$,

        (ii)
            for every $t_0 \in [t_\ast, T)$, every $x_0 \in C([0, T], H)$, every $x \in C([0, T], H)$
            with $x({\color{black}s}) = x_0({\color{black}s})$ for every ${\color{black}s} \in [0, t_0]$ and $x|_{(t_0, T)} \in W_{p q}(t_0, T)$, and every $t \in [t_0, T]$, the functional chain rule holds:
            \begin{equation*}
                \varphi(t, x) - \varphi(t_0, x_0)
                = \int_{t_0}^{t} \bigl( \partial_t \varphi(s, x)
                + \langle x^\prime(s), \partial_x \varphi(s, x) \rangle \bigr) \, ds.
            \end{equation*}
        See \cite[Section 2.4]{Bayraktar_Keller_2018} for more details.
    \end{definition}

    {\color{black}
    In a similar way, we define a set $\mathcal{C}^{1, 1}_V([t_\ast, T] \times C([0, T], H) \times C([0, T], H))$, where $t_\ast \in [0, T)$.
    See also \cite[Definition 2.16]{Bayraktar_Keller_2018}.
    \begin{definition}
        Given $t_\ast \in [0, T)$, denote by $\mathcal{C}_V^{1, 1}([t_\ast, T] \times C([0, T], H) \times C([0, T], H))$ the set of all continuous functions $\varphi \colon [t_\ast, T] \times C([0, T], H) \times C([0, T], H) \to \mathbb{R}$ for each of which there exist continuous functions $\partial_t \varphi \colon [t_\ast, T] \times C([0, T], H) \times C([0, T], H) \to \mathbb{R}$, $\partial_x \varphi \colon [t_\ast, T] \times C([0, T], H) \times C([0, T], H) \to H$, and $\partial_y \varphi \colon [t_\ast, T] \times C([0, T], H) \times C([0, T], H) \to H$ such that:

        (i)
            for any $x$, $y \in C([0, T], H)$ with $x(t)$, $y(t) \in V$ for a.e. $t \in [t_\ast, T]$, it holds that $\partial_x \varphi(t, x, y)$, $\partial_y \varphi(t, x, y) \in V$,

        (ii)
            for every $t_0 \in [t_\ast, T)$, any $x_0$, $y_0 \in C([0, T], H)$, any $x$, $y \in C([0, T], H)$ with $x(s) = x_0(s)$, $y(s) = y_0(s)$ for every $s \in [0, t_0]$ and $x|_{(t_0, T)}$, $y|_{(t_0, T)} \in W_{p q}(t_0, T)$, and every $t \in [t_0, T]$,
            \begin{align*}
                & \varphi(t, x, y) - \varphi(t_0, x_0, y_0) \\
                & \quad = \int_{t_0}^{t} \bigl( \partial_t \varphi(s, x, y)
                + \langle x^\prime(s), \partial_x \varphi(s, x, y) \rangle
                + \langle y^\prime(s), \partial_y \varphi(s, x, y) \rangle \bigr) \, ds.
            \end{align*}
    \end{definition}
    }

    Let functions $F \colon [0, T] \times C([0, T], H) \times H \to \mathbb{R}$ and $h \colon C([0, T], H) \to \mathbb{R}$ be given and satisfy the following hypotheses.

    $\mathbf{H}(F)$:
    (i)
        The function $F$ is continuous.

    (ii)
        There exists a number $L_0 \geq 0$ such that, for every $t \in [0, T]$, every $x \in C([0, T], H)$, and every $z_1$, $z_2 \in H$,
        \begin{equation*}
            F(t, x, z_1) - F(t, x, z_2)
            \leq L_0 (1 + \|x(\cdot \wedge t)\|_\infty) |z_1 - z_2|.
        \end{equation*}

    (iii)
        For every $L \geq 0$ and every $(t_0, x_0) \in [0, T) \times C([0, T], H)$, there exists a modulus of continuity $m_{L, t_0, x_0}$ such that, for every $\varepsilon > 0$, every $t \in [t_0, T]$, and every $x_1$, $x_2 \in \mathcal{X}^L(t_0, x_0)$,
        \begin{equation} \label{condition_F}
            \begin{aligned}
                & F \bigl( t, x_1, \varepsilon^{- 1} (x_1(t) - x_2(t)) \bigr)
                - F \bigl( t, x_2, \varepsilon^{- 1} (x_1(t) - x_2(t)) \bigr) \\
                & \quad \leq m_{L, t_0, x_0} \Bigl( \varepsilon^{- 1} |x_1(t) - x_2(t)|
                \|x_1(\cdot \wedge t) - x_2(\cdot \wedge t)\|_\infty
                + \|x_1(\cdot \wedge t) - x_2(\cdot \wedge t)\|_\infty \Bigr).
            \end{aligned}
        \end{equation}

    $\mathbf{H}(h)$:
        The function $h$ is continuous.

    \begin{remark} \label{remark_conditions_comparison}
        The hypotheses $\mathbf{H}(F)$ and $\mathbf{H}(h)$ are the same as in \cite[Section 3]{Bayraktar_Keller_2018} except for the condition $\mathbf{H}(F)$(iii).
        Namely, in \cite[Section 3]{Bayraktar_Keller_2018}, the following condition is considered:
        for every $L \geq 0$ and every $(t_0, x_0) \in [0, T) \times C([0, T], H)$, there exists a modulus of continuity $m_{L, t_0, x_0}$ such that, for every $\varepsilon > 0$, every $t \in [t_0, T]$, and every $x_1$, $x_2 \in \mathcal{X}^L(t_0, x_0)$,
        \begin{equation} \label{condition_F_old}
            \begin{aligned}
                & F \bigl( t, x_1, \varepsilon^{- 1} (x_1(t) - x_2(t)) \bigr)
                - F \bigl( t, x_2, \varepsilon^{- 1} (x_1(t) - x_2(t)) \bigr) \\
                & \quad \leq m_{L, t_0, x_0} \Bigl( \varepsilon^{- 1} |x_1(t) - x_2(t)|^2
                + \|x_1(\cdot \wedge t) - x_2(\cdot \wedge t)\|_\infty \Bigr).
            \end{aligned}
        \end{equation}
        Since \eqref{condition_F_old} implies \eqref{condition_F}, this condition is stronger than $\mathbf{H}(F)$(iii).
    \end{remark}

    Consider the terminal-value problem for the path-dependent Hamilton--Jacobi equation
    \begin{subequations} \label{TVP}
    \begin{equation} \label{HJ}
        \begin{gathered}
        \partial_t u(t, x) - \langle A(t, x(t)), \partial_x u(t, x) \rangle + F(t, x, \partial_x u(t, x))
        = 0, \\
        (t, x) \in [0, T) \times C([0, T], H),
        \end{gathered}
    \end{equation}
    under the right-end boundary condition
    \begin{equation} \label{boundary_condition}
        u(T, x)
        = h(x),
        \quad  x \in C([0, T], H).
    \end{equation}
    \end{subequations}
   Here, $\partial_t u$ and $\partial_x u$ are the path derivatives of the unknown $u$.

\section{Minimax and viscosity solutions}

    In this section, we give definitions of minimax and viscosity solutions of the terminal-value problem \eqref{TVP}.
    We formulate a theorem on the existence and uniqueness of a minimax solution (we give a proof of the theorem in Section \ref{S:proof}, which is preceded by a discussion of the functional chain rule in Section \ref{S:Functional_chain_rule}).
    Then, we prove a result on the equivalence of the definitions of minimax and viscosity solutions.
    As a consequence, we obtain a theorem on the existence and uniqueness of a viscosity solution.
    In addition, we prove a stability theorem for viscosity solutions.

\subsection{Minimax solutions}

    We start with the definition of minimax solutions (cf. \cite[Definition 3.2]{Bayraktar_Keller_2018}).
    Note that we require slightly less regularity compared to the corresponding definition in \cite{Bayraktar_Keller_2018} (see also the regularity requirements for non-smooth solutions in \cite{Jiang_Keller_2025} as well as in \cite{Gomoyunov_2024}).

    \begin{definition} \label{D:minimax_solutions}
        Let $L\ge 0$ and $u \colon [0, T] \times C([0,T],H) \to \mathbb{R}$.

        (i)
            We call $u$ a minimax $L$-subsolution of \eqref{TVP} if
            the restrictions of $u$ to the sets $[t^\ast, T] \times \mathcal{X}^{L^\ast}(t^\ast, x^\ast)$, $t^\ast \in [0, T)$, $x^\ast \in C([0,T], H)$, $L^\ast \geq 0$, are upper semi-continuous,
            $u$ satisfies the inequality $u(T, x) \leq h(x)$ for all $x \in C([0,T],H)$,
            and, for every $(t_0, x_0, z) \in [0, T) \times C([0,T], H) \times H$, there exists an $x \in \mathcal{X}^L(t_0, x_0)$ such that, for every $t \in [t_0, T]$,
            \begin{equation}\label{E:Minimax:L:Subsolution}
                u(t_0, x_0)
                \leq \int_{t_0}^{t} \bigl( (- f^x(s), z) + F(s, x, z) \bigr) \, ds + u(t, x).
            \end{equation}

        (ii)
            We call $u$ a minimax $L$-supersolution of \eqref{TVP} if
            the restrictions of $u$ to the sets $[t^\ast, T] \times \mathcal{X}^{L^\ast}(t^\ast, x^\ast)$, $t^\ast\in [0, T)$, $x^\ast \in C([0,T], H)$, $L^\ast \geq 0$, are lower semi-continuous,
            $u$ satisfies the inequality $u(T, x) \geq h(x)$ for all $x \in C([0,T],H)$, and, for every $(t_0, x_0, z) \in [0, T) \times C([0,T], H) \times H$, there exists an $x \in \mathcal{X}^L(t_0, x_0)$ such that, for every $t \in [t_0, T]$,
            \begin{equation}\label{E:Minimax:L:Supersolution}
                u(t_0, x_0)
                \geq \int_{t_0}^{t} \bigl( (- f^x(s), z) + F(s, x, z) \bigr) \, ds + u(t, x).
            \end{equation}

        {\rm (iii)}
            We call $u$ a minimax $L$-solution of \eqref{TVP} if
            the restrictions of $u$ to the sets $[t^\ast, T] \times \mathcal{X}^{L^\ast}(t^\ast, x^\ast)$, $t^\ast\in [0, T)$, $x^\ast\in C([0,T], H)$, $L^\ast \geq 0$, are continuous,
            $u$ satisfies the equality $u(T, x) = h(x)$ for all $x \in C([0,T], H)$,
            and, for every $(t_0, x_0, z) \in [0, T) \times C([0,T], H) \times H$, there exists an $x \in \mathcal{X}^L(t_0, x_0)$ such that, for every $t \in [t_0, T]$,
            \begin{equation} \label{E:Minimax:L:Solution}
                u(t_0, x_0)
                = \int_{t_0}^{t} \bigl( (- f^x(s), z) + F(s, x, z) \bigr) \, ds + u(t, x).
            \end{equation}

        {\rm (iv)}
            We call $u$ a minimax subsolution (resp. supersolution, resp. solution) of \eqref{TVP} if $u$ is a minimax $L$-subsolution (resp. $L$-supersolution, resp. $L$-solution) of \eqref{TVP} for some $L \geq 0$.
    \end{definition}

    Note that (cf. \cite[Proposition 3.7]{Bayraktar_Keller_2018}), given $L \geq 0$, a function $u \colon [0, T] \times C([0,T],H) \to \mathbb{R}$ is a minimax $L$-solution of \eqref{TVP} if and only if $u$ is a minimax $L$-supersolution as well as a minimax $L$-subsolution of \eqref{TVP}.

    The theorem below is a generalization of \cite[Theorem 5.8]{Bayraktar_Keller_2018} to the case of the weaker condition $\mathbf{H}(F)$ and the larger space $C([0,T], H)$.
    Recall that, in \cite{Bayraktar_Keller_2018}, existence was only established for minimax solutions defined on a smaller set $[0, T] \times \Omega$.
    Here, $\Omega$ is the union of $\Omega^L \coloneq \mathcal{X}^L(0, x_\ast)$ over $L \geq 0$, where $x_\ast \in H$ is fixed and considered as a constant function on $[0, T]$.

    \begin{theorem} \label{theorem_minimax_existence_uniqueness}
        Let $\mathbf{H}(A)$, $\mathbf{H}(F)$, and $\mathbf{H}(h)$ be satisfied.
        Then, there exists a unique minimax solution $u$ of \eqref{TVP}.
        Moreover, $u$ is a minimax $L_0$-solution of \eqref{TVP}, where $L_0 \geq 0$ is the number from $\mathbf{H}(F)${\rm(ii)}.
    \end{theorem}

    \begin{remark}\label{R:existenceComparedToBK18}
        Note that the existence result in  \cite{Bayraktar_Keller_2018} for minimax solutions on $[0, T] \times \Omega$ relied on showing that semi-continuous envelopes of a suitable function are minimax sub/supersolutions.
        In particular, the compactness of $\mathcal{X}^L(0, x_\ast)$ was crucial to establish
        the infinite-dimensional counterpart \cite[Proposition 2.12]{Bayraktar_Keller_2018}
        of the stability result \cite[Proposition 4.2]{Lukoyanov_2003a}, which was needed for establishing the mentioned result for the semi-continuous envelopes.
    \end{remark}

    Similarly to the proof of \cite[Theorem 5.8]{Bayraktar_Keller_2018}, the proof of Theorem \ref{theorem_minimax_existence_uniqueness} is carried out according to the ``five-step-scheme'' described in \cite[Section 1.3]{Bayraktar_Keller_2018}.
    Note that the only step where the condition $\mathbf{H}(F)$(iii) is used is the first step, called ``doubled comparison'' and precisely formulated in \cite[Theorem 4.2]{Bayraktar_Keller_2018}, and the only step where the use of the larger space $C([0,T], H)$ in contrast to $\Omega$ (see Remark \ref{R:existenceComparedToBK18}) matters is the fourth step, called ``Perron'' (see \cite[Section 5]{Bayraktar_Keller_2018}).
    Thus, these are the only facts that we should check under our more general conditions.
    This is dealt with in Section \ref{S:proof}.

    We also need an infinitesimal formulation for minimax $L$-subsolutions and $L$-super\-solutions.
    \begin{proposition} \label{P:Minimax:Infinitesimal}
        Fix a number $L \geq 0$
        and a function $u \colon [0, T] \times C([0,T],H) \to \mathbb{R}$.
        Then, the following statements hold:

        {\rm (i)}
            $u$ is a minimax $L$-subsolution of \eqref{TVP} if and only if
            the restrictions of $u$ to the sets $[t^\ast, T] \times \mathcal{X}^{L^\ast}(t^\ast, x^\ast)$, $t^\ast\in [0, T)$, $x^\ast\in C([0, T], H)$, $L^\ast \geq 0$, are upper semi-continuous,
            $u$ satisfies the inequality $u(T, x) \leq h(x)$ for all $x \in C([0, T], H)$,
            and, for every $(t_0, x_0) \in [0, T) \times C([0, T], H)$ and every $z \in H$,
            \begin{equation} \label{subsolution_infinitesimal}
                \begin{aligned}
                    \lim_{\delta \to 0^+}
                    \sup \biggl\{
                    & \frac{1}{t - t_0}
                    \biggl( u(t, x) - u(t_0, x_0)
                    + \int_{t_0}^{t} (- f^x(s), z) \, ds \biggr) \colon \\
                    & \quad
                    t \in (t_0, t_0 + \delta], \, x \in \mathcal{X}^L(t_0, x_0) \biggr\}
                    + F(t_0, x_0, z)
                    \geq 0.
                \end{aligned}
            \end{equation}

        {\rm (ii)}
            $u$ is a minimax $L$-supersolution of \eqref{TVP}  if and only if
            the restrictions of $u$ to the sets $[t^\ast, T] \times \mathcal{X}^{L^\ast}(t^\ast, x^\ast)$, $t^\ast\in [0, T)$, $x^\ast \in C([0, T], H)$, $L^\ast \geq 0$, are lower semi-continuous,
            $u$ satisfies the inequality $u(T, x) \geq h(x)$ for all $x \in C([0, T], H)$,
            and, for every $(t_0, x_0) \in [0, T) \times C([0,T],H)$ and every $z \in H$,
            \begin{equation} \label{supersolution_infinitesimal}
                \begin{aligned}
                    \lim_{\delta \to 0^+}
                    \inf \biggl\{
                    & \frac{1}{t - t_0}
                    \biggl( u(t, x) - u(t_0, x_0)
                    + \int_{t_0}^{t} (- f^x(s), z) \, ds \biggr) \colon \\
                    & \quad
                    t \in (t_0, t_0 + \delta], \, x \in \mathcal{X}^L(t_0, x_0) \biggr\}
                    + F(t_0, x_0, z)
                    \leq 0.
                \end{aligned}
            \end{equation}
    \end{proposition}

    As $F$ is continuous (see $\mathbf{H}(F)$(i)) and the sets $\mathcal{X}^L(t_0, x_0)$ are compact,
    the proof almost completely repeats that of \cite[Proposition 3.8]{Bayraktar_Keller_2018} (cf. also \cite[Theorem 8.1]{Lukoyanov_2003a} for a very similar result in the finite-dimensional case).
    For the reader's convenience, we provide the proof in \ref{S:proof_P:Minimax:Infinitesimal}.

    Note that there are some similarities of the values used in \eqref{subsolution_infinitesimal} and \eqref{supersolution_infinitesimal} to certain directional derivatives such as \cite[Definition 5.1]{Keller_2024} and \cite[equation (4.32)]{Gomoyunov_Plaksin_2023}.

\subsection{Viscosity solutions}

    We proceed with the definition of viscosity solutions.

    For every $z \in V$, $L \geq 0$, $t_0 \in [0, T)$, and $x_0 \in C([0, T], H)$, define a function $\tilde{\varphi}^{z, L, t_0, x_0} \colon [t_0, T] \times \mathcal{X}^L(t_0,x_0) \to \mathbb{R}$ by
    \begin{equation} \label{E:tilde:varphi}
        \tilde{\varphi}^{z, L, t_0, x_0}(t, x)
        \coloneq \int_{t_0}^t \langle A(s, x(s)), z \rangle \, ds,
        \quad t \in [t_0, T], \, x \in \mathcal{X}^L(t_0, x_0).
    \end{equation}
    In what follows, if $L$, $t_0$, and $x_0$ are fixed, we use the shorthand notation $\tilde{\varphi}^z$ instead of $\tilde{\varphi}^{z, L, t_0, x_0}$.

    Given a number $L \geq 0$, a function $u \colon [0, T] \times C([0, T], H) \to \mathbb{R}$, and a point $(t_0, x_0) \in [0, T) \times C([0,T],H)$, put
    \begin{align*}
        \overline{\mathcal{A}}^L u(t_0, x_0)
        \coloneq \Bigl\{
        & (\varphi, z) \in \mathcal{C}_V^{1, 1} ([t_0, T] \times C([0, T], H)) \times V \colon
        \exists T_0 \in (t_0, T] \\
        & \quad 0 = (\varphi + \tilde{\varphi}^z - u)(t_0, x_0)
        = \sup_{(t, x) \in [t_0, T_0] \times \mathcal{X}^L (t_0, x_0)}
        (\varphi + \tilde{\varphi}^z - u)(t, x) \Bigr\}, \\
        \underline{\mathcal{A}}^L u(t_0, x_0)
        \coloneq \Bigl\{
        & (\varphi, z) \in \mathcal{C}_V^{1, 1} ([t_0, T] \times C([0, T], H)) \times V \colon \exists T_0 \in (t_0, T] \\
        & \quad 0 = (\varphi + \tilde{\varphi}^z - u)(t_0, x_0)
        = \inf_{(t, x) \in [t_0, T_0] \times \mathcal{X}^L(t_0, x_0)}
        (\varphi + \tilde{\varphi}^z - u)(t, x) \Bigr\}.
    \end{align*}
    Note that, given $(\varphi, z) \in \overline{\mathcal{A}}^L\,u(t_0, x_0)$ (or $(\varphi, z) \in \underline{\mathcal{A}}^L\,u(t_0,x_0)$), the function $(t, x) \mapsto (\varphi + \tilde{\varphi}^z)(t, x)$,
    $[t_0, T] \times \mathcal{X}^L(t_0,x_0) \to \mathbb{R}$, serves as a test function with $\varphi$ being its smooth and $\tilde{\varphi}^z$ being its less smooth part.

    \begin{definition} \label{definition_viscosity}
        Let $L \geq 0$ and $u \colon [0, T] \times C([0, T], H) \to \mathbb{R}$.

        (i)
            We call $u$ a viscosity $L$-subsolution of \eqref{TVP} if
            the restrictions of $u$ to the sets $[t^\ast, T] \times \mathcal{X}^{L^\ast}(t^\ast, x^\ast)$, $t^\ast \in [0, T)$, $x^\ast \in C([0, T], H)$, $L^\ast \geq 0$, are
            upper semi-continuous,
            $u$ satisfies the inequality $u(T, x) \leq h(x)$ for all $x \in C([0, T], H)$,
            and, for every $(t_0, x_0) \in [0, T) \times C([0,T],H)$ and every $(\varphi, z) \in \underline{\mathcal{A}}^L u(t_0, x_0)$,
            \begin{equation} \label{definition_viscosity_subsolution}
                \begin{aligned}
                    & \partial_t \varphi(t_0, x_0)
                    + \limsup_{\delta \to 0^+} \sup_{x \in \mathcal{X}^L (t_0, x_0)}
                    \frac{1}{\delta} \int_{t_0}^{t_0 + \delta}
                    \bigl( \partial_t \tilde{\varphi}^z (s, x)
                    - \langle A(s, x(s)), \partial_x \varphi(s, x) \rangle \bigr) \, ds \\
                    & \quad + F(t_0, x_0, \partial_x \varphi(t_0, x_0))
                    \geq 0.
                \end{aligned}
            \end{equation}

        (ii)
            We call $u$ a viscosity $L$-supersolution of \eqref{TVP} if
            the restrictions of $u$ to the sets $[t^\ast, T] \times \mathcal{X}^{L^\ast}(t^\ast, x^\ast)$, $t^\ast \in [0, T)$, $x^\ast \in C([0, T], H)$, $L^\ast \ge 0$, are lower semi-continuous,
            $u$ satisfies the inequality $u(T, x) \geq h(x)$ for all $x \in C([0, T], H)$,
            and, for every $(t_0, x_0) \in [0, T) \times C([0, T], H)$ and every $(\varphi, z) \in \overline{\mathcal{A}}^L u(t_0, x_0)$,
            \begin{equation} \label{definition_viscosity_supersolution}
                \begin{aligned}
                    & \partial_t \varphi(t_0, x_0)
                    + \liminf_{\delta \to 0^+} \inf_{x \in \mathcal{X}^L(t_0, x_0)}
                    \frac{1}{\delta} \int_{t_0}^{t_0 + \delta}
                    \bigl( \partial_t \tilde{\varphi}^z (s, x)
                    - \langle A(s, x(s)), \partial_x \varphi(s, x) \rangle \bigr) \, ds \\
                    & \quad + F(t_0, x_0, \partial_x \varphi(t_0, x_0))
                    \leq 0.
                \end{aligned}
            \end{equation}

        (iii)
            We call $u$ a viscosity $L$-solution of \eqref{TVP} if $u$ is a viscosity $L$-supersolution as well as a viscosity $L$-subsolution of \eqref{TVP}.

        (iv)
            We call $u$ a viscosity subsolution (resp. supersolution, resp. solution) of \eqref{TVP} if $u$ is a viscosity $L$-subsolution (resp. $L$-supersolution, resp. $L$-solution) of \eqref{TVP} for some $L \geq 0$.
    \end{definition}

    In \eqref{definition_viscosity_subsolution} and \eqref{definition_viscosity_supersolution}, with slight abuse of notation, we write $\partial_t \tilde{\varphi}^z (s, x) \coloneq \langle A(s, x(s)), z \rangle$.

    Note also that there is some similarity of the given definition to \cite[Definition 5.1]{Zhou_Touzi_Zhang_2024}.

    The following result shows that the notions of minimax and viscosity $L$-supersolutions and $L$-subsolutions are equivalent (see \cite[Theorem 2]{Lukoyanov_2006} for a counterpart of this result in the finite-dimensional case and also \cite[Theorems 6.9 and 6.10]{Keller_2024} as well as \cite[Theorems 4.16 and 4.17]{Jiang_Keller_2025} for similar results concerning relationships between viscosity solutions and so-called quasi-contingent solutions).

    \begin{theorem} \label{theorem_equivalence}
        Let $\mathbf{H}(A)$, $\mathbf{H}(F)${\rm(i)}, and $\mathbf{H}(h)$ be satisfied.
        Fix a number $L \geq 0$ and a function $u \colon [0, T] \times C([0,T],H) \to \mathbb{R}$.
        Then, the following statements hold:

        {\rm (i)}
            the function $u$ is a minimax $L$-supersolution of \eqref{TVP} if and only if $u$ is a viscosity $L$-supersolution of \eqref{TVP};

        {\rm (ii)}
            the function $u$ is a minimax $L$-subsolution of \eqref{TVP} if and only if $u$ is a viscosity $L$-subsolution of \eqref{TVP}.
    \end{theorem}
    \begin{proof}
        We prove only part (i), since the proof for (ii) is similar.

        First, we let $u$ be a minimax $L$-supersolution of \eqref{TVP} and show that $u$ is a viscosity $L$-supersolution of \eqref{TVP}.
        To this end, fix $(t_0, x_0) \in [0, T) \times C([0, T], H)$ and $(\varphi, z) \in \overline{\mathcal{A}}^L u(t_0, x_0)$.
        The latter means that there is a number $T_0 \in (t_0, T]$ such that, for every $(t, x) \in [t_0, T_0] \times \mathcal{X}^L(t_0, x_0)$,
        \begin{equation*}
            u(t, x) - u(t_0, x_0)
            \geq (\varphi + \tilde{\varphi}^z)(t, x) - (\varphi + \tilde{\varphi}^z)(t_0, x_0).
        \end{equation*}
        Put $z_0 \coloneq \partial_x \varphi(t_0, x_0)$ and note that $z_0 \in H$.
        As $u$ is a minimax $L$-supersolution, there is an $x_1 \in \mathcal{X}^L(t_0, x_0)$, such
        that, for every $t \in [t_0, T]$,
        \begin{equation*}
            u(t, x_1) - u(t_0, x_0)
            \leq \int_{t_0}^t \bigl( (f^{x_1}(s), z_0) - F(s, x_1, z_0) \bigr) \, ds.
        \end{equation*}
        For every $t \in [t_0, T_0]$, since
        \begin{equation*}
            (\varphi + \tilde{\varphi}^z)(t, x_1)
            - (\varphi + \tilde{\varphi}^z)(t_0, x_0)
            \leq \int_{t_0}^t \bigl( (f^{x_1}(s),z_0) - F(s,x_1,z_0) \bigr) \, ds,
        \end{equation*}
        then, using the functional chain rule and recalling the definitions of $\tilde{\varphi}^z$ and $f^{x_1}$, we obtain
        \begin{align*}
            & \int_{t_0}^t \bigl( \partial_t \varphi(s, x_1)
            + (\partial_x \varphi(s, x_1), f^{x_1}(s))
            - \langle A(s, x_1(s), \partial_x \varphi(s, x_1) \rangle
            + \langle A(s, x_1(s)), z \rangle \bigr) \, ds \\
            & \quad \leq \int_{t_0}^t \bigl( (f^{x_1}(s), z_0) - F(s, x_1, z_0) \bigr) \, ds,
        \end{align*}
        which is equivalent to
        \begin{align*}
            & \int_{t_0}^t \bigl( \partial_t \varphi(s, x_1)
            + \partial_t \tilde{\varphi}^z (s, x_1)
            - \langle A(s, x_1(s)), \partial_x \varphi(s, x_1) \rangle
            + F(s, x_1, \partial_x \varphi(s,x_1)) \bigr) \, ds \\
            & \quad \leq \int_{t_0}^t \bigl( (f^{x_1}(s), z_0 - \partial_x \varphi(s, x_1))
            + F(s, x_1, \partial_x \varphi(s, x_1)) - F(s, x_1, z_0) \bigr) \, ds.
        \end{align*}
        Consequently, thanks to continuity of $F$, $\partial_t \varphi$, and $\partial_x \varphi$, and the choice of $z_0$,
        \begin{align*}
            & \partial_t \varphi(t_0, x_0)
            + \liminf_{\delta \to 0^+} \inf_{x \in \mathcal{X}^L (t_0,x_0)}
            \frac{1}{\delta} \int_{t_0}^{t_0 + \delta}
            \bigl( \partial_t \tilde{\varphi}^z(s, x)
            - \langle A(s, x(s)), \partial_x \varphi(s, x) \rangle \bigr) \, ds \\
            & \qquad + F(t_0, x_0, \partial_x \varphi(t_0, x_0))\\
            & \quad \leq \partial_t \varphi(t_0, x_0)
            + \liminf_{\delta \to 0^+}  \frac{1}{\delta} \int_{t_0}^{t_0 + \delta}
            \bigl( \partial_t \tilde{\varphi}^z(s, x_1)
            - \langle A(s, x_1(s)), \partial_x \varphi(s, x_1) \rangle \bigr) \, ds \\
            & \qquad + F(t_0, x_0, \partial_x \varphi(t_0, x_0)) \\
            & \quad \leq \liminf_{\delta \to 0^+} \frac{1}{\delta} \int_{t_0}^{t_0 + \delta}
            \bigl( \partial_t \varphi(s, x_1) + \partial_t \tilde{\varphi}^z(s, x_1)
            - \langle A(s, x_1(s)), \partial_x \varphi(s, x_1) \rangle \\
            & \qquad + F(s, x_1, \partial_x \varphi(s, x_1)) \bigr) \, ds \\
            & \quad \leq \lim_{\delta \to 0^+} \frac{1}{\delta} \int_{t_0}^{t_0 + \delta}
            \bigl( (f^{x_1}(s), z_0 - \partial_x \varphi(s, x_1))
            + F(s, x_1, \partial_x \varphi(s, x_1)) - F(s, x_1, z_0) \bigr) \, ds \\
            & \quad = 0.
        \end{align*}
        As a result, we conclude that $u$ is a viscosity $L$-supersolution of \eqref{TVP}.

        Now, we let $u$ be a viscosity $L$-supersolution of \eqref{TVP}  and show that
         $u$ is a minimax $L$-supersolution of \eqref{TVP}.
        For the sake of a contradiction, taking Proposition \ref{P:Minimax:Infinitesimal} into account, assume that there is a triple $(t_0, x_0, z_0) \in [0, T) \times C([0,T],H) \times H$ and there are numbers $T_0 \in (t_0, T)$ and $c > 0$ such that,
        for all $(t, x) \in (t_0, T_0] \times \mathcal{X}^L(t_0, x_0)$,
        \begin{equation*}
            u(t, x) - u(t_0, x_0)
            + \int_{t_0}^{t} (- f^x(s), z_0) \, ds
            + F(t_0, x_0, z_0) (t - t_0)
            > 2 c (t - t_0).
        \end{equation*}
        Since $V$ is dense in $H$ and $F$ is continuous, there exists a $z \in V$ such that, for all $(t, x) \in (t_0, T_0] \times \mathcal{X}^L(t_0, x_0)$,
        \begin{equation*}
            u(t, x) - u(t_0, x_0)
            + \int_{t_0}^{t} (- f^x(s), z) \, ds
            + F(t_0, x_0, z) (t - t_0)
            > c (t - t_0)
        \end{equation*}
        and, thus,
        \begin{align*}
            u(t, x) - u(t_0, x_0)
            & > (c - F(t_0, x_0, z)) (t - t_0) + \int_{t_0}^t (f^x (s), z) \, ds \\
            & = (c - F(t_0, x_0, z)) (t - t_0)
            + \int_{t_0}^t \bigl( (x^\prime(s), z) + \langle A(s, x(s)), z \rangle \bigr) \, ds \\
            & = (\varphi + \tilde{\varphi}^z)(t, x) - (\varphi + \tilde{\varphi}^z)(t_0, x_0)
        \end{align*}
        with the function $\varphi \colon [t_0, T] \times C([0, T], H) \to \mathbb{R}$ defined by
        \begin{equation*}
            \varphi(t, x)
            \coloneq u(t_0, x_0) + (t - t_0) (c - F(t_0, x_0, z)) + (x(t) - x_0(t_0), z)
        \end{equation*}
        for all $(t, x) \in [t_0, T] \times C([0, T], H)$ and the function $\tilde{\varphi}^z$ defined in \eqref{E:tilde:varphi}.
        Consequently, we have $(\varphi, z) \in \overline{\mathcal{A}}^L u(t_0, x_0)$ with $\partial_t \varphi(t, x) = c - F(t_0, x_0, z)$ and $\partial_x \varphi(t, x) = z$ for all $(t, x) \in [t_0, T) \times C([0, T], H)$.
        However,
        \begin{align*}
            & \partial_t \varphi(t_0, x_0)
            + \liminf_{\delta \to 0^+} \inf_{x \in \mathcal{X}^L(t_0, x_0)}
            \frac{1}{\delta} \int_{t_0}^{t_0 + \delta}
            \bigl( \partial_t \tilde{\varphi}^z (s, x)
            - \langle A(s, x(s)), \partial_x \varphi(s, x) \rangle \bigr) \, ds \\*
            & \qquad + F(t_0, x_0, \partial_x \varphi(t_0, x_0)) \\
            & \quad = c - F(t_0, x_0, z) + F(t_0, x_0, \partial_x \varphi(t_0, x_0))
            = c
            > 0,
        \end{align*}
        which is a contradiction to $u$ being a viscosity $L$-supersolution of \eqref{TVP}.
    \end{proof}

    From Theorems \ref{theorem_minimax_existence_uniqueness} and \ref{theorem_equivalence}, we derive
    \begin{theorem} \label{theorem_viscosity_existence_uniqueness}
        Let $\mathbf{H}(A)$, $\mathbf{H}(F)$, and $\mathbf{H}(h)$ be satisfied.
        Then, there exists a unique viscosity solution $u$ of \eqref{TVP}.
        Moreover, $u$ is a viscosity $L_0$-solution of \eqref{TVP}, where $L_0 \geq 0$ is the number from $\mathbf{H}(F)${\rm(ii)}.
    \end{theorem}

\section{Functional chain rule and co-invariant derivatives}
\label{S:Functional_chain_rule}

    The proof of Theorem \ref{theorem_minimax_existence_uniqueness} is preceded by a discussion of the functional chain rule and co-invariant derivatives.
    As an application, we establish the required smoothness properties of a function used in the proof of Theorem \ref{theorem_minimax_existence_uniqueness}.
    Note that the technique developed in this section is also of independent interest.

    Let $t_\ast \in [0, T)$.
    In general, it is rather difficult to prove that a function $\varphi \colon [t_\ast, T] \times C([0, T], H) \to \mathbb{R}$ belongs to the set $\mathcal{C}_V^{1, 1}([t_\ast, T] \times C([0, T], H))$.
    In \cite[Example 2.18]{Bayraktar_Keller_2018}, this fact is verified for some specific functions $\varphi$ using the integration-by-parts formula \cite[Proposition 2.4]{Bayraktar_Keller_2018}.
    This section describes another approach based on the development of the technique of so-called co-invariant derivatives (see, e.g., \cite{Kim_1999,Lukoyanov_2003a} and also \cite{Gomoyunov_Lukoyanov_Plaksin_2021,Gomoyunov_Lukoyanov_RMS_2024} in the finite-dimensional setting).

    \begin{definition} \label{definition_ci_smooth_A3}
        Let $t_\ast \in [0, T)$ and $\varphi \colon [t_\ast, T] \times C([0, T], H) \to \mathbb{R}$.

        (i)
            Let $(t_0, x_0) \in [t_\ast, T) \times C([0, T], H)$.
            We call $\varphi$ co-invariantly differentiable ($ci$-dif\-fer\-entiable for short) at $(t_0, x_0)$ if there exist $\partial_t^{\, ci} \varphi(t_0, x_0) \in \mathbb{R}$ and $\partial_x^{\, ci} \varphi(t_0, x_0) \in H$ such that, for every $x \in C([0, T], H)$ with $x(t) = x_0(t)$ for all $t \in [0, t_0]$ and $x|_{[t_0, T]} \in C^1([t_0, T], V)$,
            \begin{equation} \label{ci_derivatives_definition_A3}
                \varphi(t, x) - \varphi(t_0, x_0)
                = \partial_t^{\, ci} \varphi(t_0, x_0) (t - t_0)
                + (\partial_x^{\, ci} \varphi(t_0, x_0), x(t) - x_0(t_0)) + o(t - t_0),
            \end{equation}
            where $t \in (t_0, T]$ and the function $\delta \mapsto o(\delta)$, $(0, T - t_0] \to \mathbb{R}$, may depend on $x$ and $o(\delta) / \delta \to 0$ as $\delta \to 0^+$.
            In this case, the values $\partial_t^{\, ci} \varphi(t_0, x_0)$ and $\partial_x^{\, ci} \varphi(t_0, x_0)$ are called co-invariant derivatives ($ci$-derivatives for short) of $\varphi$ at $(t_0, x_0)$.

        (ii)
            We call $\varphi$ co-invariantly smooth ($ci$-smooth for short) if $\varphi$ is continuous, $ci$-differ\-entiable at every point $(t, x) \in [t_\ast, T) \times C([0, T], H)$, and its $ci$-derivatives $\partial_t^{\, ci} \varphi \colon [t_\ast, T) \times C([0, T], H) \to \mathbb{R}$ and $\partial_x^{\, ci} \varphi \colon [t_\ast, T) \times C([0, T], H) \to H$ are continuous.
    \end{definition}

    \begin{remark} \label{remark_ci-derivatives_finite-dimensional}
        In the finite-dimensional setting \cite{Lukoyanov_2003a,Gomoyunov_Lukoyanov_Plaksin_2021,Gomoyunov_Lukoyanov_RMS_2024}, in the definition of $ci$-dif\-fer\-entiability, the class of admissible extensions of $\{x_0(t)\}_{t \in [0, t_0]}$ usually consists of all functions $x$ such that the restriction $x|_{[t_0, T]}$ is Lipschitz continuous.
        However, the definition given above (Definition \ref{definition_ci_smooth_A3}, (i)) assumes that $x|_{[t_0, T]} \in C^1([t_0, T], V)$.
        Note that this assumption implies Lipschitz continuity.
        Indeed, since $V$ is continuously embedded into $H$, there is a constant $c_{V, H} > 0$ such that
        \begin{equation} \label{c_VH}
            |v|
            \leq c_{V, H} \|v\|,
            \quad v \in V.
        \end{equation}
        Then, for any $t$, $s \in [t_0, T]$, we derive
        \begin{equation*}
            |x(t) - x(s)|
            \leq c_{V, H} \|x(t) - x(s)\|
            \leq c_{V, H} \max_{r \in [t_0, T]} \|\dot{x}(r)\|  |t - s|.
        \end{equation*}
        Thus, the class of admissible extensions from Definition \ref{definition_ci_smooth_A3}(i) is narrower than the usually used one.
        In particular, this allows us to transfer the results on $ci$-differentiability of functions from the finite-dimensional case to the infinite-dimensional case under consideration.
    \end{remark}

    \begin{remark} \label{remark_ci-derivatives_equivalent}
        Since the class of admissible extensions of $\{x_0(t)\}_{t \in [0, t_0]}$ in Definition \ref{definition_ci_smooth_A3}(i) consists of functions $x$ such that $x|_{[t_0, T]} \colon [t_0,T] \to V$ is continuously differentiable with $V$-valued derivatives $\dot{x}(t)$, relation \eqref{ci_derivatives_definition_A3} can be equivalently reformulated as follows:
        \begin{align*}
            \frac{d}{dt^+} \varphi(t_0, x_0)
            & \coloneq \lim_{t \to t_0^+} \frac{\varphi(t, x) - \varphi(t_0, x_0)}{t - t_0} \\
            & = \lim_{t \to t_0^+} \frac{\partial_t^{\, ci} \varphi(t_0, x_0) (t - t_0)
            + (\partial_x^{\, ci} \varphi(t_0, x_0), x(t) - x_0(t_0)) + o(t - t_0)}{t - t_0} \\
            & = \partial_t^{\, ci} \varphi(t_0, x_0)
            + (\partial_x^{\, ci} \varphi(t_0, x_0), \dot{x}(t_0)).
        \end{align*}
    \end{remark}

    \begin{remark} \label{remark_ci-derivatives_path-derivatives}
        Let $t_\ast \in [0, T)$, $\varphi \in \mathcal{C}_V^{1, 1}([t_\ast, T] \times C([0, T], H))$, and $(t_0, x_0) \in [t_\ast, T) \times C([0, T], H)$.
        Then, $\varphi$ is $ci$-smooth and its path-derivatives coincide with its $ci$-derivatives on $[t_\ast, T) \times C([0, T], H)$.
        Indeed, note that the only thing we need to verify is the coincidence of the derivatives at every point.
        To this end, consider $x \in C([0, T], H)$ with $x(t) = x_0(t)$ for all $t \in [0, t_0]$ and $x|_{[t_0, T]} \in C^1([t_0, T], V)$.
        Then, we have $x|_{(t_0, T)} \in W_{pq}(t_0, T)$ and, therefore,
        \begin{equation*}
            \varphi(t, x) - \varphi(t_0, x_0)
            = \int_{t_0}^{t} \bigl( \partial_t \varphi(s, x)
            + \langle x^\prime(s), \partial_x \varphi(s, x) \rangle \bigr) \, ds,
            \quad t \in [t_0, T].
        \end{equation*}
        Note that $x^\prime(s) = \dot{x}(s)$ for a.e. $s \in [t_0, T]$.
        Hence,
        \begin{equation*}
            \varphi(t, x) - \varphi(t_0, x_0)
            = \int_{t_0}^{t} \bigl( \partial_t \varphi(s, x)
            + (\dot{x}(s), \partial_x \varphi(s, x)) \bigr) \, ds,
            \quad t \in [t_0, T].
        \end{equation*}
        The function under the integral sign is continuous as a function of $s$.
        Then, we can differentiate with respect to $t$ to obtain
        \begin{equation*}
            \frac{d}{dt^+} \varphi(t_0, x_0)
            = \partial_t \varphi(t_0, x_0)
            + (\partial_x \varphi(t_0, x_0), \dot{x}(t_0)).
        \end{equation*}
        So, $\varphi$ is $ci$-differentiable at the point $(t_0, x_0)$ with the $ci$-derivatives $\partial_t^{\, ci} \varphi(t_0, x_0) = \partial_t \varphi(t_0, x_0)$, $\partial_x^{\, ci} \varphi(t_0, x_0) = \partial_x \varphi(t_0, x_0)$.
        Moreover, we immediately obtain that, under the considered assumptions, the function $\varphi$ is $ci$-smooth.
    \end{remark}

    The next auxiliary result shows that, for a $ci$-smooth function with $ci$-derivatives satisfying the additional boundedness property, the functional chain rule holds provided that $x|_{[t_0, T]} \in C^1([t_0, T], V)$ (for a proof of a close result in the finite-dimensional case, see, e.g., \cite[Proposition 1]{Gomoyunov_Lukoyanov_RMS_2024}).
    \begin{proposition} \label{proposition_ci-smooth_A3}
        Let $t_\ast \in [0, T)$ and let $\varphi \colon [t_\ast, T] \times C([0, T], H) \to \mathbb{R}$ be a $ci$-smooth function.
        Fix $(t_0, x_0) \in [t_\ast, T) \times C([0, T], H)$ and consider $x \in C([0, T], H)$ with $x(t) = x_0(t)$ for all $t \in [0, t_0]$ and $x|_{[t_0, T]} \in C^1([t_0, T], V)$.
        Suppose that there exists $M \geq 0$ such that $|\partial_t^{\, ci} \varphi(t, x)| \leq M$ and $|\partial_x^{\, ci} \varphi(t, x)| \leq M$ for all $t \in [t_0, T)$.
        Then, for every $t \in [t_0, T]$, the functional chain rule holds:
        \begin{equation} \label{proposition_ci-smooth_A3_main}
            \varphi(t, x) - \varphi(t_0, x_0)
            = \int_{t_0}^{t} \bigl( \partial_t^{\, ci} \varphi(s, x)
            + (\dot{x}(s), \partial_x^{\, ci} \varphi(s, x)) \bigr) \, ds.
        \end{equation}
    \end{proposition}
    \begin{proof}
        Fix $\vartheta \in [t_0, T)$ and consider the function $\mu(t) \coloneq \varphi(t, x)$, $t \in [t_0, \vartheta]$.
        The function $\mu$ is continuous by continuity of $\varphi$.
        For every $t \in [t_0, \vartheta]$, since $\varphi$ is $ci$-differentiable at $(t, x)$, we have (see Remark \ref{remark_ci-derivatives_equivalent})
        \begin{equation*}
            \dot{\mu}^+(t)
            = \partial_t^{\, ci} \varphi(t, x) + (\partial_x^{\, ci} \varphi(t, x), \dot{x}(t)),
        \end{equation*}
        where $\dot{\mu}^+(t)$ is the right-hand derivative of $\mu$ at $t$.
        Note that the function $\dot{\mu}^+$ is continuous on $[t_0, \vartheta]$ owing to continuity of $\partial_t^{\, ci} \varphi$ and $\partial_x^{\, ci} \varphi$.
        Hence, it follows from Dini's theorem \cite[Chapter 4, Theorem 1.3]{Bruckner_1978} that $\mu$ is continuously differentiable.
        As a result, for every $t \in [t_0, \vartheta]$,
        \begin{equation*}
            \mu(t) - \mu(t_0)
            = \int_{t_0}^{t} \dot{\mu}(s) \, ds
            = \int_{t_0}^{t} \dot{\mu}^+(s) \, ds
            = \int_{t_0}^{t} \bigl( \partial_t^{\, ci} \varphi(s, x) + (\partial_x^{\, ci} \varphi(s, x), \dot{x}(s)) \bigr) \, ds.
        \end{equation*}
        Thus, formula \eqref{proposition_ci-smooth_A3_main} holds for all $t \in [t_0, T)$ and it remains to verify it for $t = T$.
        Taking the constant $c_{V, H}$ from \eqref{c_VH} and recalling the choice of $M$, we derive
        \begin{align*}
            & |\partial_t^{\, ci} \varphi(t, x) + (\partial_x^{\, ci} \varphi(t, x), \dot{x}(t))|
            \leq |\partial_t^{\, ci} \varphi(t, x)| + |\partial_x^{\, ci} \varphi(t, x)| |\dot{x}(t)| \\
            & \quad \leq M + M c_{V, H} \|\dot{x}(t)\|
            \leq M + M c_{V, H} \max_{s \in [t_0, T]} \|\dot{x}(s)\|
        \end{align*}
        for all $t \in [t_0, T)$.
        Therefore,
        \begin{align*}
            & \biggl| \int_{t_0}^{T - 1 / n} \bigl( \partial_t^{\, ci} \varphi(s, x) + (\partial_x^{\, ci} \varphi(s, x), \dot{x}(s)) \bigr) \, ds
            - \int_{t_0}^{T} \bigl( \partial_t^{\, ci} \varphi(s, x) + (\partial_x^{\, ci} \varphi(s, x), \dot{x}(s)) \bigr) \, ds \biggr| \\
            & \quad = \biggl| \int_{T - 1 / n}^T \bigl( \partial_t^{\, ci} \varphi(s, x) + (\partial_x^{\, ci} \varphi(s, x), \dot{x}(s)) \bigr) \, ds \biggr| \\
            & \quad \leq \int_{T - 1 / n}^T \bigl| \partial_t^{\, ci} \varphi(s, x) + (\partial_x^{\, ci} \varphi(s, x), \dot{x}(s)) \bigr| \, ds \\
            & \quad \leq \bigl( M + M c_{V, H} \max_{s \in [t_0, T]} \|\dot{x}(s)\| \bigr) / n \to 0
        \end{align*}
        as $n \to \infty$.
        Consequently, using continuity of $\varphi$ and the fact that formula \eqref{proposition_ci-smooth_A3_main} holds for $t = T - 1 / n$, we derive
        \begin{align*}
            \varphi(T, x) - \varphi(t_0, x_0)
            & = \lim_{n \to \infty} (\varphi(T - 1 / n, x) - \varphi(t_0, x_0)) \\
            & = \lim_{n \to \infty} \int_{t_0}^{T - 1 / n} \bigl( \partial_t^{\, ci} \varphi(s, x) + (\partial_x^{\, ci} \varphi(s, x), \dot{x}(s)) \bigr) \, ds \\
            & = \int_{t_0}^{T} \bigl( \partial_t^{\, ci} \varphi(s, x) + (\partial_x^{\, ci} \varphi(s, x), \dot{x}(s)) \bigr) \, ds.
        \end{align*}
        The proof is complete.
    \end{proof}

    The next proposition is the main result of this section.
    It provides a sufficient condition that a function $\varphi$ belongs to the set $\mathcal{C}^{1,1}_V([0, T] \times C([0, T], H))$ and, in particular, a way of calculating its path derivatives (see Remark \ref{remark_ci-derivatives_path-derivatives}).
    The proof is by approximation argument (see, e.g., \cite[Proposition 2.15]{Bayraktar_Keller_2018} and \cite[Theorem 5.7]{Nisio_2015} for similar results in the non-path-dependent case).

    \begin{proposition} \label{proposition_ci-smooth_W_A3}
        Let $t_\ast \in [0, T)$ and let $\varphi \colon [t_\ast, T] \times C([0, T], H) \to \mathbb{R}$ be a $ci$-smooth function that satisfies the following conditions.

        {\rm (i)}
            There exist continuous functions $\partial_t \varphi \colon [t_\ast, T] \times C([0, T], H) \to \mathbb{R}$ and $\partial_x \varphi \colon [t_\ast, T] \times C([0, T], H) \to H$ such that
            \begin{equation*}
                \partial_t \varphi(t, x)
                = \partial_t^{\, ci} \varphi(t, x),
                \quad \partial_x \varphi(t, x)
                = \partial_x^{\, ci} \varphi(t, x),
                \quad (t, x) \in [t_\ast, T) \times C([0, T], H).
            \end{equation*}

        {\rm (ii)}
            For every $x \in C([0, T], H)$ with $x(t) \in V$ for a.e. $t \in [t_\ast, T]$, it holds that $\partial_x \varphi(t, x) \in V$ for a.e. $t \in [t_\ast, T]$.

        {\rm (iii)}
            Let $t_0 \in [t_\ast, T)$ and $x \in C([0, T], H)$ with $x\vert_{(t_0, T)} \in W_{p q}(t_0, T)$.
            Let $(x_n)_n$ be a sequence in $C([0, T], H)$ such that $x_n\vert_{[t_0, T]} \in C^1([t_0, T], V)$ for every $n \in \mathbb{N}$ and
            \begin{equation*}
                \|x\vert_{[t_0, T]} - x_n\vert_{[t_0, T]}\|_{W_{p q}(t_0, T)}
                + \|x - x_n\|_\infty
                \to 0
                \quad \text{as } n \to \infty.
            \end{equation*}
            Then, $\| \partial_x \varphi(\cdot, x_n) - \partial_x \varphi(\cdot, x) \|_{L^p(t_0, T; V)} \to 0$ as $n \to \infty$ and there is a number $L > 0$ such that
            \begin{equation*}
                \| \partial_x \varphi(\cdot, x_n)\|_{L^p(t_0, T; V)}
                \leq L,
                \quad n \in \mathbb{N}.
            \end{equation*}

        Then, {\color{black} $\varphi \in \mathcal{C}^{1, 1}_V([t_\ast, T] \times C([0, T], H))$} with the path-derivatives from item {\rm (i)}.
    \end{proposition}
    \begin{proof}
        Note that the only fact we need to prove is the following.
        Suppose that $t_0 \in [t_\ast, T)$, $x_0 \in C([0, T], H)$, and $x \in C([0, T], H)$ with $x(t) = x_0(t)$ for all $t \in [0, t_0]$ and $x \vert_{(t_0, T)} \in W_{p q}(t_0, T)$ are given.
        Then, for every $t \in [t_0, T]$,
        \begin{equation*}
            \varphi(t, x) - \varphi(t_0, x_0)
            = \int_{t_0}^{t} \bigl( \partial_t \varphi(s, x)
            + \langle x^\prime(s), \partial_x \varphi(s, x) \rangle \bigr) \, ds.
        \end{equation*}

        Let $n \in \mathbb{N}$ be fixed.
        Since $y \coloneq x \vert_{(t_0, T)} \in W_{pq}(t_0, T)$, $C^1([t_0, T], V)$ is dense in $W_{p q}(t_0, T)$ (see, e.g., \cite[Proposition 23.23 (iii)]{ZeidlerIIA}), and $W_{p q}(t_0, T)$ is continuously embedded in $C([t_0, T], H)$ (see, e.g., \cite[Proposition 23.23 (ii)]{ZeidlerIIA}), there exists an $y_n \in C^1([t_0, T], V)$ such that
        \begin{equation} \label{E:proposition_ci-smooth_W_A3}
            \|y - y_n\|_{W_{p q}(t_0, T)}
            + \|y - y_n\|_{C([t_0, T], H)}
            \leq 1 / n.
        \end{equation}
        Now, since $|x_0(t_0) - y_n(t_0) |\leq 1 / n$, there is a function $x_n \in C([0, T], H)$ such that $\|(x_n - x_0)(\cdot \wedge t_0)\|_\infty \leq 1 / n$ and $x_n(t) = y_n(t)$ for every $t \in [t_0, T]$.

        The function $\varphi$ is $ci$-smooth and, in view of assumption (i), there exists $M_n \geq 0$ such that
        \begin{equation*}
            |\varphi_t^{\, ci}(t, x_n)|
            = |\varphi_t(t, x_n)|
            \leq M_n,
            \quad |\varphi^{\, ci}_x(t, x_n)|
            = |\varphi_x(t, x_n)|
            \leq M_n,
            \quad t \in [t_\ast, T).
        \end{equation*}
        Hence, we obtain by Proposition \ref{proposition_ci-smooth_A3} that, for every $t \in [t_0, T]$,
        \begin{equation*}
            \varphi(t, x_n) - \varphi(t_0, x_n)
            = \int_{t_0}^{t} \bigl( \partial_t \varphi(s, x_n)
            + (\dot{x}_n(s), \partial_x \varphi(s, x_n)) \bigr) \, ds.
        \end{equation*}

        Fix $t \in [t_0, T]$.
        Since the functions $\varphi$ and $\partial_t \varphi$ are continuous and $\|x_n - x\|_\infty \to 0$ as $n \to \infty$, we have
        \begin{gather*}
            |\varphi(t, x_n) - \varphi(t, x)|
            \to 0, \\
            \biggl| \int_{t_0}^{t} \bigl( \partial_t \varphi(s, x_n) - \partial_t \varphi(s, x) \bigr) \, ds \biggr|
            \leq (t - t_0) \max_{s \in [t_0, T]} |\partial_t \varphi(s, x_n) - \partial_t \varphi(s, x)|
            \to 0
        \end{gather*}
        as $n \to \infty$.
        Further, using the H\"{o}lder inequality in form of \cite[Proposition 23.6]{ZeidlerIIA}, we get
        \begin{align*}
            & \biggl| \int_{t_0}^{t}
            \bigl( (\dot{y}_n(s), \partial_x \varphi(s, x_n))
            - \langle y^\prime(s), \partial_x \varphi(s, x) \rangle\bigr) \, ds \biggr| \\
            & \quad \leq \int_{t_0}^{t} \bigl| \langle \dot{y}_n(s) - y^\prime(s), \partial_x \varphi(s, x_n) \rangle \bigr| \, ds
            + \int_{t_0}^{t} \bigl| \langle y^\prime(s), \partial_x \varphi(s, x_n) - \partial_x \varphi(s, x) \rangle \bigr| \, ds \\
            & \quad \leq \| \dot{y}_n - y^\prime \|_{L^q(t_0, T; V^\ast)}
            \| \partial_x \varphi(\cdot,x_n)\|_{L^p(t_0, T; V)} \\*
            & \qquad + \| y^\prime \|_{L^q(t_0, T; V^\ast)}
            \| \partial_x \varphi(\cdot, x_n) - \partial_x \varphi(\cdot, x) \|_{L^p(t_0, T; V)} \\
            & \quad \leq \| y_n - y \|_{W_{pq}(t_0, T)}
            \| \partial_x \varphi(\cdot,x_n)\|_{L^p(t_0, T; V)} \\*
            & \qquad + \| y \|_{W_{pq}(t_0, T)}
            \| \partial_x \varphi(\cdot, x_n) - \partial_x \varphi(\cdot, x) \|_{L^p(t_0, T; V)}
            \to 0
        \end{align*}
        as $n \to \infty$ thanks to \eqref{E:proposition_ci-smooth_W_A3} and assumption (iii).
        Then, we have
        \begin{align*}
            \varphi(t, x) - \varphi(t_0, x_0)
            & = \lim_{n \to \infty} \bigl( \varphi(t, x_n) - \varphi(t_0, x_n) \bigr) \\
            & = \lim_{n \to \infty} \int_{t_0}^{t} \bigl( \partial_t \varphi(s, x_n)
            + (\dot{x}_n(s), \partial_x \varphi(s, x_n)) \bigr) \, ds \\
            & = \int_{t_0}^{t} \bigl( \partial_t \varphi(s, x)
            + \langle x^\prime(s), \partial_x \varphi(s, x) \rangle \bigr) \, ds.
        \end{align*}
        The proof is complete.
    \end{proof}

    Let us apply Proposition \ref{proposition_ci-smooth_A3} in order to prove that a particular function belongs to $\mathcal{C}^{1, 1}_V([0, T] \times C([0, T], H))$.

    For every $(t, x) \in [0, T] \times C([0, T], H)$, denote
    \begin{equation} \label{Upsilon}
        \Upsilon(t, x)
        \coloneq
        \begin{cases}
            \displaystyle
            \frac{\bigl( \|x(\cdot \wedge t)\|_\infty^2 - |x(t)|^2 \bigr)^2}{\|x(\cdot \wedge t)\|_\infty^2} + 2 |x(t)|^2,
            & \text{if } \|x(\cdot \wedge t)\|_\infty > 0, \\
            0, & \text{if } \|x(\cdot \wedge t)\|_\infty = 0,
        \end{cases}
    \end{equation}
    and
    \begin{equation} \label{Psi_derivatives_A}
        \begin{aligned}
            \partial_t \Upsilon(t, x)
            & \coloneq 0, \\
            \partial_x \Upsilon(t, x)
            & \coloneq
            \begin{cases}
                \displaystyle
                - \frac{4 \bigl( \|x(\cdot \wedge t)\|_\infty^2 - |x(t)|^2 \bigr)}{\|x(\cdot \wedge t)\|_\infty^2} x(t) + 4 x(t)
                & \text{if } \|x(\cdot \wedge t)\|_\infty > 0, \\
                0 & \text{if } \|x(\cdot \wedge t)\|_\infty = 0.
            \end{cases}
        \end{aligned}
    \end{equation}
    Note that, in the case where $\|x(\cdot \wedge t)\|_\infty > 0$,
    \begin{equation*}
        \partial_x \Upsilon(t, x)
        = 4 \frac{\|x(\cdot \wedge t)\|_\infty^2 - \|x(\cdot \wedge t)\|_\infty^2 + |x(t)|^2}{\|x(\cdot \wedge t)\|_\infty^2} x(t)
        = 4 \frac{|x(t)|^2}{\|x(\cdot \wedge t)\|_\infty^2} x(t).
    \end{equation*}

    \begin{remark}
        The function $\Upsilon$ was introduced in \cite{Zhou_2020}, where it was also established that $\Upsilon$ is differentiable in the sense of Dupire \cite{Dupire_2019} with horizontal derivative $\partial_t \Upsilon$ and  vertical derivative $\partial_x \Upsilon$.
    \end{remark}

    \begin{proposition} \label{proposition_psi_c11V}
        The function $\Upsilon$ from \eqref{Upsilon} belongs to $\mathcal{C}^{1, 1}_V([0, T] \times C([0, T], H))$ and its path-derivatives are given by \eqref{Psi_derivatives_A}.
    \end{proposition}
    \begin{proof}
        Repeating the proof given in the finite-dimensional setting \cite[Appendix B]{Gomoyunov_Lukoyanov_Plaksin_2021} (in this connection, see Remark \ref{remark_ci-derivatives_finite-dimensional}), we obtain that the function $\Upsilon \colon [0, T] \times C([0, T], H) \to \mathbb{R}$ is $ci$-smooth, the function $\partial_x \Upsilon \colon [0, T] \times C([0, T], H) \to H$ is continuous, and $\partial_t \Upsilon (t, x) = \partial_t^{\, ci} \Upsilon (t, x)$, $\partial_x \Upsilon (t, x) = \partial_x^{\, ci} \Upsilon (t, x)$ for all $(t, x) \in [0, T) \times C([0, T], H)$.
        Hence, assumption (i) of Proposition \ref{proposition_ci-smooth_W_A3} is fulfilled.
        Owing to the explicit formula \eqref{Psi_derivatives_A} for $\partial_x \Upsilon$, assumption (ii) of Proposition \ref{proposition_ci-smooth_W_A3} is also fulfilled.

        Let us verify assumption (iii) of Proposition \ref{proposition_ci-smooth_W_A3}.
        Fix $t_0 \in [0, T)$ and $x \in C([0, T], H)$ with $x\vert_{[t_0, T]} \in W_{pq}(t_0, T)$.
        Consider a sequence $(x_n)_n$ in $C([0, T], H)$ such that $x_n\vert_{[t_0, T]} \in C^1([t_0, T], V)$ for every $n \in \mathbb{N}$ and
        \begin{equation*}
            \|x\vert_{(t_0, T)} - x_n\vert_{(t_0, T)}\|_{W_{p q}(t_0, T)}
            + \|x - x_n\|_\infty
            \to 0
            \quad \text{as } n \to \infty.
        \end{equation*}
        Since $\|x\vert_{(t_0, T)} - x_n\vert_{(t_0, T)}\|_{W_{p q}(t_0, T)} \to 0$, we have $\|x\vert_{(t_0, T)} - x_n\vert_{(t_0, T)}\|_{L^p(t_0, T; V)} \to 0$.
        Hence, there exists a number $L > 0$ such that $\|x_n\vert_{(t_0, T)}\|_{L^p(t_0, T; V)} \leq L$ for all $n \in \mathbb{N}$.
        For every $n \in \mathbb{N}$ and every $t \in [t_0, T]$, we have $\partial_x \Upsilon(t, x_n) \in V$ and
        \begin{equation*}
            \|\partial_x \Upsilon(t, x_n)\|
            \leq 4 \|x_n(t)\|.
        \end{equation*}
        Therefore,
        \begin{equation*}
            \| \partial_x \Upsilon(\cdot, x_n)\|_{L^p(t_0, T; V)}
            \leq 4 \| x_n|_{[t_0, T]}\|_{L^p(t_0, T; V)}
            \leq 4 L,
            \quad n \in \mathbb{N}.
        \end{equation*}
        So, it remains to prove that $\| \partial_x \varphi(\cdot, x_n) - \partial_x \varphi(\cdot, x) \|_{L^p(t_0, T; V)} \to 0$ as $n \to \infty$.

        First, suppose that $\|x(\cdot \wedge t)\|_\infty = 0$ for some $t \in [t_0, T]$.
        Then, we can take $\vartheta \in [t_0, T]$ such that $\|x(\cdot \wedge \vartheta)\|_\infty = 0$ and $\|x(\cdot \wedge t)\|_\infty > 0$ for all $t \in (\vartheta, T]$.
        Since
        \begin{equation*}
            \|\partial_x \Upsilon(t, x_n) - \partial_x \Upsilon(t, x) \|
            = \|\partial_x \Upsilon(t, x_n)\|
            \leq 4 \|x_n(t)\|
            \leq 4 \|x_n(t) - x(t)\|,
            \quad t \in [t_0, \vartheta],
        \end{equation*}
        we obtain
        \begin{equation*}
            \| \partial_x \varphi(\cdot, x_n) - \partial_x \varphi(\cdot, x) \|_{L^p(t_0, \vartheta; V)}
            \leq 4 \|x_n|_{(t_0, \vartheta)} - x|_{(t_0, \vartheta)} \|_{L^p(t_0, \vartheta; V)}
            \to 0
            \quad \text{as } n \to \infty.
        \end{equation*}
        Consequently, if $\vartheta = T$, we obtain the required convergence.
        Consider the case where $\vartheta < T$.
        For every $t \in (\vartheta, T]$ such that $x(t) \in V$ and every $n \in \mathbb{N}$, we derive
        \begin{align*}
            \|\partial_x \Upsilon(t, x_n) - \partial_x \Upsilon(t, x) \|
            & = \biggl\| 4 \frac{|x_n(t)|^2}{\|x_n(\cdot \wedge t)\|_\infty^2} x_n(t)
            - 4 \frac{|x(t)|^2}{\|x(\cdot \wedge t)\|_\infty^2} x(t) \biggr\| \\
            & = \biggl\| 4 \frac{|x_n(t)|^2}{\|x_n(\cdot \wedge t)\|_\infty^2} x_n(t)
            - 4 \frac{|x_n(t)|^2}{\|x_n(\cdot \wedge t)\|_\infty^2} x(t) \\
            & \quad + 4 \frac{|x_n(t)|^2}{\|x_n(\cdot \wedge t)\|_\infty^2} x(t)
            - 4 \frac{|x(t)|^2}{\|x(\cdot \wedge t)\|_\infty^2} x(t) \biggr\| \\
            & \leq 4 \frac{|x_n(t)|^2}{\|x_n(\cdot \wedge t)\|_\infty^2} \|x_n(t) - x(t)\| \\
            & \quad + 4 \biggl| \frac{|x_n(t)|^2}{\|x_n(\cdot \wedge t)\|_\infty^2}
            - \frac{|x(t)|^2}{\|x(\cdot \wedge t)\|_\infty^2} \biggr| \|x(t)\| \\
            & \leq 4 \|x_n(t) - x(t)\|
            + 4 \biggl| \frac{|x_n(t)|^2}{\|x_n(\cdot \wedge t)\|_\infty^2}
            - \frac{|x(t)|^2}{\|x(\cdot \wedge t)\|_\infty^2} \biggr| \|x(t)\|
        \end{align*}
        if $\|x_n(\cdot \wedge t)\|_\infty > 0$ and, if $\|x_n(\cdot \wedge t)\|_\infty = 0$,
        \begin{equation*}
            \|\partial_x \Upsilon(t, x_n) - \partial_x \Upsilon(t, x) \|
            = \|\partial_x \Upsilon(t, x)\|
            \leq 4 \|x(t)\|
            \leq 4 \|x_n(t) - x(t)\|.
        \end{equation*}
        Fix $\varepsilon > 0$ and, by the absolute continuity of the Lebesgue integral, choose $\vartheta_\ast \in (\vartheta, T)$ such that $\|x|_{(\vartheta, \vartheta_\ast)}\|_{L^p(\vartheta, \vartheta_\ast; V)} \leq \varepsilon$.
        Then, for every $t \in (\vartheta, \vartheta_\ast]$ such that $x(t) \in V$ and every $n \in \mathbb{N}$, we derive
        \begin{equation*}
            \|\partial_x \Upsilon(t, x_n) - \partial_x \Upsilon(t, x) \|
            \leq 4 \|x_n(t) - x(t)\| + 8 \|x(t)\|,
        \end{equation*}
        which yields
        \begin{align*}
            & \|\partial_x \Upsilon(\cdot, x_n) - \partial_x \Upsilon(\cdot, x) \|_{L^p(\vartheta, \vartheta_\ast; V)} \\
            & \quad \leq 4 \|x_n|_{(\vartheta, \vartheta_\ast)} - x|_{(\vartheta, \vartheta_\ast)}\|_{L^p(\vartheta, \vartheta_\ast; V)}
            + 8 \|x|_{(\vartheta, \vartheta_\ast)}\|_{L^p(\vartheta, \vartheta_\ast; V)} \\
            & \quad \leq 4 \|x_n|_{(\vartheta, \vartheta_\ast)} - x|_{(\vartheta, \vartheta_\ast)}\|_{L^p(\vartheta, \vartheta_\ast; V)}
            + 8 \varepsilon.
        \end{align*}
        Hence, there exists $N_1 \in \mathbb{N}$ such that, for every $n \geq N_1$,
        \begin{equation*}
            \|\partial_x \Upsilon(\cdot, x_n) - \partial_x \Upsilon(\cdot, x) \|_{L^p(\vartheta, \vartheta_\ast; V)}
            \leq 9 \varepsilon.
        \end{equation*}
        Further, consider $\alpha \coloneq \|x(\cdot \wedge \vartheta_\ast)\|_\infty$ and note that $\alpha > 0$ by the choice of $\vartheta$.
        Due to the convergence $\|x - x_n\|_\infty \to 0$ as $n \to \infty$, we can assume that $\|x_n(\cdot \wedge \vartheta_\ast)\|_\infty > \alpha / 2$ for every $n \in \mathbb{N}$.
        Therefore, for every $t \in [\vartheta_\ast, T]$ and every $n \in \mathbb{N}$,
        \begin{align*}
            &  \biggl| \frac{|x_n(t)|^2}{\|x_n(\cdot \wedge t)\|_\infty^2}
            - \frac{|x(t)|^2}{\|x(\cdot \wedge t)\|_\infty^2} \biggr|
            \leq \frac{\bigl| |x_n(t)|^2 \|x(\cdot \wedge t)\|_\infty^2 - |x(t)|^2 \|x_n(\cdot \wedge t)\|_\infty^2 \bigr|}
            {\|x_n(\cdot \wedge t)\|_\infty^2 \|x(\cdot \wedge t)\|_\infty^2} \\
            & \quad \leq \frac{4}{\alpha^4}
            \bigl| |x_n(t)|^2 \|x(\cdot \wedge t)\|_\infty^2 - |x(t)|^2 \|x_n(\cdot \wedge t)\|_\infty^2 \bigr| \\
            & \quad = \frac{4}{\alpha^4}
            \bigl| (|x_n(t)|^2 - |x(t)|^2) \|x(\cdot \wedge t)\|_\infty^2
            + |x(t)|^2 (\|x(\cdot \wedge t)\|_\infty^2 - \|x_n(\cdot \wedge t)\|_\infty^2)\bigr|  \\
            & \quad \leq \frac{4 \|x\|_\infty^2}{\alpha^4} (\|x\|_\infty + \|x_n\|_\infty)
            \bigl( |x_n(t) - x(t)|
            + \bigl| \|x(\cdot \wedge t)\|_\infty - \|x_n(\cdot \wedge t)\|_\infty \bigr|
            \bigr) \\
            & \quad \leq \frac{8 \|x\|_\infty^2}{\alpha^4} (\|x\|_\infty + \|x_n\|_\infty)
            \|x - x_n\|_\infty.
        \end{align*}
        Hence, we get the convergence
        \begin{align*}
            & \|\partial_x \Upsilon(\cdot, x_n) - \partial_x \Upsilon(\cdot, x) \|_{L^p(\vartheta_\ast, T; V)} \\
            & \quad \leq 4 \|x_n - x\|_{L^p(\vartheta_\ast, T; V)}
            + 4 \biggl( \int_{\vartheta_\ast}^{T} \biggl| \frac{|x_n(t)|^2}{\|x_n(\cdot \wedge t)\|_\infty^2}
            - \frac{|x(t)|^2}{\|x(\cdot \wedge t)\|_\infty^2} \biggr|^p \|x(t)\|^p \, dt \biggr)^{1 / p} \\
            & \quad \leq 4 \|x_n - x\|_{L^p(\vartheta_\ast, T; V)}
            + \frac{32 \|x\|_\infty^2}{\alpha^4} (\|x\|_\infty + \|x_n\|_\infty)
            \|x - x_n\|_\infty \|x\|_{L^p(\vartheta_\ast, T; V)}
            \to 0
        \end{align*}
        as $n \to \infty$.
        Therefore, there exists $N_2 \in \mathbb{N}$ such that, for every $n \geq N_2$,
        \begin{equation*}
            \|\partial_x \Upsilon(\cdot, x_n) - \partial_x \Upsilon(\cdot, x) \|_{L^p(\vartheta_\ast, T; V)}
            \leq \varepsilon.
        \end{equation*}
        As a result, for every $n \in \mathbb{N}$ with $n \geq N_1$ and $n \geq N_2$,
        \begin{align*}
            & \|\partial_x \Upsilon(\cdot, x_n) - \partial_x \Upsilon(\cdot, x) \|_{L^p(\vartheta, T; V)} \\
            & \quad \leq \|( \partial_x \Upsilon(\cdot, x_n) - \partial_x \Upsilon(\cdot, x))
            \mathbf{1}_{(\vartheta, \vartheta_\ast)}(\cdot) \|_{L^p(\vartheta, T; V)} \\
            & \qquad + \|(\partial_x \Upsilon(\cdot, x_n) - \partial_x \Upsilon(\cdot, x))
            \mathbf{1}_{(\vartheta_\ast, T)}(\cdot) \|_{L^p(\vartheta, T; V)} \\
            & \quad = \|\partial_x \Upsilon(\cdot, x_n) - \partial_x \Upsilon(\cdot, x) \|_{L^p(\vartheta, \vartheta_\ast; V)}
            + \|\partial_x \Upsilon(\cdot, x_n) - \partial_x \Upsilon(\cdot, x)\|_{L^p(\vartheta_\ast, T; V)} \\
            & \quad \leq 10 \varepsilon,
        \end{align*}
        which proves that $\|\partial_x \Upsilon(\cdot, x_n) - \partial_x \Upsilon(\cdot, x) \|_{L^p(\vartheta, T; V)} \to 0$ as $n \to \infty$.
        So, we have the required convergence $\|\partial_x \Upsilon(\cdot, x_n) - \partial_x \Upsilon(\cdot, x) \|_{L^p(t_0, T; V)} \to 0$ as $n \to \infty$.

        Now, let $\|x(\cdot \wedge t)\|_\infty > 0$ for all $t \in [t_0, T]$.
        Arguing similarly as above, we can assume that there exists $\alpha > 0$ such that $\|x(\cdot \wedge t)\|_\infty \geq \alpha$ and $\|x_n(\cdot \wedge t)\|_\infty \geq \alpha / 2$ for every $t \in [t_0, T]$ and every $n \in \mathbb{N}$.
        Then, for every $t \in [t_0, T]$ such that $x(t) \in V$ and every $n \in \mathbb{N}$, we derive
        \begin{align*}
            & \|\partial_x \Upsilon(t, x_n) - \partial_x \Upsilon(t, x) \| \\
            & \quad \leq 4 \|x_n(t) - x(t)\|
            + 4 \biggl| \frac{|x_n(t)|^2}{\|x_n(\cdot \wedge t)\|_\infty^2}
            - \frac{|x(t)|^2}{\|x(\cdot \wedge t)\|_\infty^2} \biggr| \|x(t)\| \\
            & \quad \leq 4 \|x_n(t) - x(t)\|
            + \frac{32 \|x\|_\infty^2}{\alpha^4} (\|x\|_\infty + \|x_n\|_\infty)
            \|x - x_n\|_\infty \|x(t)\|.
        \end{align*}
        Hence, we get the required convergence
        \begin{align*}
            & \|\partial_x \Upsilon(\cdot, x_n) - \partial_x \Upsilon(\cdot, x) \|_{L^p(t_0, T; V)} \\
            & \quad \leq 4 \|x_n - x\|_{L^p(t_0, T; V)}
            + \frac{32 \|x\|_\infty^2}{\alpha^4} (\|x\|_\infty + \|x_n\|_\infty)
            \|x - x_n\|_\infty \|x\|_{L^p(t_0, T; V)}
            \to 0
        \end{align*}
        as $n \to \infty$.
        The proof is complete.
    \end{proof}

    We also need the following strengthening of Proposition \ref{proposition_psi_c11V}.
    \begin{proposition} \label{proposition_nuEpsilon_c11V}
        Let $\hat{\nu} \colon [0, T] \times [0, \infty) \to \mathbb{R}$ be a continuously differentiable function.
        Then, the function
        \begin{equation*}
            \nu(t, x)
            \coloneq \hat{\nu}(t, \Upsilon(t, x)),
            \quad (t, x) \in [0, T] \times C([0, T], H),
        \end{equation*}
        belongs to $\mathcal{C}^{1, 1}_V([0, T] \times C([0, T], H))$ and its path-derivatives are given by
        \begin{equation} \label{path-derivatives_nu}
            \partial_t \nu(t, x)
            = \frac{\partial\hat{\nu}}{\partial t}(t, \Upsilon(t, x)),
            \quad \partial_x \nu(t, x)
            = \frac{\partial\hat{\nu}}{\partial r}(t, \Upsilon(t, x)) \partial_x \Upsilon(t, x)
        \end{equation}
        for all $(t, x) \in [0, T] \times C([0, T], H)$.
    \end{proposition}
    \begin{proof}
        Since $\hat{\nu}$ is continuously differentiable, $\Upsilon \in \mathcal{C}^{1, 1}_V([0, T] \times C([0, T], H))$ by Proposition \ref{proposition_psi_c11V}, and $\partial_t \Upsilon(t, x) = 0$ for all $(t, x) \in [0, T] \times C([0, T], H)$, and noting that $(t_0, x_0) \in [0, T] \times C([0, T], H)$ with $x \coloneq x_0 \vert_{[t_0, T]}\in C^1([t_0, T], V)$ implies continuous differentiability of $t \mapsto \mu(t) \coloneq \Upsilon(t, x)$, $[t_0, T] \to \mathbb{R}$, (cf. the proof of Proposition \ref{proposition_ci-smooth_A3}) and thus, for every $t \in [t_0, T]$,
        \begin{equation*}
            \nu(t, x)
            = \hat{\nu}(t, \mu(t))
            = \int_{t_0}^t \biggl( \frac{\partial\hat{\nu}}{\partial t}(s, \mu(s))
            + \frac{\partial\hat{\nu}}{\partial r} (s, \mu(s)) \dot{\mu}(s) \biggr) \, ds
            + \nu(t_0, x_0),
        \end{equation*}
        we conclude (see also Remarks \ref{remark_ci-derivatives_equivalent} and \ref{remark_ci-derivatives_path-derivatives}) that $\nu$ is $ci$-smooth, its $ci$-derivatives are given by
        \begin{equation*}
            \partial_t^{\, ci} \nu(t, x)
            = \frac{\partial\hat{\nu}}{\partial t}(t, \Upsilon(t, x)),
            \quad \partial_x^{\, ci} \nu(t, x)
            = \frac{\partial\hat{\nu}}{\partial r}(t, \Upsilon(t, x)) \partial_x \Upsilon(t, x)
        \end{equation*}
        for all $(t, x) \in [0, T) \times C([0, T], H)$, the functions $\partial_t \nu$ and $\partial_x \nu$ from \eqref{path-derivatives_nu} are continuous, and, for every $x \in C([0, T], H)$ with $x(t) \in V$ for a.e. $t \in [0, T]$, it holds that $\partial_x \nu(t, x) \in V$ for a.e. $t \in [0, T]$.
        So, in order to apply Proposition \ref{proposition_ci-smooth_W_A3}, it remains to verify assumption (iii).
        Fix $t_0 \in [0, T)$ and $x \in C([0, T], H)$ with $x\vert_{(t_0, T)} \in W_{p q}(t_0, T)$ and consider a sequence $(x_n)_n$ in $C([0, T], H)$ such that $x_n\vert_{[t_0, T]} \in C^1([t_0, T], V)$ for every $n \in \mathbb{N}$ and
        \begin{equation*}
            \|x\vert_{(t_0, T)} - x_n\vert_{(t_0, T)}\|_{W_{p q}(t_0, T)}
            + \|x - x_n\|_\infty
            \to 0
            \quad \text{as } n \to \infty.
        \end{equation*}
        In the proof of Proposition \ref{proposition_ci-smooth_W_A3}, it is established that $\Upsilon$ satisfies assumption (iii), which implies that $\| \partial_x \Upsilon (\cdot, x_n) - \partial_x \Upsilon(\cdot, x) \|_{L^p(t_0, T; V)} \to 0$ as $n \to \infty$ and there is a number $L > 0$ such that
        \begin{equation*}
            \|\partial_x \Upsilon(\cdot, x_n)\|_{L^p(t_0, T; V)}
            \leq L,
            \quad n \in \mathbb{N}.
        \end{equation*}
        Since $\|x - x_n\|_\infty \to 0$ as $n \to \infty$, we have
        \begin{equation*}
            \max_{s \in [0, T]} \biggl| \frac{\partial\hat{\nu}}{\partial r}(s, \Upsilon(s, x_n))
            - \frac{\partial\hat{\nu}}{\partial r}(s, \Upsilon(s, x)) \biggr|
            \to 0
            \quad \text{as } n \to \infty
        \end{equation*}
        and there exists $M > 0$ such that
        \begin{equation*}
            \max_{s \in [0, T]} \biggl| \frac{\partial\hat{\nu}}{\partial r}(s, \Upsilon(s, x_n)) \biggr|
            \leq M,
            \quad n \in \mathbb{N}.
        \end{equation*}
        Hence,
        \begin{equation*}
            \|\partial_x \nu(\cdot, x_n)\|_{L^p(t_0, T; V)}
            \leq M \|\partial_x \Upsilon(\cdot, x_n)\|_{L^p(t_0, T; V)}
            \leq M L,
            \quad n \in \mathbb{N},
        \end{equation*}
        and, for every $t \in [0, T]$ such that $x(t) \in V$ and every $n \in \mathbb{N}$,
        \begin{align*}
            \| \partial_x \nu (t, x_n) - \partial_x \nu(t, x) \|
            & \leq \biggl\| \frac{\partial\hat{\nu}}{\partial r}(t, \Upsilon(t, x_n)) \partial_x \Upsilon(t, x_n)
            - \frac{\partial\hat{\nu}}{\partial r}(t, \Upsilon(t, x)) \partial_x \Upsilon(t, x_n) \biggr\| \\
            & \quad + \biggl\| \frac{\partial\hat{\nu}}{\partial r}(t, \Upsilon(t, x)) \partial_x \Upsilon(t, x_n)
            - \frac{\partial\hat{\nu}}{\partial r}(t, \Upsilon(t, x)) \partial_x \Upsilon(t, x) \biggr\| \\
            & \leq \biggl| \frac{\partial\hat{\nu}}{\partial r}(t, \Upsilon(t, x_n))
            - \frac{\partial\hat{\nu}}{\partial r}(t, \Upsilon(t, x)) \biggr|
            \|\partial_x \Upsilon(t, x_n)\| \\
            & \quad + \biggl| \frac{\partial\hat{\nu}}{\partial r}(t, \Upsilon(t, x)) \biggr|
            \| \partial_x \Upsilon(t, x_n) - \partial_x \Upsilon(t, x) \|,
        \end{align*}
        which gives
        \begin{align*}
            & \| \partial_x \nu (\cdot, x_n) - \partial_x \nu(\cdot, x) \|_{L^p(t_0, T; V)} \\
            & \quad \leq L \max_{s \in [0, T]} \biggl| \frac{\partial\hat{\nu}}{\partial r}(s, \Upsilon(s, x_n))
            - \frac{\partial\hat{\nu}}{\partial r}(s, \Upsilon(s, x)) \biggr| \\*
            & \qquad + M \| \partial_x \Upsilon(\cdot, x_n) - \partial_x \Upsilon(\cdot, x) \|_{L^p(t_0, T; V)}
            \to 0
        \end{align*}
        as $n \to \infty$.
        This concludes the proof.
    \end{proof}

\section{Proof of Theorem \ref{theorem_minimax_existence_uniqueness}}
\label{S:proof}

    In the auxiliary statements below, we assume that the hypotheses $\mathbf{H}(A)$, $\mathbf{H}(F)$, and $\mathbf{H}(h)$ are satisfied.

    We start with the following comparison result.

    \begin{lemma} \label{L:DoubledComparison}
        Fix a number $L \geq 0$ and a function $w \colon [0, T] \times C([0, T], H)\times C([0, T], H) \to \mathbb{R}$.
        Suppose that the restrictions of $w$ to the sets $[t_0, T] \times \mathcal{X}^L(t_0, x_0) \times
        \mathcal{X}^L(t_0, x_0)$, $(t_0, x_0)\in [0, T) \times C([0, T], H)$, are upper semi-continuous,
        $w$ satisfies the inequality $w(T, x, x) \leq 0$ for all $x \in C([0, T], H)$, and  $w$ possesses the following property:
        for every point $(t_0, x_0, y_0) \in [0, T) \times C([0, T], H)\times C([0, T], H)$ and every function $\varphi \in \mathcal{C}^{1, 1}_V([t_0, T] \times C([0, T], H) \times C([0, T], H))$ such that, for some $T_0 \in (t_0, T]$,
        \begin{equation} \label{lemma_condition}
            0
            = (\varphi - w)(t_0, x_0, y_0)
            = \inf_{(t, x, y) \in [t_0, T_0] \times \mathcal{X}^L(t_0, x_0) \times \mathcal{X}^L(t_0, y_0)} (\varphi - w)(t, x, y),
        \end{equation}
        there exists a pair $(x, y) \in \mathcal{X}^L(t_0, x_0) \times \mathcal{X}^L(t_0, y_0)$ such that
        \begin{align*}
            & \partial_t \varphi(t_0, x_0, y_0) \\
            & \quad + \limsup_{\delta \to 0^+} \frac{1}{\delta} \int_{t_0}^{t_0 + \delta}
            \bigl( - \langle A(t, x(t)), \partial_x \varphi(t, x, y) \rangle
            - \langle A(t, y(t)), \partial_y \varphi(t, x, y) \rangle \bigr) \, dt \\
            & \quad + F(t_0, x_0, \partial_x \varphi(t_0, x_0, y_0))
            - F(t_0, y_0, - \partial_y \varphi(t_0, x_0, y_0))
            \geq 0.
        \end{align*}
        Then, the inequality $w(t, x, x) \leq 0$ holds for all $(t, x) \in [0, T] \times C([0, T], H)$.
    \end{lemma}

    Note that the regularity and the key property of the function $w$ required in Lemma \ref{L:DoubledComparison} can be considered as a definition of a viscosity subsolution of some doubled equation \cite[Definition 4.1]{Bayraktar_Keller_2018}.
    This definition differs from Definition \ref{definition_viscosity}(i).

    \begin{proof}[Proof of Lemma \ref{L:DoubledComparison}]
        The proof follows the same lines as \cite[Theorem 4.2]{Bayraktar_Keller_2018} but we use here a different penalty functional $\Psi$ (in this connection, see also the proof of \cite[Theorem 5.3]{Bandini_Keller_2025}).

        Arguing by contradiction, assume that there is a point $(t_0, x_0) \in [0, T) \times C([0,T],H)$
        such that $M_0 \coloneq w(t_0, x_0, x_0) > 0$.
        For every $t \in [t_0, T]$ and every $x$, $y \in C([0, T], H)$, define\footnote{Note that the second term in the definition of $\Psi$ is $2 |x(t) - y(t)|^2$ instead of the usual $|x(t) - y(t)|^2$ (see, e.g., \cite{Zhou_2020,Gomoyunov_Lukoyanov_Plaksin_2021,Bandini_Keller_2025}).
        This is done to ensure that $\theta$ from \eqref{Psi_derivatives} and \eqref{theta} is non-negative, which is important when we use the monotonicity of $A$ and condition \eqref{condition_F} in \eqref{proof_Theorem_1}.}
        \begin{equation*}
            \Psi(t, x, y)
            \coloneq
            \Upsilon(t, x - y)
            = \frac{\bigl( \|x(\cdot \wedge t) - y(\cdot \wedge t)\|_\infty^2 - |x(t) - y(t)|^2 \bigr)^2}
            {\|x(\cdot \wedge t) - y(\cdot \wedge t)\|_\infty^2} + 2 |x(t) - y(t)|^2
        \end{equation*}
        if $\|x(\cdot \wedge t) - y(\cdot \wedge t)\|_\infty > 0$ and $\Psi(t, x, y) \coloneq \Upsilon(t, x - y) = 0$ otherwise, where $\Upsilon$ is given by \eqref{Upsilon}.
        Since $\Upsilon \in \mathcal{C}^{1, 1}_V([0, T] \times C([0, T], H))$ by Proposition \ref{proposition_psi_c11V}, we have $\Psi \in \mathcal{C}^{1, 1}_V([t_0, T] \times C([0, T], H) \times C([0, T], H))$ and, for every $t \in [t_0, T)$ and every $x$, $y \in C([0, T], H)$,
        \begin{equation} \label{Psi_derivatives}
            \partial_t \Psi(t, x, y)
            = 0,
            \quad \partial_x \Psi(t, x, y) {\color{black}=-\partial_y\Psi(t,x,y)}
            = \theta(t, x, y) (x(t) - y(t)),
        \end{equation}
        where
        \begin{equation} \label{theta}
            \theta(t, x, y)
            \coloneqq \frac{4 |x(t) - y(t)|^2}
            {\|x(\cdot \wedge t) - y(\cdot \wedge t)\|_\infty^2}
        \end{equation}
        if $\|x(\cdot \wedge t) - y(\cdot \wedge t)\|_\infty > 0$ and $\theta(t, x, y) \coloneq 0$ otherwise.
        Note also that, by \cite[Lemma 2.8]{Zhou_2020},
        \begin{equation} \label{varkappa}
            \varkappa \|x(\cdot \wedge t) - y(\cdot \wedge t)\|_\infty^2
            \leq \Psi(t, x, y)
            \leq 3 \|x(\cdot \wedge t) - y(\cdot \wedge t)\|_\infty^2,
            \quad \varkappa \coloneq \frac{3 - \sqrt{5}}{2},
        \end{equation}
        for all $t \in [t_0, T]$ and $x$, $y \in C([0, T], H)$.

        For every $\varepsilon > 0$, define a map
        $\Phi_\varepsilon \colon [t_0, T] \times \mathcal{X}^L(t_0, x_0) \times \mathcal{X}^L(t_0, x_0) \to \mathbb{R}$ by
        \begin{equation*}
            \Phi_\varepsilon(t, x, y)
            \coloneq w(t, x, y) - \frac{M_0 (T - t)}{2 (T - t_0)}
            - \frac{1}{\varepsilon} \Psi(t, x, y),
            \quad t \in [t_0, T], \ x, y \in \mathcal{X}^L(t_0,x_0).
        \end{equation*}
        Fix a point $k_\varepsilon \coloneq (t_\varepsilon, x_\varepsilon, y_\varepsilon)$ at which the upper semi-continuous map $\Phi_\varepsilon$ attains a maximum (recall that the domain of $\Phi_\varepsilon$ is compact).
        Note that
        \begin{equation} \label{M_varepsilon}
            M_\varepsilon
            \coloneq \Phi_\varepsilon(k_\varepsilon)
            \geq \Phi_\varepsilon(t_0, x_0, x_0)
            = M_0 / 2
            > 0.
        \end{equation}
        {\color{black} Moreover, we have}
        \begin{equation} \label{additional_limit}
            \lim_{\varepsilon \to 0^+} \frac{1}{\varepsilon} \Psi(k_\varepsilon)
            = 0,
            \quad \Psi(\hat{k})
            = 0,
            \quad \lim_{\varepsilon \to 0^+} M_\varepsilon
            = w({\color{black}\hat{k}}) - \frac{M_0 (T - \hat{t})}{2 (T - t_0)}
        \end{equation}
        for every limiting point $\hat{k} \coloneq (\hat{t}, \hat{x}, \hat{y})$ of
        $k_\varepsilon$ as $\varepsilon \to 0^+$ (cf.~\cite[Proposition 3.7]{Crandall_Ishii_Lions_1992} and (4.3) from  \cite{Bayraktar_Keller_2018}).
            {\color{black} To verify   	\eqref{additional_limit},
            it suffices to follow the proof of \cite[Proposition 3.7]{Crandall_Ishii_Lions_1992}
            (the differences are not material).
       For the convenience of the reader, we provide a complete derivation.
        First note that, as $\varepsilon$ decreases to zero,
        $M_\varepsilon$ does not increase and is bounded from below by \eqref{M_varepsilon}.
        Hence, the limit of $M_\varepsilon$ as $\varepsilon \to 0^+$ exists and is finite.
        Moreover, for every $\varepsilon > 0$,
        \begin{align*}
            M_{2 \varepsilon}
            & \geq w(k_\varepsilon)
            - \frac{M_0 (T - t_\varepsilon)}{2 (T - t_0)}
            - \frac{1}{2 \varepsilon} \Psi(k_\varepsilon) \\
            & = w(k_\varepsilon)
            - \frac{M_0 (T - t_\varepsilon)}{2 (T - t_0)}
            - \frac{1}{\varepsilon} \Psi(k_\varepsilon)
            + \frac{1}{2 \varepsilon} \Psi(k_\varepsilon) \\
            & = M_\varepsilon
            + \frac{1}{2 \varepsilon} \Psi(k_\varepsilon),
        \end{align*}
        which gives
        \begin{equation*}
            0
            \leq \frac{1}{\varepsilon} \Psi(k_\varepsilon)
            \leq 2 (M_{2 \varepsilon} - M_\varepsilon).
        \end{equation*}
        Letting $\varepsilon \to 0^+$, we get the first equality in \eqref{additional_limit}.
        In particular, $\Psi(k_\varepsilon) \to 0$ as $\varepsilon \to 0^+$ and, therefore, $\Psi(\hat{k}) = 0$ by continuity of $\Psi$, i.e., the second equality in \eqref{additional_limit} takes place.
        Further, thanks to the upper semicontinuity of the restriction of $w$ to $[t_0, T] \times \mathcal{X}^L(t_0, x_0) \times \mathcal{X}^L(t_0, x_0)$ and continuity of $\Psi$, we derive
        \begin{align*}
            \lim_{\varepsilon \to 0^+} M_\varepsilon
            & = \lim_{\varepsilon \to 0^+} \biggl( w(k_\varepsilon)
            - \frac{M_0 (T - t_\varepsilon)}{2 (T - t_0)}
            - \frac{1}{\varepsilon} \Psi(k_\varepsilon) \biggr) \\
            & \leq \limsup_{\varepsilon \to 0^+} w(k_\varepsilon)
            - \lim_{\varepsilon \to 0^+} \biggl( \frac{M_0 (T - t_\varepsilon)}{2 (T - t_0)}
            + \frac{1}{\varepsilon} \Psi(k_\varepsilon) \biggr) \\
            & \leq w(\hat{k}) - \frac{M_0 (T - \hat{t})}{2 (T - t_0)}.
        \end{align*}
        On the other hand, using the equality $\Psi(\hat{k}) = 0$, we obtain
        \begin{equation*}
            M_\varepsilon
            \geq w(\hat{k}) - \frac{M_0 (T - \hat{t})}{2 (T - t_0)} - \frac{1}{\varepsilon} \Psi(\hat{k})
            = w(\hat{k}) - \frac{M_0 (T - \hat{t})}{2 (T - t_0)}
        \end{equation*}
        for all $\varepsilon > 0$, which implies that
        \begin{equation*}
            \lim_{\varepsilon \to 0^+} M_\varepsilon
            \geq w(\hat{k}) - \frac{M_0 (T - \hat{t})}{2 (T - t_0)}.
        \end{equation*}
        Thus, the third equality in \eqref{additional_limit} holds.
        }

        Note that $\Psi(\hat{k}) = 0$ implies that $\hat{x}(t) = \hat{y}(t)$ for all $t \in [0, \hat{t}]$.
        Then, using \eqref{M_varepsilon} and the inequality $w(T, \hat{x}, \hat{x}) \leq 0$, we derive $\hat{t} < T$ and, hence, there exists $\varepsilon_\ast > 0$ such that $t_\varepsilon < T$ for all $\varepsilon \in (0, \varepsilon_\ast]$.

        For every $\varepsilon \in (0, \varepsilon_\ast]$, define a map $\varphi_\varepsilon \colon [t_\varepsilon, T] \times C([0, T], H) \times C([0, T], H) \to \mathbb{R}$ by
        \begin{equation*}
            \varphi_\varepsilon(t, x, y)
            \coloneq \frac{M_0 (T - t)}{2 (T - t_0)} + \frac{1}{\varepsilon} \Psi(t, x, y)
            + M_\varepsilon,
            \quad t \in [t_\varepsilon, T], \ x, y \in C([0, T], H).
        \end{equation*}
        Note that $\varphi_\varepsilon \in \mathcal{C}^{1, 1}_V([t_\varepsilon, T] \times C([0, T], H) \times C([0, T], H))$ with
        \begin{equation*}
            \partial_t \varphi_\varepsilon(t, x, y)
            = - \frac{M_0}{2 (T - t_0)},
            \quad \partial_x \varphi_\varepsilon(t, x, y)
            = - \partial_y \varphi_\varepsilon(t, x, y)
            = \frac{\theta(t, x, y)}{\varepsilon} (x(t) - y(t))
        \end{equation*}
        for all $t \in [t_\varepsilon, T)$ and all $x$, $y \in C([0, T], H)$ and $\varphi_\varepsilon$ satisfies condition \eqref{lemma_condition}, where we substitute $\varphi \coloneq \varphi_\varepsilon$, $(t_0, x_0, y_0) \coloneq k_\varepsilon$ and $T_0 \coloneq T$.
        Then, there exists a pair $(x, y) \in \mathcal{X}^L(t_0, x_0) \times \mathcal{X}^L(t_0, y_0)$ such that (compare with (4.4) from \cite{Bayraktar_Keller_2018}):
        \begin{align}
            0
            & \leq - \frac{M_0}{2 (T - t_0)} + \limsup_{\delta \downarrow 0} \frac{1}{\delta}
            \int_{t_0}^{t_0 + \delta}  \frac{- \theta(t, x, y)}{\varepsilon}
            \Bigl\langle A(t, x(t)) - A(t, y(t)), x(t) - y(t) \Bigr\rangle \, dt \nonumber \\
            & \quad + F \Bigl( t_\varepsilon, x_\varepsilon,
            \frac{\theta(t_\varepsilon, x_\varepsilon, y_\varepsilon)}{\varepsilon} \bigl( x_\varepsilon(t_\varepsilon) - y_\varepsilon(t_\varepsilon) \bigr) \Bigr)
            - F \Bigl( t_\varepsilon, y_\varepsilon,
            \frac{\theta(t_\varepsilon, x_\varepsilon, y_\varepsilon)}{\varepsilon} \bigl( x_\varepsilon(t_\varepsilon) - y_\varepsilon(t_\varepsilon) \bigr) \Bigr) \nonumber \\
            & \leq - \frac{M_0}{2 (T - t_0)} + m_{L, t_0, x_0} \biggl(
            \frac{\theta(t_\varepsilon, x_\varepsilon, y_\varepsilon)}{\varepsilon}
            |x_\varepsilon(t_\varepsilon) - y_\varepsilon(t_\varepsilon)|
            \|x_\varepsilon(\cdot \wedge t_\varepsilon) - y_\varepsilon(\cdot \wedge t_\varepsilon)\|_\infty \nonumber \\
            & \quad + \|x_\varepsilon(\cdot \wedge t_\varepsilon) - y_\varepsilon(\cdot \wedge t_\varepsilon)\|_\infty \biggr) \nonumber \\
            & \leq - \frac{M_0}{2 (T - t_0)} + m_{L, t_0, x_0} \biggl(
            \frac{4 \Psi(k_\varepsilon)}{\varepsilon \varkappa}
            + \sqrt{\frac{\Psi(k_\varepsilon)}{\varkappa}} \biggr). \label{proof_Theorem_1}
        \end{align}
        Here, the second inequality is obtained from condition $\mathbf{H}(F)$(iii)\footnote{Note that the condition $\mathbf{H}(F)$(iii) can be weakened.
        Namely, instead of \eqref{condition_F}, we can require
        \begin{equation*}
            \begin{aligned}
                & F \bigl( t, x_1, \varepsilon^{- 1} (x_1(t) - x_2(t)) \bigr)
                - F \bigl( t, x_2, \varepsilon^{- 1} (x_1(t) - x_2(t)) \bigr) \\
                & \quad \leq m_{L, t_0, x_0} \Bigl( \varepsilon^{- 1}
                \|x_1(\cdot \wedge t) - x_2(\cdot \wedge t)\|_\infty^2
                + \|x_1(\cdot \wedge t) - x_2(\cdot \wedge t)\|_\infty \Bigr).
            \end{aligned}
        \end{equation*}
        However, we do not have a specific example of a function $F$ satisfying this weakened condition and not condition $\mathbf{H}(F)$(iii).}.
        Letting $\varepsilon \to 0^+$, we come to a contradiction.
        The proof is complete.
    \end{proof}

    Thanks to the results in \cite{Bayraktar_Keller_2018} and Lemma \ref{L:DoubledComparison}, we  immediately have the following uniqueness result.

    \begin{lemma} \label{L:MinimaxSolution:GlobalUniqueness}
        There exists at most one minimax solution of \eqref{TVP}.
    \end{lemma}

    Now we turn to existence of minimax solutions on $C([0,T], H)$.
    We need the following definition for (local) minimax solution.
    \begin{definition} \label{D:LocalMinimaxSolution:Updated}
        Let $L \geq 0$, $t^\ast \in [0, T)$, and $x^\ast \in C([0, T], H)$.
        We call a function $u \colon [t^\ast, T] \times \mathcal{X}^L(t^\ast, x^\ast) \to \mathbb{R}$ a minimax $L$-solution of \eqref{TVP} on $[t^\ast, T] \times \mathcal{X}^L(t^\ast, x^\ast)$ if
        $u$ is continuous,
        the equality $u(T, x) = h(x)$ holds for all $x \in \mathcal{X}^L(t^\ast, x^\ast)$,
        and, for every $(t_0, x_0, z) \in [t^\ast, T) \times \mathcal{X}^L(t^\ast, x^\ast) \times H$, there exists an $x \in \mathcal{X}^L(t_0, x_0)$ such that, for every $t \in [t_0, T]$, equality \eqref{E:Minimax:L:Solution} holds.
    \end{definition}

    \begin{remark} \label{R:restrictionMinimax}
        Let $L \geq 0$, $t^\ast \in [0, T)$, $x^\ast \in C([0, T], H)$, and $(t_0, x_0) \in \mathcal{X}^L(t^\ast, x^\ast)$.
        Then the restriction of a minimax $L$-solution of \eqref{TVP} on $[t^\ast, T] \times \mathcal{X}^L(t^\ast, x^\ast)$ to $[t_0, T] \times \mathcal{X}^L(t_0, x_0)$ is also a minimax $L$-solution of \eqref{TVP} on $[t_0, T] \times \mathcal{X}^L(t_0, x_0)$
        (cf. \cite[Remark 3.4]{Bayraktar_Keller_2018}).
    \end{remark}

    Now, we fix $L \geq L_0$, $t^\ast \in [0, T)$, and $x^\ast \in C([0, T], H)$.
    Replacing $[0, T] \times \Omega^L$, where $\Omega^L \coloneq \mathcal{X}^L(0, x_\ast)$, in \cite[Section 5]{Bayraktar_Keller_2018} and in \cite[Proposition 2.12]{Bayraktar_Keller_2018} by $[t^\ast, T] \times \mathcal{X}^L(t^\ast, x^\ast)$, we obtain a local well-posedness result (cf. \cite[Theorem 5.7]{Bayraktar_Keller_2018}).

    \begin{lemma} \label{L:localExistence}
        Let $L \geq L_0$, $t^\ast \in [0, T)$, and $x^\ast \in C([0, T], H)$ be fixed.
        Then, there exists a unique minimax $L$-solution $u^L_{t^\ast, x^\ast}$ of \eqref{TVP} on $[t^\ast, T] \times \mathcal{X}^L(t^\ast, x^\ast)$.
    \end{lemma}

    From Lemma \ref{L:localExistence}, taking Definition \ref{D:LocalMinimaxSolution:Updated} into account, we derive the following result.
    \begin{corollary}\label{Cor:localExistence}
        Let $t^\ast \in [0, T)$, $x^\ast \in C([0 , T], H)$, and $L_2 \geq L_1 \geq L_0$.
        Then, we have $u^{L_2}_{t^\ast, x^\ast}(t, x) = u^{L_1}_{t^\ast, x^\ast}(t, x)$ for all $(t, x) \in [t^\ast, T] \times \mathcal{X}^{L_1}(t^\ast, x^\ast)$.
    \end{corollary}

    Finally, we are able to establish a global well-posedness result, which is Theorem \ref{theorem_minimax_existence_uniqueness}.

    \begin{proof}[Proof of Theorem \ref{theorem_minimax_existence_uniqueness}]
        Define $u \colon [0, T] \times C([0, T], H) \to \mathbb{R}$ by
        \begin{equation*}
            u(t^\ast,x^\ast)
            \coloneq u^{L_0}_{t^\ast,x^\ast}(t^\ast,x^\ast)
        \end{equation*}
        if $t^\ast < T$ and $u(T, x^\ast) \coloneq h(x^\ast)$ for all $x^\ast \in C([0, T], H)$, where $u^{L_0}_{t^\ast, x^\ast}$ is the unique minimax $L_0$-solution of \eqref{TVP} on $[t^\ast, T] \times \mathcal{X}^{L_0}(t^\ast, x^\ast)$ from Lemma \ref{L:localExistence}.
        Note that, by Corollary \ref{Cor:localExistence},
        \begin{equation} \label{E:globalExistence}
            u(t^\ast, x^\ast)
            = u^L_{t^\ast, x^\ast}(t^\ast, x^\ast),
            \quad (t^\ast, x^\ast, L) \in [0, T) \times C([0, T], H) \times [L_0, \infty).
        \end{equation}

        Due to Lemma \ref{L:MinimaxSolution:GlobalUniqueness}, we only need to  show that $u$ is a minimax $L_0$-solution of \eqref{TVP}.
        To this end, the following observation is crucial:
        for every $t^\ast \in [0, T)$, $x^\ast \in C([0, T], H)$, and $L \geq L_0$, we have $u(t, x) = u^L_{t^\ast, x^\ast}(t, x)$ for all $(t, x) \in [t^\ast, T) \times \mathcal{X}^L(t^\ast, x^\ast)$.
        To see this, fix $t^\ast\in [0,T)$, $x^\ast\in C([0,T],H)$, $L\ge L_0$, $(t,x)\in [t^\ast,T)\times\mathcal{X}^L(t^\ast,x^\ast)$, and note that, by Lemma \ref{L:localExistence}, $u^L_{t, x}$ is the unique minimax $L$-solution of \eqref{TVP} on $[t, T] \times \mathcal{X}^L(t, x)$.
        Taking additionally into account that $u^L_{t^\ast, x^\ast}$ is, again by Lemma \ref{L:localExistence}, the unique minimax $L$-solution of \eqref{TVP} on $[t^\ast, T] \times \mathcal{X}^L(t^\ast, x^\ast)$ and that $[t, T] \times \mathcal{X}^L(t, x) \subset [t^\ast, T] \times \mathcal{X}^L(t^\ast, x^\ast)$, we deduce that $u^L_{t, x} (\tilde{t}, \tilde{x}) = u^L_{t^\ast, x^\ast}(\tilde{t}, \tilde{x})$ for every $(\tilde{t}, \tilde{x}) \in [t, T] \times \mathcal{X}^L(t,x)$ (see Remark \ref{R:restrictionMinimax}).
        In particular, $u^L_{t, x}(t, x) = u^L_{t^\ast, x^\ast}(t, x)$ and thus, together with \eqref{E:globalExistence}, we have $u(t, x) = u^L_{t^\ast, x^\ast}(t, x)$.
        Hence, the mentioned observation has been established.

        Let us verify that $u$ satisfies (a) regularity, (b) terminal, and (c) interior conditions required in Definition \ref{D:minimax_solutions}(iii).

        (a)
        Fix $t^\ast \in [0, T)$, $x^\ast \in C([0, T], H)$, and $L \geq L_0$.
        Note that $u(T, x) = h(x) = u^L_{t^\ast, x^\ast}(T, x)$ for all $x \in \mathcal{X}^L(t^\ast, x^\ast)$ and, from the observation above, $u(t, x) = u^L_{t^\ast, x^\ast}(t, x)$ for all $(t, x) \in [t^\ast, T) \times \mathcal{X}^L(t^\ast, x^\ast)$.
        Hence, recalling that $u^L_{t^\ast, x^\ast}$ is continuous since it is a minimax $L$-solution on $[t^\ast, T] \times \mathcal{X}^L(t^\ast, x^\ast)$ (see Definition \ref{D:LocalMinimaxSolution:Updated}), we obtain that the restriction of $u$ on $[t^\ast, T] \times \mathcal{X}^L(t^\ast, x^\ast)$ is continuous.

        (b)
        The terminal condition follows immediately from the definition of $u$.

        (c)
        Fix $(t_0, x_0) \in [0, T) \times C([0, T], H)$ and $z \in H$.
        Since $u^{L_0}_{t_0, x_0}$ is a minimax $L_0$-so\-lu\-tion of \eqref{TVP} on $[t_0, T] \times \mathcal{X}^{L_0}(t_0, x_0)$, there exists an $x \in \mathcal{X}^{L_0}(t_0, x_0)$ such that, for all $t\in [t_0,T]$,
        \begin{align*}
            u(t_0, x_0)
            & = u^{L_0}_{t_0, x_0}(t_0, x_0)
            = \int_{t_0}^t \bigl( (- f^x(s), z) + F(s, x, z) \bigr) \, ds + u^{L_0}_{t_0, x_0}(t, x) \\
            & = \int_{t_0}^t \bigl( (- f^x(s), z) + F(s, x, z) \bigr) \,ds + u(t, x),
        \end{align*}
        where the last equality comes from the mentioned observation above.

        Consequently, $u$ is a minimax $L_0$-solution of \eqref{TVP}.
        This concludes the proof.
  \end{proof}

\section{Stability}

    In line with the standard viscosity solution theory (see, e.g., \cite[Section 8]{Crandall_Primer}),
    we establish here counterparts of the Barles--Perthame ``method of half-relaxed limits'' to our setting, which was also done in an infinite-dimensional (but non-path-dependent) setting in \cite{Kelome_Swiech_06}, where $A$ is linear, monotone, and coercive.
    Note that coerciveness of the operator $A$ is crucial in our work, because thanks to this property
    the spaces $\mathcal{X}^L(t, x)$ are compact, and it is also crucial in \cite{Kelome_Swiech_06},
    as, e.g., even in the case $A \equiv 0$, the method of half-relaxed limits may not work in infinite dimensions (see \cite[Example 5.4.3]{Fabbri_Gozzi_Swiech_17}, which is due to \cite{Swiech_02}).
    We use a minimax solution approach to prove our stability results in a similar way as
    in \cite{Subbotin_1993} and in the path-dependent finite-dimensional setting in \cite{Lukoyanov_2001}.
    While this topic was also addressed in \cite{Bayraktar_Keller_2018}, the stability result therein was valid only for minimax solutions on the locally compact space $\Omega$ (see the discussion
    just before Theorem \ref{theorem_minimax_existence_uniqueness}) whereas our result (see Corollary \ref{C:Stability} below) is valid for viscosity solutions on the whole path space $C([0, T], H)$.

    For $n \in \mathbb{N}$, consider functions $F_n \colon [0, T] \times C([0, T], H) \times H \to \mathbb{R}$, $h_n \colon C([0, T], H) \to \mathbb{R}$ and the terminal-value problem for the path-dependent Hamilton--Jacobi equation
    \begin{subequations} \label{TVPn}
    \begin{equation} \label{HJn}
        \begin{gathered}
        \partial_t u(t, x) - \langle A(t, x(t)), \partial_x u(t, x) \rangle + F_n(t, x, \partial_x u(t, x))
        = 0, \\
        (t, x) \in [0, T) \times C([0, T], H),
        \end{gathered}
    \end{equation}
    under the right-end boundary condition
    \begin{equation} \label{boundary_conditionn}
        u(T, x)
        = h_n(x),
        \quad  x \in C([0, T], H).
    \end{equation}
    \end{subequations}

    Given $L \ge 0$, $(t^\ast, x^\ast) \in [0, T) \times C([0, T], H)$, define neighborhoods
    \begin{align*}
        O^L_{t^\ast, x^\ast; \delta}(x_0)
        & \coloneq \bigl\{ x \in \mathcal{X}^L(t^\ast, x^\ast)
        \colon \|x - x_0\|_\infty \le \delta \bigr\}, \\
        O^L_{t^\ast, x^\ast; \delta}(t_0, x_0)
        & \coloneq \bigl\{ (t, x) \in [t^\ast, T] \times \mathcal{X}^L(t^\ast, x^\ast)
        \colon \mathbf{d}_\infty((t, x), (t_0, x_0)) \le \delta \bigr\}, \\
        O^L_{t^\ast, x^\ast; \delta}(t_0, x_0, z_0)
        & \coloneq O^L_{t^\ast, x^\ast; \delta}(t_0, x_0)
        \times \bigl\{ z \in H \colon |z - z_0| \le \delta \bigr\}
    \end{align*}
    for each $\delta > 0$ and $(t_0, x_0, z_0) \in [t^\ast, T] \times \mathcal{X}^L(t^\ast, x^\ast) \times H$;
    functions $F^{L, +}_{t^\ast, x^\ast}$, $F^{L, -}_{t^\ast, x^\ast} \colon [t^\ast, T] \times \mathcal{X}^L(t^\ast, x^\ast) \times H \to \mathbb{R} \cup\{-\infty, \infty\}$, which can be considered as upper and lower half-relaxed limits of the sequence $(F_n)_n$, by
    \begin{align}
        F^{L, +}_{t^\ast, x^\ast}(t_0, x_0, z_0)
        & \coloneq \limsup_{
            \substack{n \to \infty, \\
            (t, x, z) \to (t_0, x_0, z_0) \\
            \text{in } [t^\ast, T] \times \mathcal{X}^L(t^\ast, x^\ast) \times H}}
        F_n(t, x, z) \nonumber \\*
        & \coloneq \lim_{n \to \infty} \sup_{
            \substack{(t, x, z) \in O^L_{t^\ast, x^\ast; 1/m}(t_0, x_0, z_0), \\
            m \ge n}}
        F_m(t, x, z), \label{E:FnL+} \\
        F^{L, -}_{t^\ast, x^\ast}(t_0, x_0, z_0)
        & \coloneq \liminf_{
            \substack{n \to \infty,\\
            (t, x, z) \to (t_0, x_0, z_0) \\
            \text{in } [t^\ast, T] \times \mathcal{X}^L(t^\ast, x^\ast) \times H}}
        F_n(t,x,z) \nonumber \\*
        & \coloneq \lim_{n \to \infty} \inf_{
            \substack{(t, x, z) \in O^L_{t^\ast, x^\ast; 1/m}(t_0, x_0, z_0), \\
            m \ge n}}
        F_m(t,x,z), \nonumber
    \end{align}
    where $(t_0, x_0, z_0) \in [t^\ast, T] \times \mathcal{X}^L(t^\ast, x^\ast) \times H$;
    and functions $h^{L, +}_{t^\ast, x^\ast}$, $h^{L, -}_{t^\ast, x^\ast} \colon \mathcal{X}^L(t^\ast, x^\ast) \to \mathbb{R} \cup \{-\infty, \infty\}$ by
    \begin{align}
        h^{L, +}_{t^\ast, x^\ast}(x_0)
        & \coloneq \limsup_{
            \substack{n \to \infty,\\
            x \to x_0
            \text{ in} \mathcal{X}^L(t^\ast, x^\ast)}}
        h_n(x)
        \coloneq \lim_{n \to \infty} \sup_{
            \substack{x \in O^L_{t^\ast, x^\ast; 1/m}(x_0), \\
            m \ge n}}
        h_m(x), \label{E:hnL+} \\
        h^{L, -}_{t^\ast, x^\ast}(x_0)
        & \coloneq \liminf_{
            \substack{n \to \infty, \\
            x \to x_0
            \text{ in} \mathcal{X}^L(t^\ast, x^\ast)}}
        h_n(x)
        \coloneq \lim_{n \to \infty} \inf_{
            \substack{x \in O^L_{t^\ast, x^\ast; 1/m}(x_0), \\
            m \ge n}}
        h_m(x). \nonumber
    \end{align}

    Similarly, given functions $u_n \colon [0, T] \times C([0, T], H) \to \mathbb{R}$, $n \in \mathbb{N}$, define functions $u^{L, +}_{t^\ast, x^\ast}$, $u^{L, -}_{t^\ast, x^\ast} \colon [t^\ast, T] \times \mathcal{X}^L(t^\ast, x^\ast) \to \mathbb{R} \cup \{- \infty, \infty\}$, $L \ge 0$, $t^\ast \in [0, T)$, $x^\ast \in C([0, T], H)$, by
    \begin{align}
        u^{L, +}_{t^\ast, x^\ast}(t_0, x_0)
        & \coloneq \limsup_{
            \substack{n \to \infty, \\
            (t, x) \to (t_0, x_0) \\
            \text{in } [t^\ast, T] \times \mathcal{X}^L(t^\ast, x^\ast)}}
        u_n(t,x)
        \coloneq \lim_{n \to \infty} \sup_{
            \substack{(t, x) \in O^L_{t^\ast, x^\ast; 1/m}(t_0, x_0), \\
            m \ge n}}
        u_m(t, x), \label{E:unL+} \\
        u^{L, -}_{t^\ast, x^\ast}(t_0, x_0)
        & \coloneq \liminf_{
            \substack{n \to \infty, \\
            (t, x) \to (t_0, x_0) \\
            \text{in } [t^\ast, T] \times \mathcal{X}^L(t^\ast, x^\ast)}}
        u_n(t, x)
        \coloneq \lim_{n \to \infty} \inf_{
            \substack{(t, x) \in O^L_{t^\ast, x^\ast; 1/m}(t_0, x_0), \\
            m \ge n}}
        u_m(t, x), \nonumber
    \end{align}
    where $(t_0, x_0) \in [t^\ast, T] \times \mathcal{X}^L(t^\ast, x^\ast)$;
    and functions $u^{L, +}$, $u^{L, -} \colon [0, T] \times C([0, T], H) \to \mathbb{R} \cup \{- \infty, \infty\}$ by
    \begin{equation} \label{E:uL+-:stab}
        \begin{aligned}
            u^{L, +}(t^\ast,x^\ast)
            & \coloneq \begin{cases}
                u^{L, +}_{t^\ast, x^\ast}(t^\ast, x^\ast), & \text{if } t^\ast < T, \\
                h(x^\ast), & \text{if } t^\ast = T,
                \end{cases} \\
            u^{L, -}(t^\ast, x^\ast)
            & \coloneq \begin{cases}
                u^{L, -}_{t^\ast, x^\ast}(t^\ast, x^\ast), & \text{if } t^\ast < T, \\
                h(x^\ast), & \text{if } t^\ast = T,
                \end{cases}
        \end{aligned}
    \end{equation}
    where $(t^\ast, x^\ast) \in [0, T] \times C([0, T], H)$.

    We  will need the following local notions of minimax and viscosity sub- and supersolutions (cf. Definition \ref{D:LocalMinimaxSolution:Updated}) as well as viscosity solutions of the terminal-value problem \eqref{TVP}.

    \begin{definition}\label{D:LocalSemiSolution}
        Let $L \ge 0$, $t^\ast \in [0, T)$, $x^\ast \in C([0, T], H)$, $u \colon [t^\ast, T] \times \mathcal{X}^L(t^\ast, x^\ast) \to \mathbb{R}$.

        (i)
            The function $u$ is a minimax $L$-subsolution (resp. supersolution) of \eqref{TVP}         on $[t^\ast, T] \times \mathcal{X}^L(t^\ast, x^\ast)$ if $u$ is upper (resp. lower) semi-continuous, $u(T, x) \le$ (resp. $\ge$) $h(x)$ for each $x \in \mathcal{X}^L(t^\ast, x^\ast)$, and, for every $(t_0, x_0, z) \in [t^\ast, T) \times \mathcal{X}^L(t^\ast, x^\ast) \times H$, there is an $x \in \mathcal{X}^L(t_0, x_0)$ such that, for each $t \in [t_0, T]$, inequality \eqref{E:Minimax:L:Subsolution} (resp. \eqref{E:Minimax:L:Supersolution}) holds.

        (ii)
            The function $u$ is a viscosity $L$-subsolution (resp. supersolution) of \eqref{TVP} on $[t^\ast, T] \times \mathcal{X}^L(t^\ast, x^\ast)$ if $u$ is upper (resp. lower) semi-continuous, $u(T, x) \le$ (resp. $\ge$) $h(x)$ for each $x \in \mathcal{X}^L(t^\ast, x^\ast)$, and, for every $(t_0, x_0) \in [t^\ast, T) \times \mathcal{X}^L(t^\ast, x^\ast)$ and every $(\varphi, z) \in \underline{\mathcal{A}}^L u(t_0, x_0)$ (resp. $\overline{\mathcal{A}}^L u(t_0, x_0)$), inequality \eqref{definition_viscosity_subsolution} (resp. \eqref{definition_viscosity_supersolution}) holds.

        (iii)
            The function $u$ is a viscosity $L$-solution  of \eqref{TVP} on $[t^\ast, T] \times \mathcal{X}^L(t^\ast, x^\ast)$ if $u$ is a viscosity $L$-sub- and $L$-supersolution of \eqref{TVP} on $[t^\ast, T] \times \mathcal{X}^L(t^\ast, x^\ast)$.
    \end{definition}

    \begin{remark}\label{remark_equivalence_local}
        Note that the equivalence result between viscosity sub/super- and minimax sub/supersolutions stated in Theorem \ref{theorem_equivalence} and the comparison principle in form of Lemma \ref{L:DoubledComparison} are also valid for the local notions in Definition \ref{D:LocalSemiSolution} with obvious adjustments.
    \end{remark}

    \begin{theorem} \label{T:Stability}
        Fix $L \ge 0$.
        Suppose that the following holds:

        {\rm (i)}
            For each $n \in \mathbb{N}$, the function $u_n$ is a viscosity $L$-subsolution {\rm(}resp. $L$-supersolution{\rm)} of \eqref{TVPn} and the Hamiltonian $F_n$ is continuous.
            Moreover, $(u_n)_n${\color{black}, $(F_n)_n$, and $(h_n)_n$ are} 
            locally uniformly bounded in the following sense:
            \begin{equation}\label{E:T:HalfRelaxedLimit:UniformBound}
                \sup_{n \in \mathbb{N}, \, {\color{black} (t, x) \in [t^\ast, T] \times \mathcal{X}^L (t^\ast, x^\ast)}}
                {\color{black}
                \bigl( {\color{black}|u_n(t, x)|}
                 +|F_n(t,x,z)
                 + |h_n(x)|
                \bigr)}
                < \infty
            \end{equation}
            for all $(t^\ast, x^\ast{\color{black},z}) \in [0, T) \times C([0, T], H){\color{black}\times H}$.

        {\rm (ii)}
            The function $F$ satisfies hypothesis $\mathbf{H}(F)${\rm(i)} and $F \ge F^{L, +}_{t^\ast, x^\ast}$ {\rm(}resp. $F \le F^{L, -}_{t^\ast, x^\ast}${\rm)} on each set $[t^\ast, T] \times \mathcal{X}^L (t^\ast, x^\ast)$, $(t^\ast, x^\ast) \in [0, T) \times C([0, T], H)$.

        {\rm (iii)}
            The function $h$ satisfies hypothesis $\mathbf{H}(h)$ and $h \ge h^{L, +}_{t^\ast, x^\ast}$
            {\rm(}resp. $h \le h^{L,-}_{t^\ast, x^\ast}${\rm)} on each set $\mathcal{X}^L (t^\ast, x^\ast)$, $(t^\ast, x^\ast) \in [0, T) \times C([0, T], H)$.

        Then, the functions $u^{L, +}_{t^\ast, x^\ast}$ {\rm(}resp. $u^{L, -}_{t^\ast, x^\ast}${\rm)}, $(t^\ast, x^\ast) \in [0, T) \times C([0, T], H)$, are viscosity $L$-subsolutions
        {\rm(}resp. $L$-supersolutions{\rm)} of \eqref{TVP} on $[t^\ast, T] \times \mathcal{X}^L(t^\ast, x^\ast)$.
    \end{theorem}

    \begin{corollary}\label{C:Stability}
        Fix $L \ge 0$.
        Suppose that the following holds:

        {\rm (i)}
            For each $n \in \mathbb{N}$, the function $u_n$ is a viscosity $L$-solution of \eqref{TVPn}
            and the Hamiltonian $F_n$ is continuous.
            Furthermore, \eqref{E:T:HalfRelaxedLimit:UniformBound} holds.

        {\rm (ii)}
            The functions $F$ and $h$ satisfy hypotheses $\mathbf{H}(F)$ and $\mathbf{H}(h)$.
            Moreover, let $F = F^{L, +}_{t^\ast, x^\ast} = F^{L, -}_{t^\ast, x^\ast}$ on $[t^\ast, T] \times \mathcal{X}^L (t^\ast, x^\ast)$ and $h = h^{L, +}_{t^\ast, x^\ast} = h^{L, -}_{t^\ast, x^\ast}$ on $\mathcal{X}^L(t^\ast, x^\ast)$ for each $(t^\ast, x^\ast) \in [0, T) \times C([0, T], H)$.

        Then, the functions $u^{L, +}$ and $u^{L, -}$ defined by \eqref{E:uL+-:stab} coincide and, for every $(t^\ast, x^\ast) \in [0, T) \times C([0, T], H)$, the restriction of $u \coloneq u^{L, +} = u^{L, -}$ to $[t^\ast, T] \times \mathcal{X}^L(t^\ast, x^\ast)$ is the unique viscosity $L$-solution of \eqref{TVP} on $[t^\ast, T] \times \mathcal{X}^L (t^\ast, x^\ast)$.\footnote{The function $u$ is basically a viscosity $L$-solution in the sense of Definition \ref{definition_viscosity} but with slightly less regularity.}
    \end{corollary}

    \begin{proof}[Proof of Theorem \ref{T:Stability}]
        Fix $t^\ast \in [0, T)$ and $x^\ast \in C([0, T], H)$.
        We consider only the case of subsolutions.
        First, note that then each function $u_n$, $n \in \mathbb{N}$, is a minimax $L$-subsolution of \eqref{TVPn} thanks to Theorem \ref{theorem_equivalence}.
        We will show that $u^{L, +}_{t^\ast, x^\ast}$ is a minimax $L$-subsolution of \eqref{TVP} on $[t^\ast, T] \times \mathcal{X}^L(t^\ast, x^\ast)$, which will immediately imply that it is also a viscosity $L$-subsolution of \eqref{TVP} on $[t^\ast, T] \times \mathcal{X}^L(t^\ast, x^\ast)$ by Theorem \ref{theorem_equivalence} (see Remark \ref{remark_equivalence_local}).

        We start with the boundary condition.
        Let $x_0 \in \mathcal{X}^L(t^\ast, x^\ast)$.
        By \eqref{E:unL+}, there are a sequence $((t_k, x_k))_k$ in $[t^\ast, T] \times \mathcal{X}^L(t^\ast, x^\ast)$ and a subsequence $(u_{n_k})_k$ of $(u_n)_n$ such that
        \begin{equation*}
            (T, x_0)
            = \lim_{k \to \infty} (t_k, x_k),
            \quad u^{L, +}_{t^\ast,x^\ast}(T, x_0) = \lim_{k \to \infty} u_{n_k}(t_k, x_k).
        \end{equation*}
        Since the functions $u_{n_k}$, $k \in \mathbb{N}$, are minimax $L$ -subsolutions of \eqref{TVPn} with $n$ replaced by $n_k$, there is an $\tilde{x}_k \in \mathcal{X}^L(t_k, x_k)$ such that
        \begin{equation*}
            u_{n_k}(t_k, x_k)
            \le \int_{t_k}^T F_{n_k}(s, \tilde{x}_k, 0) \, ds + u_{n_k}(T, \tilde{x}_k)
            \le \int_{t_k}^T F_{n_k}(s, \tilde{x}_k, 0) \, ds + h_{n_k}(\tilde{x}_k).
        \end{equation*}
        By the proof of \cite[Proposition 2.12]{Bayraktar_Keller_2018}, we can assume that $(\tilde{x}_k)_k$ converges in $\mathcal{X}^L(t^\ast, x^\ast)$ to $x_0$.
        Moreover, we can assume that $\|\tilde{x}_k - x_0\|_\infty \le 1 / n_k$ for each $k \in \mathbb{N}$.
        Hence, by \eqref{E:FnL+}, \eqref{E:hnL+}, and 
        \eqref{E:T:HalfRelaxedLimit:UniformBound} {\color{black}with $z=0$,}
        {\color{black}which allows us to apply the reverse Fatou lemma
         (\cite[Lemma~II.8.4, p.~97]{Rogers_Williams_2000}),}
       we have
        \begin{align*}
            - \infty
            & < u^{L, +}_{t^\ast, x^\ast} (T, x_0)
            \le \limsup_{k \to \infty}
            \biggl( \int_{t_k}^T F_{n_k}(s, \tilde{x}_k, 0) \, ds
            + h_{n_k}(\tilde{x}_k) \biggr) \\
            &{\color{black}
            \le \int_{t^\ast}^T  \limsup_{k \to \infty}  \Bigl(
           \mathbf{1}_{[t_k, T]}(s)  F_{n_k}(s,\tilde{x}_k,0) \Bigr) \, ds
            + \limsup_{k \to \infty}  h_{n_k}(\tilde{x}_k)
            }\\
            & \le \int_{t^\ast}^T \Bigl(
            \limsup_{k \to \infty} \mathbf{1}_{[t_k, T]}(s)  F^{L, +}_{t^\ast, x^\ast}(s, x_0,0) \Bigr) \, ds
            + h^{L, +}_{t^\ast, x^\ast}(x_0)
            \le h(x_0).
        \end{align*}

        Next, we deal with the interior condition.
        Let $(t_0, x_0, z_0) \in [t^\ast, T) \times \mathcal{X}^L(t^\ast, x^\ast) \times H$ and $t \in (t_0, T]$.
        By \eqref{E:unL+}, there are a sequence $((t_k, x_k))_k$ in  $[t^\ast, t] \times \mathcal{X}^L(t^\ast, x^\ast)$ and a subsequence $(u_{n_k})_k$ of $(u_n)_n$ such that
        \begin{equation*}
            (t_0, x_0)
            = \lim_{k \to\infty} (t_k, x_k),
            \quad u^{L, +}_{t^\ast, x^\ast}(t_0, x_0)
            = \lim_{k \to \infty} u_{n_k}(t_k, x_k).
        \end{equation*}
        As each $u_{n_k}$ is a minimax $L$-subsolution of \eqref{TVPn} with $n$ replaced by $n_k$, there is an $\tilde{x}_k \in \mathcal{X}^L(t_k, x_k)$ such that
        \begin{equation*}
            u_{n_k}(t_k, x_k)
            \le \int^t_{t_k} \bigl( (- f^{\tilde{x}_k}(s), z_0)
            + F_{n_k}(s, \tilde{x}_k, z_0) \bigr) \, ds
            + u_{n_k}(t, \tilde{x}_k).
        \end{equation*}
        By \cite[Proposition 2.12]{Bayraktar_Keller_2018}, we can assume that $(\tilde{x}_k)_k$ converges in $\mathcal{X}^L(t^\ast, x^\ast)$ to some $\tilde{x}_0 \in \mathcal{X}^L(t_0, x_0)$.
        Note that then also $(f^{\tilde{x}_k})_k$ converges weakly to $f^{\tilde{x}_0}$
        in $L^2(t^\ast,T;H)$ (see \cite[Lemma A.2]{Bayraktar_Keller_2018}).
        Furthermore, we can assume that $\|\tilde{x}_k - \tilde{x}_0\|_\infty \le 1 / n_k$ for each $k \in \mathbb{N}$.
        Thus, together with \eqref{E:FnL+}, \eqref{E:unL+}, \eqref{E:T:HalfRelaxedLimit:UniformBound}
        {\color{black}with $z=z_0$},
            {\color{black}thanks to which we can use the reverse Fatou lemma
            (\cite[Lemma~II.8.4, p.~97]{Rogers_Williams_2000}),}
        and \cite[Proposition 21.23 (j)]{ZeidlerIIA}, we obtain
        \begin{align*}
            - \infty
            < u^{L, +}_{t^\ast, x^\ast}(t_0, x_0)
            & \le \limsup_{k \to \infty}
            \biggl( \int^t_{t_k} \bigl( (- f^{\tilde{x}_k}(s), z_0)
            + F_{n_k}(s, \tilde{x}_k, z_0) \bigr) \, ds
            + u_{n_k}(t, \tilde{x}_k) \biggr) \\
            & \le \limsup_{k \to \infty}
            \int_{t^\ast}^t (- f^{\tilde{x}_k}(s), \mathbf{1}_{[t_k,t]}(s) z_0) \, ds \\
            & \quad + \int_{t^\ast}^t \Bigl( \limsup_{k \to \infty}
            \mathbf{1}_{[t_k, t]}(s) F^{L, +}_{t^\ast, x^\ast}(s, \tilde{x}_0, z_0) \Bigr) \, ds
            + u^{L, +}_{t^\ast, x^\ast}(t, \tilde{x}_0) \\
            & \le \int_{t_0}^t ( - f^{\tilde{x}_0}(s), z_0) \, ds
            + \int_{t_0}^t F(s, \tilde{x}_0, z_0) \, ds
            + u^{L, +}_{t^\ast, x^\ast}(t, \tilde{x}_0).
        \end{align*}
        Applying \cite[Lemma 3.6]{Bayraktar_Keller_2018}, we can conclude that the minimax $L$-subsolution interior property is satisfied.

        Note that $u^{L, +}_{t^\ast, x^\ast}$ is upper semi-continuous and $\mathbb{R}$-valued by construction and by \eqref{E:T:HalfRelaxedLimit:UniformBound}.
        Consequently, $u^{L, +}_{t^\ast, x^\ast}$ is a minimax $L$-subsolution of \eqref{TVP} on $[t^\ast, T] \times \mathcal{X}^L(t^\ast, x^\ast)$.
    \end{proof}

    \begin{proof}[Proof of Corollary \ref{C:Stability}]
        Let $(t^\ast, x^\ast) \in [0, T) \times C([0, T], H)$.
        By construction,
        \begin{align}\label{E:C:Stability:1}
            u^{L, -}_{t^\ast, x^\ast}
            \leq u^{L, +}_{t^\ast, x^\ast}.
        \end{align}
        By Theorem \ref{T:Stability}, $u^{L, +}_{t^\ast, x^\ast}$ is a viscosity $L$-subsolution of \eqref{TVP} on $[t^\ast, T] \times \mathcal{X}^L(t^\ast, x^\ast)$ and $u^{L,-}_{t^\ast,x^\ast}$ is a viscosity $L$-supersolution of \eqref{TVP} on $[t^\ast, T] \times \mathcal{X}^L(t^\ast, x^\ast)$.
        Thus, by the comparison principle (Lemma \ref{L:DoubledComparison}) together with the ``doubling theorem'' \cite[Lemma 4.3]{Bayraktar_Keller_2018} for minimax solution and the equivalence between minimax and viscosity solutions (Theorem \ref{theorem_equivalence}), we have $u^{L, +}_{t^\ast, x^\ast} \le u^{L, -}_{t^\ast, x^\ast}$, from which, together with \eqref{E:C:Stability:1}, the equality $u^{L, +}_{t^\ast, x^\ast} = u^{L, -}_{t^\ast, x^\ast}$ follows.
        With the same argument, we can also deduce that $u^{L, +}_{t^\ast, x^\ast}$ is the unique viscosity  and minimax $L$-solution of \eqref{TVP} on $[t^\ast, T] \times \mathcal{X}^L(t^\ast, x^\ast)$.
        This implies that $u^{L, +}$ defined by \eqref{E:uL+-:stab} satisfies  $u^{L, +} (t, x) = u^{L, +}_{t^\ast, x^\ast}(t, x)$ for all $(t, x) \in \mathcal{X}^L(t^\ast, x^\ast)$ (cf. the proof of Theorem \ref{theorem_minimax_existence_uniqueness}).
        Now, one can follow the lines of the end of the proof of Theorem \ref{theorem_minimax_existence_uniqueness} (and make very minor obvious adjustments) to establish the required regularity as well as the terminal and interior conditions for the restriction of $u^{L, +}$ to $[t^\ast, T] \times \mathcal{X}^L (t^\ast, x^\ast)$ to be a minimax $L$-solution on $[t^\ast, T] \times \mathcal{X}^L(t^\ast, x^\ast)$.
        Applying Theorem \ref{theorem_equivalence} again concludes the proof.
    \end{proof}

\section{Applications to differential games}

    Assume that compact topological spaces $P$ and $Q$ and functions $(f, \ell) \colon [0, T] \times C([0, T], H) \times P \times Q \to H \times \mathbb{R}$ are given and satisfy the following conditions.

    $\mathbf{H}(f, \ell)$:
    (i)
        The functions $f$ and $\ell$ are continuous.

    (ii)
        There exists a number $L_f \geq 0$ such that
        \begin{equation*}
            |f(t, x, p, q)|
            \leq L_f (1 + \|x(\cdot \wedge t)\|_\infty),
            \quad (t, x, p, q) \in [0, T] \times C([0, T], H) \times P \times Q.
        \end{equation*}

    (iii)
        For every $L \geq 0$ and every $(t_0, x_0) \in [0, T) \times C([0, T], H)$, there exists a number $\Lambda_{L, t_0, x_0} > 0$ such that, for every $t \in [t_0, T]$, every $x_1$, $x_2 \in \mathcal{X}^L(t_0,x_0)$, every $p \in P$, and every $q \in Q$,
        \begin{align*}
            & |f(t, x_1, p, q) - f(t, x_2, p, q)|
            + |\ell(t, x_1, p, q) - \ell(t, x_2, p, q)| \\
            & \quad \leq \Lambda_{L,t_0,x_0} \|x_1(\cdot \wedge t) - x_2(\cdot \wedge t)\|_\infty.
        \end{align*}

    (iv)
        The Isaacs condition holds \cite[Assumption 8.6]{Bayraktar_Keller_2018}.

    Following \cite[Section 8]{Bayraktar_Keller_2018}, consider a differential game described by the time-delay evolution equation
    \begin{subequations} \label{game}
    \begin{equation} \label{system}
        x^\prime(t) + A(t, x(t))
        = f(t, x, a(t), b(t))
        \quad \text{for a.e. } t \in [t_0, T]
    \end{equation}
    under the initial condition
    \begin{equation} \label{initial_condition}
        x(t)
        = x_0(t),
        \quad t \in [0, t_0],
    \end{equation}
    and the cost functional
    \begin{equation} \label{cost_functional}
        J(t_0, x_0, a, b)
        \coloneq \int_{t_0}^{T} \ell(t, x^{t_0, x_0, a, b}, a(t), b(t)) dt
        + h(x^{t_0, x_0, a, b}).
    \end{equation}
    \end{subequations}
    Here, $A$ and $h$ satisfy $\mathbf{H}(A)$ and $\mathbf{H}(h)$ respectively, $(t_0, x_0) \in [0, T] \times C([0, T], H)$,
    \begin{equation*}
        a \in \mathcal{A}^{t_0}
        \coloneq \bigl\{ \hat{a} \colon [t_0, T] \to P \text{ measurable} \bigr\},
        \quad b \in \mathcal{B}^{t_0}
        \coloneq \bigl\{ \hat{b} \colon [t_0, T] \to Q \text{ measurable} \bigr\}.
    \end{equation*}
    Thanks to the assumptions made, by \cite[Proposition 2.13]{Bayraktar_Keller_2018},
    for every $(t_0, x_0) \in [0, T] \times C([0,T],H)$, every $a \in \mathcal{A}^{t_0}$, and every $b \in \mathcal{B}^{t_0}$, a solution $x^{t_0, x_0, a, b}$ of \eqref{system}, \eqref{initial_condition} exists and is unique, being understood as an element of $\mathcal{X}^{L_f}(t_0, x_0)$.
    The goal of the first player (or the controller) is to minimize the value of the cost functional \eqref{cost_functional} via $a \in \mathcal{A}^{t_0}$, while the goal of the second player (or the disturbance) is to maximize this value via $b \in \mathcal{B}^{t_0}$.

    Note that the considered conditions are weaker than those imposed in \cite{Bayraktar_Keller_2018} (see the hypotheses $\mathbf{H}(f)_2$, $\mathbf{H}(\ell)_2$, and $\mathbf{H}(h)_1$ therein).
    Namely,
    (i) we do not assume that $f$ and $\ell$ are globally Lipschitz continuous and we consider the (local) Lipschitz continuity condition with respect to the supremum norm,
    (ii) we do not assume that $\ell$ satisfies a (global) sublinear growth condition,
    and (iii) we do not assume that $h$ is Lipschitz continuous (regarding $h$, we only assume continuity, i.e., $\mathbf{H}(h)$).

    Similarly to \cite[Section 1]{Bayraktar_Keller_2018}, as an illustrative example of the game \eqref{game}, one can consider a differential game for a time-delay quasi-linear parabolic partial differential equation (see, e.g., \cite[p. 878]{ZeidlerIIB})
    \begin{align*}
        & \frac{\partial x}{\partial t}(t, \xi)
        - \sum_{i = 1}^{N} D_i \bigl( |D_i x(t, \xi)|^{p - 2} D_i x(t, \xi) \bigr) \\
        & \quad = \boldsymbol{f} \bigl( t, \xi, \{x(s, \xi)\}_{s \in [0, t]}, a(t, \xi), b(t, \xi) \bigr)
        \quad \text{on } (t, \xi) \in [t_0, T] \times G
    \end{align*}
    under the conditions
    \begin{align*}
        x(t, \xi)
        & = 0
        \quad \text{on } (t, \xi) \in [t_0, T] \times \partial G, \\
        x(t, \xi)
        & = x_0(t, \xi)
        \quad \text{on }(t, \xi) \in [0, t_0] \times G
    \end{align*}
    with a cost functional
    \begin{equation*}
        \int_{t_0}^{T} \boldsymbol{\ell} \bigl( t, \{x(s, \cdot)\}_{s \in [0, t]}, a(t, \cdot), b(t, \cdot) \bigr) dt
        + \boldsymbol{h} \bigl( \{x(t, \cdot)\}_{t \in [0, T]} \bigr).
    \end{equation*}
    Here, $G$ is a bounded region in $\mathbb{R}^N$, $N \in \mathbb{N}$, $\xi \coloneq (\xi_1, \ldots, \xi_N)$, $D_i \coloneq \partial / \partial \xi_i$, $i \in \overline{1, N}$, and $\boldsymbol{f}$, $\boldsymbol{\ell}$, and $\boldsymbol{h}$ are appropriate functions.

    Below, based on the results obtained in the paper, the question of the existence of the value of the game \eqref{game} is investigated.
    The classes of non-anticipative players' strategies are first considered, and then (in a certain sense narrower) classes of feedback players' strategies are considered.

\subsection{Non-anticipating strategies}

    In addition to the equivalence with minimax solutions (see Theorem \ref{theorem_equivalence}), another motivation for Definition \ref{definition_viscosity} of a viscosity solution of \eqref{TVP} is that the scheme for proving the existence of the value of a differential game in the classes of non-anticipating strategies of the players developed for the classical finite-dimensional case (see, for example, \cite{Evans_Souganidis_1984} and also \cite{Fleming_Soner_2006,Yong_2015}) is transferred practically without change to the path-dependent infinite-dimensional case under consideration.

    For every $(t_0, x_0) \in [0, T] \times C([0,T],H)$, define the lower and upper game values
    \begin{align*}
        v_-(t_0, x_0)
        & \coloneq \inf_{\mathfrak{a} \in \mathfrak{A}^{t_0}} \sup_{b \in \mathcal{B}^{t_0}}
        J(t_0, x_0, \mathfrak{a}[b], b), \\
        v_+(t_0, x_0)
        & \coloneq \sup_{\mathfrak{b} \in \mathfrak{B}^{t_0}} \inf_{a \in \mathcal{A}^{t_0}}
        J(t_0, x_0, a, \mathfrak{b}[a]),
    \end{align*}
    where
    \begin{align*}
        \mathfrak{A}^{t_0}
        & \coloneq \bigl\{ \mathfrak{a} \colon \mathcal{B}^{t_0} \to \mathcal{A}^{t_0}
        \text{ non-anticipating} \bigr\}, \\
        \mathfrak{B}^{t_0}
        & \coloneq \bigl\{ \mathfrak{b} \colon \mathcal{A}^{t_0} \to \mathcal{B}^{t_0}
        \text{ non-anticipating} \bigr\}.
    \end{align*}
    Recall that a mapping $\mathfrak{a} \colon \mathcal{B}^{t_0} \to \mathcal{A}^{t_0}$ is called non-anticipating if, for every $t_\ast \in [t_0, T]$ and every $b_1$, $b_2 \in \mathcal{B}^{t_0}$ satisfying $b_1(t) = b_2(t)$ for a.e. $t \in [t_0, t_\ast]$, it holds that $\mathfrak{a}[b_1](t) = \mathfrak{a}[b_2](t)$ for a.e. $t \in [t_0, t_\ast]$;
    non-anticipating mappings $\mathfrak{b} \colon \mathcal{A}^{t_0} \to \mathcal{B}^{t_0}$ are defined similarly.

    If $v_- = v_+$, we say that the differential game \eqref{game} has value $v_0 \coloneq v_- = v_+$ in the classes of non-anticipating strategies.

    Consider the functions (usually called the lower and upper Hamiltonians of the game \eqref{game})
    \begin{align*}
        F_-(t, x, z)
        & \coloneq \max_{q \in Q} \min_{p \in P} \bigl( \ell(t, x, p, q)
        + (f(t, x, p, q), z) \bigr), \\
        F_+(t, x, z)
        & \coloneq \min_{p \in P} \max_{q \in Q} \bigl( \ell(t, x, p, q)
        + (f(t, x, p, q), z) \bigr)
    \end{align*}
    for all $(t, x, z) \in [0, T] \times C([0, T], H) \times H$.
    Note that $F_-$ and $F_+$ satisfy the hypotheses $\mathbf{H}(F)$ since $(f, \ell)$ satisfies the hypotheses $\mathbf{H}(f, \ell)$(i)--(iii);
    moreover, $\mathbf{H}(F)$(ii) is satisfied with $L_0 \coloneq L_f$\footnote{In general case, $F_-$ and $F_+$ do not satisfy the condition $\mathbf{H}(F)$(iii) from \cite[Section 3]{Bayraktar_Keller_2018} (see also Remark \ref{remark_conditions_comparison}).
    For corresponding examples, see \cite[p. S1103]{Gomoyunov_Lukoyanov_Plaksin_2021}.}.

    \begin{theorem} \label{theorem_value_viscosity}
        Let $\mathbf{H}(A)$, $\mathbf{H}(f, \ell)${\rm (i)--(iii)}, and $\mathbf{H}(h)$ be satisfied.
        Then,

        {\rm (i)}
            the lower value function $v_-$ coincides with a unique viscosity {\rm(}or, equivalently, minimax{\rm)} $L_f$-solution of \eqref{TVP} with $F \coloneq F_-$;

        {\rm (ii)}
            the upper value function $v_+$ coincides with a unique viscosity {\rm(}or, equivalently, minimax{\rm)} $L_f$-solution of \eqref{TVP} with $F \coloneq F_+$;

        {\rm (iii)}
            under the additional assumption $\mathbf{H}(f, \ell)${\rm(iv)}, the differential game \eqref{game} has value in the classes of non-anticipating strategies and $v_0 = u$, where $u$ is a unique viscosity {\rm(}or, equivalently, minimax{\rm)} solution of \eqref{TVP} with
            \begin{equation} \label{E:Hamiltonian:DifferentialGame}
                F(t, x, z)
                \coloneq F_+(t, x, z)
                = F_-(t, x, z),
                \quad (t, x, z) \in [0, T] \times C([0, T], H) \times H.
            \end{equation}
    \end{theorem}

    Below, we will prove item (ii) only, since the proof for (i) is similar and (iii) follows from (i) and (ii) and the uniqueness result for viscosity solutions (see Theorem \ref{theorem_viscosity_existence_uniqueness}).
    In the proof, we will use the following two facts.

    \begin{proposition} \label{proposition_v_+_DPP}
        The function $v_+$ satisfies the dynamic programming principle:
        \begin{equation} \label{DPP}
            v_+(t_0, x_0)
            = \sup_{\mathfrak{b} \in \mathfrak{B}^{t_0}} \inf_{a \in \mathcal{A}^{t_0}}
            \biggl( \int_{t_0}^{t} \ell (s, x^{t_0, x_0, a, \mathfrak{b}[a]}, a(s), \mathfrak{b}[a](s)) \, ds
            + v_+(t, x^{t_0, x_0, a, \mathfrak{b}[a]}) \biggr)
        \end{equation}
        for all $(t_0, x_0) \in [0, T] \times C([0,T],H)$ and $t \in [t_0, T]$.
    \end{proposition}

    \begin{proposition} \label{proposition_v_+_continuous}
        The restrictions of the function $v_+$ to the sets $[t_0, T] \times \mathcal{X}^L(t_0, x_0)$,
        $t_0 \in [0, T)$, $x_0 \in C([0, T], H)$, $L \geq 0$, are continuous.
    \end{proposition}

    The proofs of Propositions \ref{proposition_v_+_DPP} and \ref{proposition_v_+_continuous} are rather standard.
    For the reader's convenience, they are given in \ref{S:proof_proposition_v_+_DPP} and \ref{S:proof_proposition_v_+_continuous}, respectively.

    In addition, we will use the following result, which is a generalization of \cite[Lemma XI.6.1]{Fleming_Soner_2006}, where $P$ and $Q$ are compact subsets of finite-dimensional spaces.
    \begin{proposition} \label{proposition_measurable_selection}
        Let $h \colon P \times Q \to \mathbb{R}$ be a continuous function.
        Then, for every $\varepsilon > 0$, there exists a Borel measurable function $q_\varepsilon \colon P \to Q$ such that
        \begin{equation*}
            h(p, q_\varepsilon(p))
            \geq \max_{q \in Q} h(p, q) - \varepsilon,
            \quad p \in P.
        \end{equation*}
    \end{proposition}

    The proof is given in \ref{S:proof_proposition_measurable_selection}.

    \begin{proof}[Proof of Theorem \ref{theorem_value_viscosity}]
        By Theorem \ref{theorem_viscosity_existence_uniqueness}, we only need to prove that $v_+$ is a viscosity $L_f$-solution of \eqref{TVP} with $F \coloneq F_+$.
        The proof follows a rather standard scheme (see, e.g., \cite[Theorem 3.3.6]{Yong_2015}).

        Note that the restrictions of $v_+$ to $[t^\ast, T] \times \mathcal{X}^{L^\ast}(t^\ast, x^\ast)$, $t^\ast \in [0, T)$, $x^\ast \in C([0,T], H)$, $L^\ast \geq 0$, are continuous by Proposition \ref{proposition_v_+_continuous} and the boundary condition $v_+(x) = h(x)$ for
        all $x \in C([0,T],H)$ is satisfied by construction.
        Hence, it remains to verify the inferior conditions (a) for viscosity $L_f$-subsolutions and (b) for viscosity $L_f$-supersolutions (see Definition \ref{definition_viscosity}).

        (a)
        Fix $(t_0, x_0) \in [0, T) \times C([0, T], H)$.
        Let $(\varphi, z) \in \underline{\mathcal{A}}^{L_f} v_+(t_0, x_0)$ with the corresponding number $T_0 \in (t_0, T]$ be given.
        For all $\mathfrak{b} \in \mathfrak{B}^{t_0}$ and $a \in \mathcal{A}^{t_0}$, recalling that $x^{t_0, x_0, a, \mathfrak{b}[a]} \in \mathcal{X}^{L_f}(t_0, x_0)$, we have
        \begin{align*}
            & \varphi(t_0, x_0) + \tilde{\varphi}^z(t_0, x_0) - v_+(t_0, x_0) \\*
            & \quad \leq \varphi(t_0 + \delta, x^{t_0, x_0, a, \mathfrak{b}[a]})
            + \tilde{\varphi}^z(t_0 + \delta, x^{t_0, x_0, a, \mathfrak{b}[a]})
            - v_+(t_0 + \delta, x^{t_0, x_0, a, \mathfrak{b}[a]})
        \end{align*}
        for all $\delta \in (0, T_0 - t_0]$, and, hence,
        \begin{align*}
            & v_+(t_0 + \delta, x^{t_0, x_0, a, \mathfrak{b}[a]}) - v_+(t_0, x_0) \\*
            & \quad \leq \varphi(t_0 + \delta, x^{t_0, x_0, a, \mathfrak{b}[a]})
            - \varphi(t_0, x_0)
            + \tilde{\varphi}^z(t_0 + \delta, x^{t_0, x_0, a, \mathfrak{b}[a]})
            - \tilde{\varphi}^z(t_0, x_0)
        \end{align*}
        for all $\delta \in (0, T_0 - t_0]$.
        Consequently, using \eqref{DPP}, we derive
        \begin{align*}
            0
            & = \sup_{\mathfrak{b} \in \mathfrak{B}^{t_0}} \inf_{a \in \mathcal{A}^{t_0}}
            \biggl( \int_{t_0}^{t_0 + \delta} \ell (s, x^{t_0, x_0, a, \mathfrak{b}[a]}, a(s), \mathfrak{b}[a](s)) \, d s \\*
            & \quad + v_+(t_0 + \delta, x^{t_0, x_0, a, \mathfrak{b}[a]})
            - v_+(t_0, x_0) \biggr) \\
            & \leq \sup_{\mathfrak{b} \in \mathfrak{B}^{t_0}} \inf_{a \in \mathcal{A}^{t_0}}
            \biggl( \int_{t_0}^{t_0 + \delta} \ell (s, x^{t_0, x_0, a, \mathfrak{b}[a]}, a(s), \mathfrak{b}[a](s)) \, d s \\*
            & \quad + \varphi(t_0 + \delta, x^{t_0, x_0, a, \mathfrak{b}[a]})
            - \varphi(t_0, x_0)
            + \tilde{\varphi}^z(t_0 + \delta, x^{t_0, x_0, a, \mathfrak{b}[a]})
            - \tilde{\varphi}^z(t_0, x_0) \biggr)
        \end{align*}
        for all $\delta \in (0, T_0 - t_0]$.
        Applying the functional chain rule to $\varphi$, recalling the dynamic equation and the definition of $\tilde{\varphi}^z$, and dividing by $\delta$, we obtain
        \begin{align*}
            0 & \leq \sup_{\mathfrak{b} \in \mathfrak{B}^{t_0}} \inf_{a \in \mathcal{A}^{t_0}}
            \frac{1}{\delta} \int_{t_0}^{t_0 + \delta}
            \Bigl( h(s, x^{t_0, x_0, a, \mathfrak{b}[a]}(s), a(s), \mathfrak{b}[a](s)) \\
            & \quad
            + \langle A(s, x^{t_0, x_0, a, \mathfrak{b}[a]}(s)),
            z - \partial_x \varphi(s, x^{t_0, x_0, a, \mathfrak{b}[a]}) \rangle \Bigr) \, d s,
            \quad \delta \in (0, T_0 - t_0],
        \end{align*}
        where we denote
        \begin{equation} \label{h}
            h(t, x, p, q)
            \coloneq \partial_t \varphi(t, x)
            + \ell (t, x, p, q)
            + (f(t, x, p, q), \partial_x \varphi(t, x)).
        \end{equation}
        Let $\kappa > 0$ be fixed.
        By continuity of $h$, there exists a number $\delta_\ast \in (0, T_0 - t_0]$ such that, for every
        $t \in [t_0, t_0 + \delta_\ast]$, every $x \in \mathcal{X}^{L_f}(t_0, x_0)$, every $p \in P$, and every $q \in Q$,
        \begin{equation*}
            |h(t, x, p, q) - h(t_0, x_0, p, q)|
            \leq \kappa.
        \end{equation*}
        Hence,
        \begin{align*}
            - \kappa
            & \leq \sup_{\mathfrak{b} \in \mathfrak{B}^{t_0}} \inf_{a \in \mathcal{A}^{t_0}}
            \frac{1}{\delta} \int_{t_0}^{t_0 + \delta}
            \Bigl( h(t_0, x_0, a(s), \mathfrak{b}[a](s)) \\*
            & \quad + \langle A(s, x^{t_0, x_0, a, \mathfrak{b}[a]}(s)),
            z -  \partial_x \varphi(s, x^{t_0, x_0, a, \mathfrak{b}[a]}) \rangle \Bigr) \, d s
        \end{align*}
        for all $\delta \in (0, \delta_\ast]$.
        Choose $p_0 \in P$ from the condition
        \begin{equation*}
            \max_{q \in Q} h(t_0, x_0, p_0, q)
            = \min_{p \in P} \max_{q \in Q} h(t_0, x_0, p, q)
        \end{equation*}
        and denote $a_0(t) \coloneq p_0$ for all $t \in [t_0, T]$.
        Then, for every $\delta \in (0, \delta_\ast]$, there exists $\mathfrak{b}_\delta \in \mathfrak{B}^{t_0}$ such that
        \begin{align*}
            - 2 \kappa
            & \leq \frac{1}{\delta} \int_{t_0}^{t_0 + \delta}
            h(t_0, x_0, p_0, \mathfrak{b}_\delta[a_0](s)) \, d s \\*
            & \quad + \frac{1}{\delta} \int_{t_0}^{t_0 + \delta}
            \langle A(s, x^{t_0, x_0, a_0, \mathfrak{b}_\delta[a_0]}(s)),
            z - \partial_x \varphi(s, x^{t_0, x_0, a_0, \mathfrak{b}_\delta[a_0]}) \rangle \, d s \\
            & \leq \partial_t \varphi(t_0, x_0) + F(t_0, x_0, \partial_x \varphi(t_0, x_0)) \\*
            & \quad + \sup_{x \in \mathcal{X}^{L_f}(t_0, x_0)}
            \frac{1}{\delta} \int_{t_0}^{t_0 + \delta} \langle A(s, x(s)),
            z - \partial_x \varphi(s, x) \rangle \, d s.
        \end{align*}
        As a result, recalling the definition of $\partial_t \tilde{\varphi}^z$, we get
        \begin{align*}
            - 2 \kappa
            & \leq \partial_t \varphi(t_0, x_0) + F(t_0, x_0, \partial_x \varphi(t_0, x_0)) \\
            & \quad + \limsup_{\delta \to 0^+} \sup_{x \in \mathcal{X}^{L_f}(t_0, x_0)}
            \frac{1}{\delta} \int_{t_0}^{t_0 + \delta}
            \bigl( \partial_t \tilde{\varphi}^z(s, x) - \langle A(s, x(s)),
            \partial_x \varphi(s, x) \rangle \bigr) \, d s.
        \end{align*}
        Since $\kappa > 0$ is arbitrary, we conclude that $v_+$ is a viscosity $L_f$-subsolution.

        (b)
        Fix $(t_0, x_0) \in [0, T) \times C([0, T], H)$.
        Let $(\varphi, z) \in \overline{\mathcal{A}}^{L_f} v_+(t_0, x_0)$ with the corresponding number $T_0 \in (t_0, T]$ be given.
        Arguing as above, for a fixed $\kappa > 0$, we can choose $\delta_\ast \in (0, \varepsilon]$ such that, for every $\delta \in (0, \delta_\ast]$,
        \begin{align*}
            \kappa
            & \geq \sup_{\mathfrak{b} \in \mathfrak{B}^{t_0}} \inf_{a \in \mathcal{A}^{t_0}}
            \frac{1}{\delta} \int_{t_0}^{t_0 + \delta}
            \Bigl( h(t_0, x_0, a(s), \mathfrak{b}[a](s)) \\
            & \quad + \langle A(s, x^{t_0, x_0, a, \mathfrak{b}[a]}(s)),
            z - \partial_x \varphi(s, x^{t_0, x_0, a, \mathfrak{b}[a]}) \rangle \Bigr) \, d s,
        \end{align*}
        where the notation from \eqref{h} is used.
        Since $h$ is continuous and $P$, $Q$ are compact, there exists (see Proposition \ref{proposition_measurable_selection}) a Borel measurable function $q_0 \colon P \to Q$ such that
        \begin{equation*}
            h(t_0, x_0, p, q_0(p))
            \geq \max_{q \in Q} h(t_0, x_0, p, q) - \kappa,
            \quad p \in P.
        \end{equation*}
        For every $a \in \mathcal{A}^{t_0}$, put $\mathfrak{b}_0[a](s) \coloneq q_0(a(s))$ for all $s \in [t_0, T]$.
        Note that $\mathfrak{b}_0[a] \in \mathcal{B}^{t_0}$ since $q_0$ is Borel measurable and $a$ is Lebesgue measurable.
        In addition, $\mathfrak{b}_0$ is non-anticipating, which gives $\mathfrak{b}_0 \in \mathfrak{B}^{t_0}$.
        Hence, we derive
        \begin{align*}
            2 \kappa
            & \geq \inf_{a \in \mathcal{A}^{t_0}}
            \frac{1}{\delta} \int_{t_0}^{t_0 + \delta}
            \Bigl( h(t_0, x_0, a(s), \mathfrak{b}_0[a](s)) \\*
            & \quad + \langle A(s, x^{t_0, x_0, a, \mathfrak{b}_0[a]}(s)),
            z - \partial_x \varphi(s, x^{t_0, x_0, a, \mathfrak{b}_0[a]}) \rangle \Bigr) \, d s + \kappa \\
            & \geq \inf_{a \in \mathcal{A}^{t_0}}
            \frac{1}{\delta} \int_{t_0}^{t_0 + \delta}
            \Bigl( \max_{q \in Q} h(t_0, x_0, a(s), q) \\*
            & \quad + \langle A(s, x^{t_0, x_0, a, \mathfrak{b}_0[a]}(s)),
            z - \partial_x \varphi(s, x^{t_0, x_0, a, \mathfrak{b}_0[a]}) \rangle \Bigr) \, d s \\
            & \geq \inf_{a \in \mathcal{A}^{t_0}}
            \frac{1}{\delta} \int_{t_0}^{t_0 + \delta}
            \max_{q \in Q} h(t_0, x_0, a(s), q) \, d s \\
            & \quad + \inf_{x \in \mathcal{X}^{L_f}(t_0, x_0)}
            \frac{1}{\delta} \int_{t_0}^{t_0 + \delta}
            \langle A(s, x(s)),
            z - \partial_x \varphi(s, x) \rangle \, d s
        \end{align*}
        for all $\delta \in (0, \delta_\ast]$.
        Note that, for every $a \in \mathcal{A}^{t_0}$, we have
        \begin{align*}
            & \frac{1}{\delta} \int_{t_0}^{t_0 + \delta}
            \max_{q \in Q} h(t_0, x_0, a(s), q) \, d s
            \geq \frac{1}{\delta} \int_{t_0}^{t_0 + \delta}
            \min_{p \in Q} \max_{q \in Q} h(t_0, x_0, p, q) \, d s \\
            & \quad = \frac{1}{\delta} \int_{t_0}^{t_0 + \delta}
            \bigl( \partial_t \varphi(t_0, x_0)
            + F(t_0, x_0, \partial_x \varphi(t_0, x_0)) \bigr) \, d s \\
            & \quad = \partial_t \varphi(t_0, x_0)
            + F(t_0, x_0, \partial_x \varphi(t_0, x_0))
        \end{align*}
        for all $\delta \in (0, \delta_\ast]$ and, therefore,
        \begin{equation*}
            \inf_{a \in \mathcal{A}^{t_0}}
            \frac{1}{\delta} \int_{t_0}^{t_0 + \delta}
            \max_{q \in Q} h(t_0, x_0, a(s), q) \, d s
            \geq \partial_t \varphi(t_0, x_0)
            + F(t_0, x_0, \partial_x \varphi(t_0, x_0)),
            \quad \delta \in (0, \delta_\ast].
        \end{equation*}
        As a result, we get
        \begin{align*}
            2 \kappa
            & \geq \partial_t \varphi(t_0, x_0)
            + F(t_0, x_0, \partial_x \varphi(t_0, x_0)) \\
            & \quad + \inf_{x \in \mathcal{X}^{L_f}(t_0, x_0)}
            \frac{1}{\delta} \int_{t_0}^{t_0 + \delta}
            \langle A(s, x(s)),
            z - \partial_x \varphi(s, x) \rangle \, d s
        \end{align*}
        for all $\delta \in (0, \delta_\ast]$.
        Hence, recalling the definition of $\partial_t \tilde{\varphi}^z$, we deduce
        \begin{align*}
            2 \kappa
            & \geq \partial_t \varphi(t_0, x_0)
            + F(t_0, x_0, \partial_x \varphi(t_0, x_0)) \\
            & \quad + \liminf_{\delta \to 0^+} \inf_{x \in \mathcal{X}^{L_f}(t_0, x_0)}
            \frac{1}{\delta} \int_{t_0}^{t_0 + \delta}
            \Bigl( \partial_t \tilde{\varphi}^z(s, x)
            - \langle A(s, x(s)), \partial_x \varphi(s, x) \rangle \Bigr) \, d s.
        \end{align*}
        Thus, $v_+$ is a viscosity $L_f$-supersolution.
        The proof is complete.
    \end{proof}

\subsection{Feedback strategies}

    Fix $t_0 \in [0, T]$.
    The spaces of feedback strategies of the players are defined as follows:
    \begin{align*}
        \mathbb{A}^{t_0}
        & \coloneq \bigl\{ \mathbf{a} \colon [t_0, T] \times C([0,T],H) \to P \text{ non-anticipating} \bigr\}, \\
        \mathbb{B}^{t_0}
        & \coloneq \bigl\{ \mathbf{b} \colon [t_0, T] \times C([0,T],H) \to Q \text{ non-anticipating} \bigr\}.
    \end{align*}

    For every $\delta > 0$, every partition $\pi \colon t_0 < t_1 < \cdots < t_n = T$ with mesh size $|\pi| \leq \delta$, every $x_0 \in C([0, T], H)$, every feedback strategy $\mathbf{a} \in \mathbb{A}^{t_0}$, and every control $b \in \mathcal{B}^{t_0}$, define a step-by-step feedback control $a \coloneq a^{\pi, x_0, \mathbf{a}, b} \in \mathcal{A}^{t_0}$ recursively by
    \begin{equation*}
        a(t)
        \coloneq a(t_i, x^{t_0, x_0, a, b}),
        \quad t \in [t_i, t_{t + 1}),
        \quad i = 0, \ldots, n - 1.
    \end{equation*}
    The guaranteed result of the first player given a strategy $\mathbf{a} \in \mathbb{A}^{t_0}$ and a state $x_0 \in C([0, T], H)$ is defined by
    \begin{equation*}
        J_a(t_0, x_0; \mathbf{a})
        \coloneq \inf_{\delta > 0} \sup_{|\pi| \leq \delta, \, b \in \mathcal{B}^{t_0}}
        J(t_0, x_0, a^{\pi, x_0, \mathbf{a}, b}, b).
    \end{equation*}
    The optimal guaranteed result of the first player $v_a \colon [0, T] \times C([0, T], H) \to \mathbb{R}$ is defined by
    \begin{equation*}
        v_a(t_0, x_0)
        \coloneq \inf_{\mathbf{a} \in \mathbb{A}^{t_0}}
        J_a(t_0, x_0; \mathbf{a}),
        \quad (t_0, x_0) \in [0, T] \times C([0, T], H).
    \end{equation*}

    The optimal guaranteed result of the second player $v_b \colon [0, T] \times C([0, T], H) \to \mathbb{R}$ is defined similarly with clear changes (see \cite[Section 8]{Bayraktar_Keller_2018} for details).

    If $v_a = v_b$, it is said that the differential game \eqref{game} has value $v \coloneq v_a = v_b$ in the classes of feedback strategies.

    The theorem below is a generalization of \cite[Theorem 8.7]{Bayraktar_Keller_2018}.
    \begin{theorem} \label{theorem_feedback_strategies}
        Let $\mathbf{H}(A)$, $\mathbf{H}(f, \ell)$, and $\mathbf{H}(h)$ be satisfied.
        Then, the differential game \eqref{game} has value in the classes of feedback strategies and $v = u$, where $u$ is the unique minimax {\rm(}or, equivalently, viscosity{\rm)} solution of \eqref{TVP} with $F$ given by \eqref{E:Hamiltonian:DifferentialGame}.
    \end{theorem}

    It follows from Theorems \ref{theorem_value_viscosity} and \ref{theorem_feedback_strategies} that, under the conditions $\mathbf{H}(A)$, $\mathbf{H}(f, \ell)$, and $\mathbf{H}(h)$, for every $(t_0, x_0) \in [0, T] \times C([0, T], H)$,
    \begin{equation*}
        v_-(t_0, x_0)
        = v_a(t_0, x_0),
        \quad v_+(t_0, x_0)
        = v_b(t_0, x_0).
    \end{equation*}
    These equalities mean that the first (resp., second) player can ensure achievement of the lower  (resp., upper) game value with the help of feedback strategies that do not involve the direct use of information about the control actions of the opponent.

    \begin{proof}[Proof of Theorem \ref{theorem_feedback_strategies}]
        The proof of this result repeats that of \cite[Theorem 8.7]{Bayraktar_Keller_2018}.
        The only change that needs to be made is to use a different auxiliary function $\nu_\varepsilon$ (in this connection, see also the proof of \cite[Theorem 5.3]{Gomoyunov_Lukoyanov_2024}).

        By Theorem \ref{theorem_minimax_existence_uniqueness}, the function $u$ is well-defined.
        Since $v_b \leq v_a$ (see, e.g., \cite[Proposition 8.5]{Bayraktar_Keller_2018}), we only need to show that $v_a \leq u \leq v_b$.
        Fix $(t_0, x_0) \in [0, T] \times C([0, T], H)$.
        If $t_0 = T$, we have $v_a(t_0, x_0) = u(t_0, x_0) = v_b(t_0, x_0) = h(x_0)$ for all $x_0 \in C([0, T], H)$, and, therefore, we assume further that $t_0 < T$.
        We prove only $v_a(t_0, x_0) \leq u(t_0, x_0)$.
        To this end, we specify a family $(\mathbf{a}^\varepsilon)_{\varepsilon}$ of strategies in $\mathbb{A}^{t_0}$ to show that $J_a(t_0, x_0; \mathbf{a}^\varepsilon) \leq u(t_0, x_0) + m_a(\varepsilon)$ for all $\varepsilon > 0$ and some modulus of continuity $m_a$.

        Recall $L_f$ is from $\mathbf{H}(f, \ell)$(ii) and choose $\Lambda_L \coloneq \Lambda_{L_f, t_0, x_0}$ according to $\mathbf{H}(f, \ell)$(iii).
        Further, take $\varkappa$ from \eqref{varkappa} and denote $\varepsilon_0 \coloneq e^{- 2 \Lambda_L T / \varkappa} / (2 \sqrt{\varkappa})$.
        For every $\varepsilon \in (0, \varepsilon_0]$ and every $(t, x) \in [0, T] \times C([0, T], H)$, put
        \begin{equation*}
            \alpha^\varepsilon(t)
            \coloneq \frac{e^{- 2 \Lambda_L t / \varkappa} - \varepsilon \sqrt{\varkappa}}{\varepsilon},
            \quad \beta^\varepsilon(t, x)
            \coloneq \sqrt{\varepsilon^4 + \Upsilon(t, x)},
            \quad \nu^\varepsilon(t, x)
            \coloneq \alpha^\varepsilon(t) \beta^\varepsilon(t, x),
        \end{equation*}
        where $\Upsilon$ is from \eqref{Upsilon}.
        By Proposition \ref{proposition_nuEpsilon_c11V}, for every $\varepsilon > 0$, we have $\nu^\varepsilon \in \mathcal{C}^{1, 1}_V([0, T] \times C([0, T], H))$ and, for every $t \in [0, T)$ and every $x \in C([0, T], H)$,
        \begin{equation*}
            \partial_t \nu^\varepsilon(t, x)
            = - \frac{2 \Lambda_L e^{- \Lambda_L t / \varkappa}}{\varkappa \varepsilon}
            \beta^\varepsilon(t, x),
            \quad \partial_x \nu^\varepsilon(t, x)
            = \frac{\alpha^\varepsilon(t)}{2 \beta^\varepsilon(t, x)} \theta(t, x, 0) x(t),
        \end{equation*}
        where $\theta(t, x, 0)$ is given by \eqref{theta}.

        Let $\varepsilon > 0$.
        Define the function $u_a^\varepsilon \colon [t_0, T] \times \mathcal{X}^{L_f}(t_0, x_0) \to \mathbb{R}$ by
        \begin{equation*}
            u_a^\varepsilon(t, x)
            \coloneq \min_{\tilde{x} \in \mathcal{X}^{L_f}(t_0, x_0)}
            \bigl( u(t, \tilde{x}) + \nu^\varepsilon(t, x - \tilde{x}) \bigr),
            \quad t \in [t_0, T], \ x \in \mathcal{X}^{L_f}(t_0, x_0).
        \end{equation*}
        For every $(t, x) \in [t_0, T] \times \mathcal{X}^{L_f}(t_0, x_0)$,
        choose a path $x_a^{\varepsilon, t, x}$ for which the minimum is attained and
        an $\mathbf{a}^\varepsilon(t, x) \in P$ such that
        \begin{equation} \label{E:def:a_epsilon}
            \begin{aligned}
                & \max_{q \in Q} \bigl( \ell (t, x, \mathbf{a}^\varepsilon(t, x), q)
                + (f(t, x, \mathbf{a}^\varepsilon(t, x), q),
                \partial_x \nu^\varepsilon (t, x - x_a^{\varepsilon, t, x})) \bigr) \\
                & \quad = \min_{p \in P} \max_{q \in Q} \bigl( \ell (t, x, p, q)
                + (f(t, x, p, q),
                \partial_x \nu^\varepsilon (t, x - x_a^{\varepsilon, t, x})) \bigr).
            \end{aligned}
        \end{equation}
        We can do this in such a way that $\mathbf{a}^\varepsilon$ is non-anticipating (this is possible since $f$, $\ell$, $u$, $\nu^\varepsilon$, and $\partial_x \nu^\varepsilon$ are non-anticipative).
        Thus, we have $\mathbf{a}^\varepsilon \in \mathbb{A}^{t_0}$.

        Let us prove that there exists a modulus of continuity $m$ such that, for every number $\delta > 0$, every partition $\pi$ with $|\pi| \leq \delta$, every control $b \in \mathcal{B}^{t_0}$, and every $i = 0, \ldots, n - 1$,
        \begin{equation}\label{E:FeedbackGame:DesiredInequality}
            \int_{t_i}^{t_{i + 1}} \ell(t, x^{\pi, b}, \mathbf{a}^\varepsilon(t_i, x^{\pi, b}), b(t)) \, dt
            + u_a^\varepsilon(t_{i + 1}, x^{\pi, b}) - u_a^\varepsilon(t_i, x^{\pi, b})
            \leq m(\delta) (t_{i + 1} - t_i),
        \end{equation}
        where $x^{\pi, b} \coloneq x^{t_0, x_0, a^{\pi, x_0, \mathbf{a}^\varepsilon, b}, b}$.
        Denote $x^\varepsilon_{a, i} \coloneq x_a^{\varepsilon, t_i, x^{\pi, b}}$.
        Note that $u$ is a minimax $L_f$-solution of \eqref{TVP} (see Theorem \ref{theorem_minimax_existence_uniqueness} and recall that $F$ from \eqref{E:Hamiltonian:DifferentialGame} satisfies $\mathbf{H}(F)$(ii) with $L_0 \coloneq L_f$).
        Hence, there exists $\tilde{x} \in \mathcal{X}^{L_f}(t_i, x^\varepsilon_{a, i})$ such that
        \begin{align*}
            & u(t_{i + 1}, \tilde{x}) - u(t_i, x^\varepsilon_{a, i}) \\
            & \quad = \int_{t_i}^{t_{i + 1}} \bigl( (f^{\tilde{x}}(s),
            \partial_x \nu^\varepsilon (t_i, x^{\pi, b} - x_{a, i}^{\varepsilon}))
            - F(s, \tilde{x}, \partial_x \nu^\varepsilon (t_i, x^{\pi, b} - x_{a, i}^{\varepsilon}))
            \bigr) \, ds.
        \end{align*}
        Then, according to the definition of $u_a^\varepsilon$ and the functional chain rule, we have
        \begin{align*}
            & u_a^\varepsilon(t_{i + 1}, x^{\pi, b}) - u_a^\varepsilon(t_i, x^{\pi, b}) \\
            & \quad \leq u(t_{i + 1}, \tilde{x})
            + \nu^\varepsilon (t_{i + 1}, x^{\pi, b} - \tilde{x})
            - u(t_i, x^\varepsilon_{a, i})
            - \nu^\varepsilon (t_i, x^{\pi, b} - x^\varepsilon_{a, i}) \\
            & \quad = \int_{t_i}^{t_{i + 1}} \bigl( (f^{\tilde{x}}(s),
            \partial_x \nu^\varepsilon (t_i, x^{\pi, b} - x_{a, i}^{\varepsilon}))
            - F(s, \tilde{x}, \partial_x \nu^\varepsilon (t_i, x^{\pi, b} - x_{a, i}^{\varepsilon}))
            \bigr) \, ds \\
            & \qquad + \nu^\varepsilon (t_{i + 1}, x^{\pi, b} - \tilde{x})
            - \nu^\varepsilon (t_i, x^{\pi, b} - x^\varepsilon_{a, i}) \\
            & \quad = \int_{t_i}^{t_{i + 1}} \bigl( (f^{\tilde{x}}(s),
            \partial_x \nu^\varepsilon (t_i, x^{\pi, b} - x_{a, i}^{\varepsilon}))
            - F(s, \tilde{x}, \partial_x \nu^\varepsilon (t_i, x^{\pi, b} - x_{a, i}^{\varepsilon}))
            \bigr) \, ds \\
            & \qquad + \int_{t_i}^{t_{i + 1}}
            \bigl( \partial_t \nu^\varepsilon (s, x^{\pi, b} - \tilde{x})
            + \langle - A(s, x^{\pi, b}(s)) + A(s, \tilde{x}(s)),
            \partial_x \nu^\varepsilon (s, x^{\pi, b} - \tilde{x}) \rangle \\
            & \qquad + (f(s, x^{\pi, b}, \mathbf{a}^\varepsilon(t_i, x^{\pi, b}), b(s))
            - f^{\tilde{x}}(s),
            \partial_x \nu^\varepsilon (s, x^{\pi, b} - \tilde{x})) \bigr) \, ds \\
            & \quad = \int_{t_i}^{t_{i + 1}} \bigl(
            \partial_t \nu^\varepsilon (s, x^{\pi, b} - \tilde{x})
            + (f(s, x^{\pi, b}, \mathbf{a}^\varepsilon(t_i, x^{\pi, b}), b(s)),
            \partial_x \nu^\varepsilon (s, x^{\pi, b} - \tilde{x})) \\
            & \qquad - F(s, \tilde{x}, \partial_x \nu^\varepsilon (t_i, x^{\pi, b} - x_{a, i}^{\varepsilon}))
            + (f^{\tilde{x}}(s),
            \partial_x \nu^\varepsilon (t_i, x^{\pi, b} - x_{a, i}^{\varepsilon})
            - \partial_x \nu^\varepsilon (s, x^{\pi, b} - \tilde{x})) \\
            & \qquad + \langle - A(s, x^{\pi, b}(s)) + A(s, \tilde{x}(s)),
            \partial_x \nu^\varepsilon (s, x^{\pi, b} - \tilde{x}) \rangle
            \bigr) \, ds.
        \end{align*}
        Note that, by the monotonicity of $A$ and the formula for $\partial_x \nu^\varepsilon (s, x^{\pi, b} - \tilde{x})$ (recall the inclusion $\partial_x \nu^\varepsilon (s, x^{\pi, b} - \tilde{x}) \in V$ for a.e. $s$ and that $\theta$ is non-negative),
        \begin{align*}
            & \langle - A(s, x^{\pi, b}(s)) + A(s, \tilde{x}(s)),
            \partial_x \nu^\varepsilon (s, x^{\pi, b} - \tilde{x}) \rangle \\
            & \quad = \biggl\langle - A(s, x^{\pi, b}(s)) + A(s, \tilde{x}(s)),
            \frac{\alpha^\varepsilon(s)}{2 \beta^\varepsilon(s, x^{\pi, b} - \tilde{x})}
            \theta(t, x^{\pi, b} - \tilde{x}, 0) (x^{\pi, b}(s) - \tilde{x}(s)) \biggr\rangle \\
            & \quad = \frac{\alpha^\varepsilon(s)}{2 \beta^\varepsilon(s, x^{\pi, b} - \tilde{x})}
            \theta(t, x^{\pi, b} - \tilde{x}, 0)
            \biggl\langle - A(s, x^{\pi, b}(s)) + A(s, \tilde{x}(s)),
            x^{\pi, b}(s) - \tilde{x}(s) \biggr\rangle \\
            & \quad \leq 0.
        \end{align*}
        Due to compactness of $\mathcal{X}^{L_f}(t_0, x_0)$ and continuity of $\partial_t \nu^\varepsilon$, $\partial_x \nu^\varepsilon$, $f$, and $F$, there exists a modulus $m$ such that (recall the definition of $\mathbf{a}^\varepsilon(t_i, x^{\pi, b})$
        in \eqref{E:def:a_epsilon} and the definition of $F$ in \eqref{E:Hamiltonian:DifferentialGame})
        \begin{align*}
            & \int_{t_i}^{t_{i + 1}} \ell(t, x^{\pi, b}, \mathbf{a}^\varepsilon(t_i, x^{\pi, b}), b(t)) \, dt
            + u_a^\varepsilon(t_{i + 1}, x^{\pi, b}) - u_a^\varepsilon(t_i, x^{\pi, b}) \\
            & \quad \leq \bigl( \partial_t \nu^\varepsilon (t_i, x^{\pi, b} - x_{a, i}^{\varepsilon})
            + F(t_i, x^{\pi, b}, \partial_x \nu^\varepsilon (t_i, x^{\pi, b} - x_{a, i}^{\varepsilon})) \\
            & \qquad - F(t_i, x_{a, i}^{\varepsilon}, \partial_x \nu^\varepsilon (t_i, x^{\pi, b} - x_{a, i}^{\varepsilon})) \bigr) (t_{i + 1} - t_i)
            + m(\delta) (t_{i + 1} - t_i).
        \end{align*}
        We have (recall the choice of $\Lambda_L$ and the formulas for the derivatives of $\nu^\varepsilon$)
        \begin{align*}
            & \partial_t \nu^\varepsilon (t_i, x^{\pi, b} - x_{a, i}^{\varepsilon})
            + F(t_i, x^{\pi, b}, \partial_x \nu^\varepsilon (t_i, x^{\pi, b} - x_{a, i}^{\varepsilon}))
            - F(t_i, x_{a, i}^{\varepsilon}, \partial_x \nu^\varepsilon (t_i, x^{\pi, b} - x_{a, i}^{\varepsilon})) \\
            & \leq \partial_t \nu^\varepsilon (t_i, x^{\pi, b} - x_{a, i}^{\varepsilon})
            + \Lambda_L (1 + |\partial_x \nu^\varepsilon (t_i, x^{\pi, b} - x_{a, i}^{\varepsilon})|)
            \|x^{\pi, b}(\cdot \wedge t_i) - x_{a, i}^{\varepsilon}(\cdot \wedge t_i)\|_\infty \\
            & \leq - \frac{2 \Lambda_L e^{- \Lambda_L t_i / \varkappa}}{\varkappa \varepsilon}
            \beta^\varepsilon(t_i, x^{\pi, b} - x_{a, i}^{\varepsilon}) \\
            & \quad + \Lambda_L \biggl( 1 + \frac{\alpha^\varepsilon(t_i) \theta(t_i, x^{\pi, b} - x_{a, i}^{\varepsilon}, 0)}{2 \beta^\varepsilon(t_i, x^{\pi, b} - x_{a, i}^{\varepsilon})}
            |x^{\pi, b}(t_i) - x_{a, i}^{\varepsilon}(t_i)| \biggr)
            \|x^{\pi, b}(\cdot \wedge t_i) - x_{a, i}^{\varepsilon}(\cdot \wedge t_i)\|_\infty.
        \end{align*}
        Note that $\theta(t_i, x^{\pi, b} - x_{a, i}^{\varepsilon}, 0) \leq 4$ and
        \begin{align*}
            |x^{\pi, b}(t_i) - x_{a, i}^{\varepsilon}(t_i)|
            & \leq \|x^{\pi, b}(\cdot \wedge t_i) - x_{a, i}^{\varepsilon}(\cdot \wedge t_i)\|_\infty \\
            & \leq \sqrt{\frac{\Upsilon(t_i, x^{\pi, b} - x_{a, i}^{\varepsilon})}{\varkappa}}
            \leq \frac{\beta^\varepsilon(t_i, x^{\pi, b} - x_{a, i}^{\varepsilon})}{\sqrt{\varkappa}}.
        \end{align*}
        Hence (recall the definition of $\alpha^\varepsilon$),
        \begin{align*}
            & - \frac{2 \Lambda_L e^{- \Lambda_L t_i / \varkappa}}{\varkappa \varepsilon}
            \beta^\varepsilon(t_i, x^{\pi, b} - x_{a, i}^{\varepsilon}) \\
            & + \Lambda_L \biggl( 1 + \frac{\alpha^\varepsilon(t_i) \theta(t_i, x^{\pi, b} - x_{a, i}^{\varepsilon}, 0)}{2 \beta^\varepsilon(t_i, x^{\pi, b} - x_{a, i}^{\varepsilon})}
            |x^{\pi, b}(t_i) - x_{a, i}^{\varepsilon}(t_i)| \biggr)
            \|x^{\pi, b}(\cdot \wedge t_i) - x_{a, i}^{\varepsilon}(\cdot \wedge t_i)\|_\infty \\
            & \quad \leq - \frac{2 \Lambda_L e^{- \Lambda_L t_i / \varkappa}}{\varkappa \varepsilon}
            \beta^\varepsilon(t_i, x^{\pi, b} - x_{a, i}^{\varepsilon})
            + \Lambda_L \biggl( 1 + \frac{2 \alpha^\varepsilon(t_i)}{\sqrt{\varkappa}} \biggr)
            \frac{\beta^\varepsilon(t_i, x^{\pi, b} - x_{a, i}^{\varepsilon})}{\sqrt{\varkappa}} \\
            & \quad = \frac{\Lambda_L}{\varkappa \varepsilon}
            \beta^\varepsilon(t_i, x^{\pi, b} - x_{a, i}^{\varepsilon})
            \biggl( - 2 e^{- \Lambda_L t_i / \varkappa} + \sqrt{\varkappa} \varepsilon
            + 2 \varepsilon \alpha^\varepsilon(t_i) \biggr) \\
            & \quad = \frac{\Lambda_L}{\varkappa \varepsilon}
            \beta^\varepsilon(t_i, x^{\pi, b} - x_{a, i}^{\varepsilon})
            \biggl( - 2 e^{- \Lambda_L t_i / \varkappa} + \sqrt{\varkappa} \varepsilon
            + 2 \varepsilon \frac{e^{- 2 \Lambda_L t_i / \varkappa} - \varepsilon \sqrt{\varkappa}}{\varepsilon} \biggr) \\
            & \quad = \frac{\Lambda_L}{\varkappa \varepsilon}
            \beta^\varepsilon(t_i, x^{\pi, b} - x_{a, i}^{\varepsilon})
            \biggl( - 2 e^{- \Lambda_L t_i / \varkappa} + \sqrt{\varkappa} \varepsilon
            + 2 e^{- 2 \Lambda_L t_i / \varkappa} - 2 \varepsilon \sqrt{\varkappa} \biggr) \\
            & \quad = \frac{\Lambda_L}{\varkappa \varepsilon}
            \beta^\varepsilon(t_i, x^{\pi, b} - x_{a, i}^{\varepsilon})
            \biggl(  \sqrt{\varkappa} \varepsilon - 2 \varepsilon \sqrt{\varkappa} \biggr)
            = - \frac{\Lambda_L}{\sqrt{\varkappa}}
            \beta^\varepsilon(t_i, x^{\pi, b} - x_{a, i}^{\varepsilon})
            < 0.
        \end{align*}
        As a result, we obtain the desired inequality \eqref{E:FeedbackGame:DesiredInequality}.

        From this inequality, we derive that, for every number $\delta > 0$, every partition $\pi$ with $|\pi| \leq \delta$, and every control $b \in \mathcal{B}^{t_0}$,
        \begin{equation*}
            \int_{t_0}^{T} \ell(t, x^{\pi, b}, a^{\pi, x_0, \mathbf{a}^\varepsilon, b}, b(t)) \, dt
            + u_a^\varepsilon(T, x^{\pi, b})
            \leq u_a^\varepsilon(t_0, x_0) + m(\delta) (T - t_0).
        \end{equation*}

        Further, note that, for every $\tilde{x} \in \mathcal{X}^{L_f}(t_0, x_0)$, we have
        \begin{align*}
            u_a^\varepsilon(t_0, x_0)
            & \leq u(t_0, x_0) + \nu^\varepsilon(t_0, x_0 - \tilde{x})
            = u(t_0, x_0) + \nu^\varepsilon(t_0, x_0 - x_0) \\
            & = u(t_0, x_0) + \frac{e^{- 2 \Lambda_L t_0 / \varkappa} - \varepsilon \sqrt{\varkappa}}{\varepsilon} \varepsilon^2
            \leq u(t_0, x_0) + \varepsilon.
        \end{align*}
        On the other hand, we have
        \begin{align*}
            u_a^\varepsilon(T, x^{\pi, b})
            & = u(T, x_a^{\varepsilon, T, x^{\pi, b}})
            + \nu^\varepsilon(T, x^{\pi, b} - x_a^{\varepsilon, T, x^{\pi, b}}) \\
            & = h(x_a^{\varepsilon, T, x^{\pi, b}})
            + \nu^\varepsilon(T, x^{\pi, b} - x_a^{\varepsilon, T, x^{\pi, b}})
        \end{align*}
        and
        \begin{equation*}
            u_a^\varepsilon(T, x^{\pi, b})
            \leq
            u(T, x^{\pi, b})
            + \nu^\varepsilon(T, x^{\pi, b} - x^{\pi, b})
            \leq h(x^{\pi, b}) + \varepsilon.
        \end{equation*}
        Therefore,
        \begin{equation*}
            \nu^\varepsilon(T, x^{\pi, b} - x_a^{\varepsilon, T, x^{\pi, b}})
            = u_a^\varepsilon(T, x^{\pi, b}) - h(x_a^{\varepsilon, T, x^{\pi, b}})
            \leq h(x^{\pi, b}) - h(x_a^{\varepsilon, T, x^{\pi, b}}) + \varepsilon
            \leq C_1 + \varepsilon
        \end{equation*}
        with $C_1 \coloneq 2 \max_{x \in \mathcal{X}^{L_f}(t_0, x_0)} |h(x)|$.
        By the definition of $\nu^\varepsilon$, we get
        \begin{equation*}
            \frac{e^{- 2 \Lambda_L T / \varkappa} - \varepsilon \sqrt{\varkappa}}{\varepsilon}
            \sqrt{\varepsilon^4 + \Upsilon(T, x^{\pi, b} - x_a^{\varepsilon, T, x^{\pi, b}})}
            \leq C_1 + \varepsilon,
        \end{equation*}
        which is equivalent to
        \begin{equation*}
            \sqrt{\varepsilon^4 + \Upsilon(T, x^{\pi, b} - x_a^{\varepsilon, T, x^{\pi, b}})}
            \leq \frac{\varepsilon (C_1 + \varepsilon)}{e^{- 2 \Lambda_L T / \varkappa} - \varepsilon \sqrt{\varkappa}}.
        \end{equation*}
        Hence,
        \begin{align*}
            & \|x^{\pi, b} - x_a^{\varepsilon, T, x^{\pi, b}}\|_\infty
            \leq \sqrt{\frac{\Upsilon(T, x^{\pi, b} - x_a^{\varepsilon, T, x^{\pi, b}})}{\varkappa}}
            \leq \frac{\varepsilon (C_1 + \varepsilon)}{e^{- 2 \Lambda_L T / \varkappa} - \varepsilon \sqrt{\varkappa}} \frac{1}{\sqrt{\varkappa}} \\
            & \quad \leq \frac{\varepsilon (C_1 + \varepsilon_0)}{e^{- 2 \Lambda_L T / \varkappa} - \varepsilon_0 \sqrt{\varkappa}} \frac{1}{\sqrt{\varkappa}}
            = \frac{\varepsilon (C_1 + e^{- 2 \Lambda_L T / \varkappa} / (2 \sqrt{\varkappa}))}{e^{- 2 \Lambda_L T / \varkappa} - e^{- 2 \Lambda_L T / \varkappa} / (2 \sqrt{\varkappa}) \sqrt{\varkappa}} \frac{1}{\sqrt{\varkappa}} \\
            & \quad = \frac{2 \varepsilon (C_1 + e^{- 2 \Lambda_L T / \varkappa} / (2 \sqrt{\varkappa}))}{e^{- 2 \Lambda_L T / \varkappa}} \frac{1}{\sqrt{\varkappa}}
            = \varepsilon (2 C_1 e^{2 \Lambda_L T / \varkappa} +  1 / \sqrt{\varkappa}) \frac{1}{\sqrt{\varkappa}} \\
            & \quad = \varepsilon \frac{2 C_1 \sqrt{\varkappa} e^{2 \Lambda_L T / \varkappa} + 1}{\varkappa}
            = C_2 \varepsilon
        \end{align*}
        with $C_2 \coloneq (2 C_1 \sqrt{\varkappa} e^{2 \Lambda_L T / \varkappa} + 1) / \varkappa$.
        As a result, we have
        \begin{align*}
            u_a^\varepsilon(T, x^{\pi, b})
            & = h(x_a^{\varepsilon, T, x^{\pi, b}})
            + \nu^\varepsilon(T, x^{\pi, b} - x_a^{\varepsilon, T, x^{\pi, b}}) \\
            & \geq h(x_a^{\varepsilon, T, x^{\pi, b}})
            \geq h(x^{\pi, b})
            - m_h(\|x^{\pi, b} - x_a^{\varepsilon, T, x^{\pi, b}}\|_\infty)
            \geq h(x^{\pi, b}) - m_h(C_2 \varepsilon),
        \end{align*}
        where $m_h$ is the modulus of continuity of $h|_{\mathcal{X}^{L_f}(t_0, x_0)}$.

        Summarizing, we conclude that, for every number $\delta > 0$, every partition $\pi$ with $|\pi| \leq \delta$, and every control $b \in \mathcal{B}^{t_0}$,
        \begin{align*}
            J(t_0, x_0, a^{\pi, x_0, \mathbf{a}^\varepsilon, b}, b)
            & = \int_{t_0}^{T} \ell(t, x^{\pi, b}, a^{\pi, x_0, \mathbf{a}^\varepsilon, b}, b(t)) \, dt
            + h(x^{\pi, b}) \\
            & \leq \int_{t_0}^{T} \ell(t, x^{\pi, b}, a^{\pi, x_0, \mathbf{a}^\varepsilon, b}, b(t)) \, dt
            + u_a^\varepsilon(T, x^{\pi, b}) + m_h(C_2 \varepsilon) \\
            & \leq u_a^\varepsilon(t_0, x_0) + m(\delta) (T - t_0) + m_h(C_2 \varepsilon) \\
            & \leq u(t_0, x_0) + m(\delta) (T - t_0) + m_h(C_2 \varepsilon) + \varepsilon.
        \end{align*}
        Thus,
        \begin{equation*}
            J_a(t_0, x_0; \mathbf{a}^\varepsilon)
            = \inf_{\delta > 0} \sup_{|\pi| \leq \delta, \, b \in \mathcal{B}^{t_0}}
            J(t_0, x_0, a^{\pi, x_0, \mathbf{a}^\varepsilon, b}, b)
            \leq u(t_0, x_0) + m_h(C_2 \varepsilon) + \varepsilon,
        \end{equation*}
        and the result follows.
    \end{proof}

\appendix

\section{Additional proofs}

    \subsection{Proof of Proposition \ref{P:Minimax:Infinitesimal}.}
    \label{S:proof_P:Minimax:Infinitesimal}

        We prove part (ii) only, since the proof for (i) is similar.

        Suppose that $u$ is a minimax $L$-supersolution of \eqref{TVP} and fix $(t_0, x_0) \in [0, T) \times C([0,T],H)$ and $z \in H$.
        Then, by the definition, there exists $x \in \mathcal{X}^L(t_0, x_0)$ such that
        \begin{equation*}
            u(t_0, x_0)
            \geq \int_{t_0}^{t} \bigl( (- f^x(s), z) + F(s, x, z) \bigr) \, ds + u(t, x),
            \quad t \in [t_0, T].
        \end{equation*}

        Fix $\varepsilon > 0$.
        {\color{black}Note} that  as $F$ is continuous  {\color{black}with respect to  \eqref{pseudometric}} by $\mathbf{H}(F)$(i){\color{black}, it is also
        continuous with respect to the metric}
        \begin{align*}
        {\color{black}
        ((t_1,x_1),(t_2,x_2))\mapsto|t_1-t_2|+\|x_1-x_2\|_\infty,\quad \left([0,T]\times C([0,T],H)\right)^2 \to
        \mathbb{R}.}
        \end{align*}
        {\color{black}Thus, together with the compactness of} $[t_0, T] \times \mathcal{X}^L(t_0, x_0)$,
        {\color{black} we obtain a}
        $\delta_\ast \in (0, T - t_0)$ such that, for every $s_1$, $s_2 \in [t_0, T]$ with $|s_1 - s_2| \leq \delta_\ast$ and every $x \in \mathcal{X}^L(t_0, x_0)$,
        \begin{equation*}
            | F(s_1, x, z) - F(s_2, x, z) |
            \leq \varepsilon.
        \end{equation*}
        {\color{black}
        Moreover, continuity of $F$ with respect to  \eqref{pseudometric} implies
        \begin{equation*}
            F(t_0, x, z)
            = F(t_0, x(\cdot \wedge t_0), z)
            = F(t_0, x_0, z)
        \end{equation*} for all $x \in \mathcal{X}^L (t_0, x_0)$ (see Definition~\ref{definitionXL}).}
        Hence, for every $s \in [t_0, t_0 + \delta_\ast]$ and every $x \in \mathcal{X}^L (t_0, x_0)$, we have
        \begin{equation*}
            | F(s, x, z) - F(t_0, x_0, z) |
            = | F(s, x, z) - F(t_0, x, z) |
            \leq \varepsilon.
        \end{equation*}
        As a result, for every $\delta \in (0, \delta_\ast]$ and every $t \in [t_0, t_0 + \delta]$, we derive
        \begin{align*}
            0
            & \geq \frac{1}{t - t_0} \biggl( u(t, x) - u(t_0, x_0)
            + \int_{t_0}^{t} \bigl( (- f^x(s), z) + F(s, x, z) \bigr) \, ds \biggr) \\
            & \geq \frac{1}{t - t_0} \biggl( u(t, x) - u(t_0, x_0)
            + \int_{t_0}^{t} (- f^x(s), z) \, ds \biggr)
            + F(t_0, x_0, z) - \varepsilon \\
            & \geq \inf \biggl\{
            \frac{1}{t - t_0}
            \biggl( u(t, x) - u(t_0, x_0)
            + \int_{t_0}^{t} (- f^x(s), z) \, ds \biggr) \colon \\
            & \qquad \qquad
            t \in (t_0, t_0 + \delta], \, x \in \mathcal{X}^L(t_0, x_0) \biggr\}
            + F(t_0, x_0, z) - \varepsilon.
        \end{align*}
        So, we obtain \eqref{supersolution_infinitesimal}.

        On the other hand, for the sake of a contradiction, we assume that $u$ is not a minimax $L$-supersolution.
        By \cite[Lemma 3.6]{Bayraktar_Keller_2018}, there exist $(t_0, x_0, z) \in [0, T) \times
        C([0,T], H) \times H$ and $t_1 \in [t_0, T]$ such that, for every $x \in \mathcal{X}^L(t_0, x_0)$,
        \begin{equation*}
            u(t_1, x) - u(t_0, x_0)
            + \int_{t_0}^{t_1} \bigl( (- f^x(s), z) + F(s, x, z) \bigr) \, ds
            > 0.
        \end{equation*}
        Note that $t_1 > t_0$.
        By lower semi-continuity of the restriction of $u$ on $[t_0, T] \times \mathcal{X}^L(t_0, x_0)$, continuity of $F$, and compactness of $\mathcal{X}^L(t_0, x_0)$, there exists a number $\tilde{\delta} > 0$ such that, for every $x \in \mathcal{X}^L(t_0, x_0)$,
        \begin{equation*}
            u(t_1, x) - u(t_0, x_0)
            + \int_{t_0}^{t_1} \bigl( (- f^x(s), z) + F(s, x, z) \bigr) \, ds
            > \tilde{\delta}.
        \end{equation*}
        Denote by $\tilde{t}$ the supremum of all $t \in [t_0, t_1]$ such that
        \begin{equation*}
            \min_{x \in \mathcal{X}^L(t_0, x_0)} \biggl( u(t, x) - u(t_0, x_0)
            + \int_{t_0}^{t} \bigl( (- f^x(s), z) + F(s, x, z) \bigr) \, ds \biggr)
            \leq \frac{t - t_0}{t_1 - t_0} \tilde{\delta}.
        \end{equation*}
        We have $\tilde{t} < t_1$ and, by lower semi-continuity of the restriction of $u$ on $[t_0, T] \times \mathcal{X}^L(t_0, x_0)$, continuity of $F$, and compactness of $\mathcal{X}^L(t_0, x_0)$, we derive that $\tilde{t}$ is the maximum.
        In particular, there exists a function $\tilde{x} \in \mathcal{X}^L(t_0, x_0)$ such that
        \begin{equation*}
            u(\tilde{t}, \tilde{x}) - u(t_0, x_0)
            + \int_{t_0}^{\tilde{t}} \bigl( (- f^{\tilde{x}}(s), z) + F(s, \tilde{x}, z) \bigr) \, ds
            \leq \frac{\tilde{t} - t_0}{t_1 - t_0} \tilde{\delta}.
        \end{equation*}
        By the assumption (see \eqref{supersolution_infinitesimal}), there exist a number $t \in (\tilde{t}, t_1]$ and a function $x \in \mathcal{X}^L(\tilde{t}, \tilde{x})$ such that
        \begin{equation*}
            \frac{1}{t - \tilde{t}}
            \biggl( u(t, x) - u(\tilde{t}, \tilde{x})
            + \int_{\tilde{t}}^{t} (- f^x(s), z) \, ds \biggr)
            + F(\tilde{t}, \tilde{x}, z)
            \leq \frac{\tilde{\delta}}{2 (t_1 - t_0)}.
        \end{equation*}
        Arguing as in the first part of the proof, we can assume that $t - \tilde{t}$ is small enough in order to have
        \begin{equation*}
            \frac{1}{t - \tilde{t}}
            \biggl( u(t, x) - u(\tilde{t}, \tilde{x})
            + \int_{\tilde{t}}^{t} \bigl( (- f^x(s), z) + F(s, x, z) \bigr) \, ds \biggr)
            \leq \frac{\tilde{\delta}}{t_1 - t_0}.
        \end{equation*}
        Then, we get
        \begin{equation*}
            u(t, x) - u(t_0, x_0)
            + \int_{t_0}^{t} \bigl( (- f^x(s), z) + F(s, x, z) \bigr) \, ds
            \leq \frac{\tilde{t} - t_0}{t_1 - t_0} \tilde{\delta}
            + \frac{t - \tilde{t}}{t_1 - t_0} \tilde{\delta}
            = \frac{t - t_0}{t_1 - t_0} \tilde{\delta},
        \end{equation*}
        a contradiction.
        The proof is complete.

    \subsection{Proof of Proposition \ref{proposition_v_+_DPP}}
    \label{S:proof_proposition_v_+_DPP}

        Denote the right-hand side of \eqref{DPP} by $\hat{v}$.
        For every $\varepsilon > 0$, take $\mathfrak{b} \in \mathfrak{B}^{t_0}$ such that
        \begin{equation*}
            \hat{v}
            \leq \inf_{a \in \mathcal{A}^{t_0}}
            \biggl( v_+(\theta, x^{t_0, x_0, a, \mathfrak{b}[a]})
            + \int_{t_0}^{\theta} \ell(s, x^{t_0, x_0, a, \mathfrak{b}[a]}, a(s), \mathfrak{b}[a](s)) \, d s \biggr) + \varepsilon.
        \end{equation*}
        For every $a \in \mathcal{A}^{t_0}$, choose $\bar{\mathfrak{b}}^{x^{t_0, x_0, a, \mathfrak{b}[a]}} \in \mathfrak{B}^\theta$ from the condition
        \begin{equation*}
            v_+(\theta, x^{t_0, x_0, a, \mathfrak{b}[a]})
            \leq \inf_{\bar{a} \in \mathcal{A}^\theta}
            J (\theta, x^{t_0, x_0, a, \mathfrak{b}[a]}, \bar{a},
            \bar{\mathfrak{b}}^{x^{t_0, x_0, a, \mathfrak{b}[a]}}[\bar{a}]) + \varepsilon.
        \end{equation*}
        Consider a function $\tilde{\mathfrak{b}} \colon \mathcal{A}^{t_0} \to \mathcal{B}^{t_0}$ that maps $a \in \mathcal{A}^{t_0}$ to
        \begin{equation*}
            b(s)
            \coloneq \begin{cases}
                \mathfrak{b}[a](s) & \text{if } s \in [t_0, \theta), \\
                \bar{\mathfrak{b}}^{x^{t_0, x_0, a, \mathfrak{b}[a]}}[a|_{[\theta, T]}](s) & \text{if } s \in [\theta, T].
              \end{cases}
        \end{equation*}
        It can be directly verified that $\tilde{\mathfrak{b}}$ is non-anticipating, i.e., $\tilde{\mathfrak{b}} \in \mathfrak{B}^{t_0}$.
        Hence,
        \begin{equation*}
            v_+(t_0, x_0)
            \geq \inf_{a \in \mathcal{A}^{t_0}}
            J(t_0, x_0, a, \tilde{\mathfrak{b}}[a]).
        \end{equation*}
        For every $a \in \mathcal{A}^{t_0}$, we get
        \begin{align*}
            & J(t_0, x_0, a, \tilde{\mathfrak{b}}[a]) \\
            & \quad = J (\theta, x^{t_0, x_0, a, \mathfrak{b}[a]}, a|_{[\theta, T]}, \bar{\mathfrak{b}}^{x^{t_0, x_0, a, \mathfrak{b}[a]}}[a|_{[\theta, T]}]) \\
            & \qquad + \int_{t_0}^{\theta} \ell(s, x^{t_0, x_0, a, \mathfrak{b}[a]}(s), a(s), \mathfrak{b}[a](s)) \, d s \\
            & \quad \geq v_+(\theta, x^{t_0, x_0, a, \mathfrak{b}[a]})
            + \int_{t_0}^{\theta} \ell(s, x^{t_0, x_0, a, \mathfrak{b}[a]}(s), a(s), \mathfrak{b}[a](s)) \, d s
            - \varepsilon \\
            & \quad \geq \hat{v} - 2 \varepsilon.
        \end{align*}
        As a result, $v_+(t_0, x_0) \geq \hat{v} - 2 \varepsilon$, which gives $v_+(t_0, x_0) \geq \hat{v}$ since $\varepsilon > 0$ is arbitrary.

        It remains to show that $\hat{v} \geq v_+(t_0, x_0)$.
        For every $\varepsilon > 0$, take $\mathfrak{b}^\ast \in \mathfrak{B}^{t_0}$ such that
        \begin{equation*}
            v_+(t_0, x_0)
            \leq \inf_{a \in \mathcal{A}^{t_0}}
            J(t_0, x_0, a, \mathfrak{b}^\ast[a]) + \varepsilon
        \end{equation*}
        and then choose $a^\ast \in \mathcal{A}^{t_0}$ from the condition
        \begin{equation*}
            \hat{v}
            \geq v_+(\theta, x^{t_0, x_0, a^\ast, \mathfrak{b}^\ast[a^\ast]})
            + \int_{t_0}^{\theta} \ell(s, x^{t_0, x_0, a^\ast, \mathfrak{b}^\ast[a^\ast]}(s),
            a^\ast(s), \mathfrak{b}^\ast[a^\ast](s)) \, d s
            - \varepsilon.
        \end{equation*}
        Consider a mapping $\bar{\mathfrak{b}}^\ast \colon \mathcal{A}^\theta \to \mathcal{B}^\theta$ that maps $\bar{a} \in \mathcal{A}^\theta$ to $\bar{b} \coloneq \mathfrak{b}^\ast[a]|_{[\theta, T]}$, where
        \begin{equation*}
            a(s)
            \coloneq
            \begin{cases}
                a^\ast(s) & \text{if } s \in [t, \theta), \\
                \bar{a}(s) & \text{if } s \in [\theta, T].
            \end{cases}
        \end{equation*}
        It can be verified that $\bar{\mathfrak{b}}^\ast$ is non-anticipative, i.e., $\bar{\mathfrak{b}}^\ast \in \mathfrak{B}^{\theta}$.
        Hence, there is an $\bar{a}^\ast \in \mathcal{A}^\theta$ such that
        \begin{equation*}
            v_+(\theta, x^{t_0, x_0, a^\ast, \mathfrak{b}^\ast[a^\ast]})
            \geq J(\theta,  x^{t_0, x_0, a^\ast, \mathfrak{b}^\ast[a^\ast]}, \bar{a}^\ast, \bar{\mathfrak{b}}^\ast[\bar{a}^\ast])
            - \varepsilon.
        \end{equation*}
        Consequently, we obtain
        \begin{align*}
            \hat{v}
            & \geq J(\theta,x^{t_0, x_0, a^\ast, \mathfrak{b}^\ast[a^\ast]}, \bar{a}^\ast, \bar{\mathfrak{b}}^\ast[\bar{a}^\ast])
            + \int_{t_0}^{\theta} \ell(s, x^{t_0, x_0, a^\ast, \mathfrak{b}^\ast[a^\ast]}(s), a^\ast(s), \mathfrak{b}^\ast[a^\ast](s) ) \, d s
            - 2 \varepsilon \\
            & = J(t_0, x_0, \tilde{a}, \tilde{b}) - 2 \varepsilon,
        \end{align*}
        where
        \begin{equation*}
            \tilde{a}(s)
            \coloneq \begin{cases}
                a^\ast(s) & \text{if } s \in [t_0, \theta), \\
                \bar{a}^\ast(s) & \text{if } s \in [\theta, T],
            \end{cases}
            \quad \tilde{b}(s)
            \coloneq \begin{cases}
                \mathfrak{b}^\ast[a^\ast](s) & \text{if } s \in [t_0, \theta), \\
                \bar{\mathfrak{b}}^\ast[\bar{a}^\ast](s) & \text{if } s \in [\theta, T].
            \end{cases}
        \end{equation*}
        Observing that $\tilde{b}(s) = \mathfrak{b}^\ast[\tilde{a}](s)$ for a.e. $s \in [t_0, T]$ by construction, we derive
        \begin{equation*}
            J(t_0, x_0, \tilde{a}, \tilde{b})
            = J(t_0, x_0, \tilde{a}, \mathfrak{b}^\ast[\tilde{a}])
            \geq v_+(t_0, x_0) - \varepsilon.
        \end{equation*}
        Thus, we have $\hat{v} \geq v_+(t_0, x_0) - 3 \varepsilon$, which yields $\hat{v} \geq v_+(t_0, x_0)$ since $\varepsilon > 0$ is arbitrary.
        The proof is complete.

    \subsection{Proof of Proposition \ref{proposition_v_+_continuous}}
    \label{S:proof_proposition_v_+_continuous}

        Fix $t^\ast \in [0, T)$, $x^\ast \in C([0, T], H)$, and $L \geq 0$.
        We will show that the restriction of $v_+$ to $[t^\ast, T] \times \mathcal{X}^L(t^\ast, x^\ast)$ is continuous.
        Without loss of generality, we can assume that $L \geq L_f$.
        Take a number $\Lambda_L \coloneq \Lambda_{L, t^\ast, x^\ast} > 0$ from $\mathbf{H}(f)$(iii) and, using compactness of $\mathcal{X}^L(t^\ast, x^\ast)$, choose a number $C > 0$ such that $\|x\|_\infty \leq C$ for all $x \in \mathcal{X}^L(t^\ast, x^\ast)$.
        Put $C_1 \coloneq e^{\Lambda_L T} \max\{1, 4 L(1 + C)\}$.
        Then, arguing similarly to the proof of \cite[Proposition 7.2]{Bayraktar_Keller_2018}, it can be shown that, for every $t_0$, $t_1 \in [t^\ast, T]$ with $t_1 \geq t_0$, every $x_0$,  $x_1 \in \mathcal{X}^L(t^\ast, x^\ast)$, every $a \in \mathcal{A}^{t_0}$, and every $b \in \mathcal{B}^{t_0}$,
        \begin{equation*}
            \|x^{t_0, x_0, a, b} - x^{t_1, x_1, a|_{[t_1, T]}, b|_{[t_1, T]}}\|_\infty
            \leq C_1 \bigl( t_1 - t_0 + \|x_0(\cdot \wedge t_0) - x_1(\cdot \wedge t_0)\|_\infty \bigr).
        \end{equation*}
        Since $\ell|_{[t^\ast, T] \times \mathcal{X}^L(t^\ast,x^\ast) \times P \times Q}$ is bounded by $\mathbf{H}(\ell)$(i) and the compactness of $\mathcal{X}^L(t^\ast, x^\ast)$, there exists a number $C_2 > 0$ such that
        \begin{equation*}
            |\ell(t, x, p, q)|
            \leq C_2
        \end{equation*}
        for all $t \in [t^\ast, T]$, $p \in P$, $q \in Q$.
        Denote by $m_h$ the modulus of continuity of $h|_{\mathcal{X}^L(t^\ast, x^\ast)}$, which exists by $\mathbf{H}(h)$ and the compactness of $\mathcal{X}^L(t^\ast, x^\ast)$.

        Let $\delta > 0$ and consider $(t_0, x_0)$, $(t_1, x_1) \in \mathcal{X}^L(t^\ast,x^\ast)$  with $t_1 \geq t_0$ and such that $\mathbf{d}_\infty((t_0, x_0), (t_1, x_1)) \leq \delta$.
        Then, for every $a \in \mathcal{A}^{t_0}$ and every $b \in \mathcal{B}^{t_0}$, we have
        \begin{align*}
            & | J(t_0, x_0, a, b) - J(t_1, x_1, a, b) | \\
            & \quad \leq \int_{t_0}^{t_1} | \ell(s, x^{t_0, x_0, a, b}, a(s), b(s)) | \, ds \\
            & \qquad + \int_{t_1}^{T} | \ell(s, x^{t_0, x_0, a, b}, a(s), b(s))
            - \ell(s, x^{t_1, x_1, a|_{[t_1, T]}, b|_{[t_1, T]}}, a(s), b(s)) | \, ds \\
            & \qquad + | h(x^{t_0, x_0, a, b}) - h(x^{t_1, x_1, a|_{[t_1, T]}, b|_{[t_1, T]}}) | \\
            & \quad \leq C_2 (t_1 - t_0)
            + T \Lambda_{L} \|x^{t_0, x_0, a, b} - x^{t_1, x_1, a|_{[t_1, T]}, b|_{[t_1, T]}}\|_\infty \\
            & \qquad + m_h(  \|x^{t_0, x_0, a, b} - x^{t_1, x_1, a|_{[t_1, T]}, b|_{[t_1, T]}}\|_\infty ) \\
            & \quad \leq C_2 \delta
            + T \Lambda_{L} C_1 \delta
            + m_h(C_1 \delta).
        \end{align*}

        Fix $\varepsilon > 0$ and choose $\delta > 0$ from the condition $C_2 \delta + T \Lambda_{L_f} C_1 \delta + m_h(C_1 \delta) \leq \varepsilon / 3$.
        Take arbitrarily $(t_0, x_0)$, $(t_1, x_1) \in \mathcal{X}^L(t^\ast,x^\ast)$
        with $t_1 \geq t_0$ and $\mathbf{d}_\infty((t_0, x_0), (t_1, x_1)) \leq \delta$.

        First, consider the difference
        \begin{equation*}
            v_+(t_0, x_0) - v_+(t_1, x_1)
            = \sup_{\mathfrak{b} \in \mathfrak{B}^{t_0}} \inf_{a \in \mathcal{A}^{t_0}}
            J(t_0, x_0, a, \mathfrak{b}[a])
            - \sup_{\mathfrak{b} \in \mathfrak{B}^{t_1}} \inf_{a \in \mathcal{A}^{t_1}}
            J(t_1, x_1, a, \mathfrak{b}[a]).
        \end{equation*}
        Take $\mathfrak{b}_\varepsilon^0 \in \mathfrak{B}^{t_0}$ such that
        \begin{equation*}
            \sup_{\mathfrak{b} \in \mathfrak{B}^{t_0}} \inf_{a \in \mathcal{A}^{t_0}}
            J(t_0, x_0, a, \mathfrak{b}[a])
            \leq \inf_{a \in \mathcal{A}^{t_0}}
            J(t_0, x_0, a, \mathfrak{b}_\varepsilon^0[a]) + \varepsilon / 3.
        \end{equation*}
        Define $\mathfrak{b}^1_\varepsilon \colon \mathcal{A}^{t_1} \to \mathcal{B}^{t_1}$ as follows:
        take arbitrarily $p \in P$ and, for every $a \in \mathcal{A}^{t_1}$, put $\mathfrak{b}^1_\varepsilon[a] \coloneq \mathfrak{b}_\varepsilon^0 [a^0]|_{[t_1, T]}$, where
        \begin{equation*}
            a^0(t)
            \coloneq \begin{cases}
                p, & \text{if } t \in [t_0, t_1) \\
                a(t), & \text{if } t \in [t_1, T].
            \end{cases}
        \end{equation*}
        Then, $\mathfrak{b}^1_\varepsilon$ is non-anticipating since $\mathfrak{b}_\varepsilon^0$ is non-anticipating, i.e.,
        we have $\mathfrak{b}^1_\varepsilon \in \mathfrak{B}^{t_1}$.
        Take $a_\varepsilon^1 \in \mathcal{A}^{t_1}$ such that
        \begin{equation*}
            \sup_{\mathfrak{b} \in \mathfrak{B}^{t_1}} \inf_{a \in \mathcal{A}^{t_1}}
            J(t_1, x_1, a, \mathfrak{b}[a])
            \geq \inf_{a \in \mathcal{A}^{t_1}}
            J(t_1, x_1, a, \mathfrak{b}_\varepsilon^1[a])
            \geq J(t_1, x_1, a_\varepsilon^1, \mathfrak{b}_\varepsilon^1[a_\varepsilon^1]) - \varepsilon / 3.
        \end{equation*}
        Consider
        \begin{equation*}
            a_\varepsilon^0(t)
            \coloneq \begin{cases}
                p, & \text{if } t \in [t_0, t_1) \\
                a_\varepsilon^1(t), & \text{if } t \in [t_1, T].
            \end{cases}
        \end{equation*}
        As a result, we obtain
        \begin{equation*}
            v_+(t_0, x_0) - v_+(t_1, x_1)
            \leq J(t_0, x_0, a_\varepsilon^0, \mathfrak{b}_\varepsilon^0[a_\varepsilon^0])
            - J(t_1, x_1, a_\varepsilon^1, \mathfrak{b}_\varepsilon^1[a_\varepsilon^1]) + 2 \varepsilon / 3.
        \end{equation*}
        Note that, by construction, $a_\varepsilon^1 = a_\varepsilon^0|_{[t_1, T]}$ and $\mathfrak{b}_\varepsilon^1[a_\varepsilon^1] = \mathfrak{b}_\varepsilon^0[a_\varepsilon^0]|_{[t_1, T]}$.
        Hence, using the estimate established above and the choice of $\delta$, we get
        \begin{equation*}
            v_+(t_0, x_0) - v_+(t_1, x_1)
            \leq \varepsilon.
        \end{equation*}

        Now, consider the difference
        \begin{equation*}
            v_+(t_1, x_1) - v_+(t_0, x_0)
            = \sup_{\mathfrak{b} \in \mathfrak{B}^{t_1}} \inf_{a \in \mathcal{A}^{t_1}}
            J(t_1, x_1, a, \mathfrak{b}[a])
            - \sup_{\mathfrak{b} \in \mathfrak{B}^{t_0}} \inf_{a \in \mathcal{A}^{t_0}}
            J(t_0, x_0, a, \mathfrak{b}[a]).
        \end{equation*}
        Take $\mathfrak{b}_\varepsilon^1 \in \mathfrak{B}^{t_1}$ such that
        \begin{equation*}
            \sup_{\mathfrak{b} \in \mathfrak{B}^{t_1}} \inf_{a \in \mathcal{A}^{t_1}}
            J(t_1, x_1, a, \mathfrak{b}[a])
            \leq \inf_{a \in \mathcal{A}^{t_1}}
            J(t_1, x_1, a, \mathfrak{b}_\varepsilon^1[a]) + \varepsilon / 3.
        \end{equation*}
        Define $\mathfrak{b}_\varepsilon^0 \colon \mathcal{A}^{t_0} \to \mathcal{B}^{t_0}$ as follows:
        take arbitrarily $q \in Q$ and, for every $a \in \mathcal{A}^{t_0}$, put
        \begin{equation*}
            \mathfrak{b}^0_\varepsilon[a](t)
            \coloneq \begin{cases}
                q, & \text{if } t \in [t_0, t_1) \\
                \mathfrak{b}_\varepsilon^1 [a|_{[t_1, T]}](t), & \text{if } t \in [t_1, T].
            \end{cases}
        \end{equation*}
        Then, $\mathfrak{b}_\varepsilon^0$ is non-anticipating since $\mathfrak{b}_\varepsilon^1$ is non-anticipating, i.e., we have $\mathfrak{b}_\varepsilon^0  \in \mathfrak{B}^{t_0}$.
        Take $a_\varepsilon^0 \in \mathcal{A}^{t_0}$ such that
        \begin{equation*}
            \sup_{\mathfrak{b} \in \mathfrak{B}^{t_0}} \inf_{a \in \mathcal{A}^{t_0}}
            J(t_0, x_0, a, \mathfrak{b}[a])
            \geq \inf_{a \in \mathcal{A}^{t_0}}
            J(t_0, x_0, a, \mathfrak{b}_\varepsilon^0 [a])
            \geq J(t_0, x_0, a_\varepsilon^0, \mathfrak{b}_\varepsilon^0[a_\varepsilon^0]) - \varepsilon / 3.
        \end{equation*}
        As a result, we obtain
        \begin{equation*}
            v_+(t_1, x_1) - v_+(t_0, x_0)
            \leq J(t_1, x_1, a_\varepsilon^0|_{[t_1, T]}, \mathfrak{b}_\varepsilon^1[a_\varepsilon^0|_{[t_1, T]}])
            - J(t_0, x_0, a_\varepsilon^0, \mathfrak{b}_\varepsilon^0[a_\varepsilon^0]) + 2 \varepsilon / 3.
        \end{equation*}
        Note that, by construction, $\mathfrak{b}_\varepsilon^1[a_\varepsilon^0|_{[t_1, T]}] = \mathfrak{b}_\varepsilon^0[a_\varepsilon^0]|_{[t_1, T]}$.
        Hence, using the estimate established above and the choice of $\delta$, we get
        \begin{equation*}
            v_+(t_1, x_1) - v_+(t_0, x_0)
            \leq \varepsilon.
        \end{equation*}
        The proof is complete.

    \subsection{Proof of Proposition \ref{proposition_measurable_selection}}
    \label{S:proof_proposition_measurable_selection}

        Let $\varepsilon > 0$ be given.
        For every $q \in Q$, consider the set
        \begin{equation*}
            O_\varepsilon(q)
            \coloneq \bigl\{ (p^\prime, q^\prime) \in P \times Q \colon
            |h(p^\prime, q^\prime) - h(p^\prime, q)|
            < \varepsilon \bigr\}.
        \end{equation*}
        Note that the set $O_\varepsilon(q)$ is open.
        Indeed, the function
        \begin{equation*}
            P \times Q \ni (p^\prime, q^\prime) \mapsto
            g(p^\prime, q^\prime) \coloneq |h(p^\prime, q^\prime) - h(p^\prime, q)| \in \mathbb{R}
        \end{equation*}
        is continuous due to continuity of $h$, and $O_\varepsilon(q)$ equals to $g^{- 1} ((- \infty, \varepsilon))$, which is the inverse image of the open set $(- \infty, \varepsilon) \subset \mathbb{R}$.

        For every pair $(p, q) \in P \times Q$, we have $(p, q) \in O_\varepsilon(q)$ since $|h(p, q) - h(p, q)| = 0 < \varepsilon$.
        Hence, the collection $\{O_\varepsilon(q)\}_{q \in Q}$ is an open cover of $P \times Q$.
        Since $P \times Q$ is a compact topological space, there exists a finite subcover, i.e., there exists a number $N \in \mathbb{N}$ and points $\{q_n\}_{n \in \overline{1, N}} \subset Q$ such that the collection $\{O_\varepsilon(q_n)\}_{n \in \overline{1, N}}$ is a cover of $P \times Q$.

        For every $p \in P$, the inequality below holds:
        \begin{equation*}
            \max_{q \in Q} h(p, q)
            \leq \max_{n \in \overline{1, N}} h(p, q_n) + \varepsilon.
        \end{equation*}
        Indeed, take $q_\ast \in Q$ such that
        \begin{equation*}
            \max_{q \in Q} h(p, q)
            = h(p, q_\ast).
        \end{equation*}
        Since $\{O_\varepsilon(q_n)\}_{n \in \overline{1, N}}$ is a cover of $P \times Q$, there exists $n_\ast \in \overline{1, N}$ such that $(p, q_\ast) \in O_\varepsilon(q_{n_\ast})$.
        Then, by the definition of $O_\varepsilon(q_{n_\ast})$, we have
        \begin{equation*}
            |h(p, q_\ast) - h(p, q_{n_\ast})|
            < \varepsilon.
        \end{equation*}
        Hence,
        \begin{equation*}
            h(p, q_\ast)
            \leq h(p, q_{n_\ast}) + \varepsilon
            \leq \max_{n \in \overline{1, N}} h(p, q_n) + \varepsilon,
        \end{equation*}
        and we obtain the desired inequality.

        Define a function $q_\varepsilon \colon P \to Q$ by
        \begin{equation*}
            q_\varepsilon(p)
            \coloneq q_{n_\varepsilon(p)}, \quad p \in P,
        \end{equation*}
        where
        \begin{equation*}
            n_\varepsilon(p)
            \coloneq \min \Bigl\{ n \in \overline{1, N} \colon
            h(p, q_n)
            = \max_{m \in \overline{1, N}} h(p, q_m) \Bigr\},
            \quad p \in P.
        \end{equation*}
        By construction, for every $p \in P$,
        \begin{equation*}
            h(p, q_\varepsilon(p))
            = h(p, q_{n_\varepsilon(p)})
            = \max_{m \in \overline{1, N}} h(p, q_m)
            \geq \max_{q \in Q} h(p, q) - \varepsilon.
        \end{equation*}
        So, in order to complete the proof, it remains to verify that $q_\varepsilon$ is Borel measurable.
        Note that the values of the function $q_\varepsilon$ belong to the finite set $\{q_n \colon n \in \overline{1, N}\}$.
        Therefore, the inverse image $q^{- 1}_\varepsilon(O)$ of an arbitrary open (in fact, any) subset $O$ of $Q$ equals a finite union of some of the sets $q_\varepsilon^{-1}(q_n)$, $n \in \overline{1, N}$.
        Thus, it suffices to show that, for every $n \in \overline{1, N}$, the set $q_\varepsilon^{- 1}(q_n)$, which is the inverse image of $q_n$, is a Borel subset of $P$.

        Let $n \in \overline{1, N}$ be fixed.
        We have
        \begin{equation*}
            q_\varepsilon^{- 1}(q_n)
            = \Bigl\{ p \in P \colon
            h(p, q_n)
            = \max_{m \in \overline{1, N}} h(p, q_m), \,
            h(p, q_j)
            < \max_{m \in \overline{1, N}} h(p, q_m), \,
            j \in \overline{1, n - 1} \Bigr\}.
        \end{equation*}
        Hence, introducing the sets
        \begin{equation*}
            \begin{aligned}
                P_j
                & \coloneq \Bigl\{ p \in P \colon
                h(p, q_j)
                = \max_{m \in \overline{1, N}} h(p, q_m) \Bigr\}, \\
                P_j^\circ
                & \coloneq \Bigl\{ p \in P \colon
                h(p, q_j)
                < \max_{m \in \overline{1, N}} h(p, q_m) \Bigr\},
                \quad j \in \overline{1, N},
            \end{aligned}
        \end{equation*}
        we get
        \begin{equation*}
            q_\varepsilon^{- 1}(q_n)
            = P_n \bigcap \Bigl( \bigcap_{j \in \overline{1, n - 1}} P_j^\circ \Bigr).
        \end{equation*}
        Note that, for every $j \in \overline{1, N}$, the function
        \begin{equation*}
            P \ni p \mapsto \ell_j(p) \coloneq h(p, q_j) - \max_{m \in \overline{1, N}} h(p, q_m) \in \mathbb{R}
        \end{equation*}
        is continuous.
        Then, we have $P_j = \ell_j^{- 1}(\{0\})$, which means that $P_j$ is a closed subset of $P$ as the inverse image of the closed set $\{0\} \subset \mathbb{R}$, and $P_j^\circ = \ell_j^{- 1}((- \infty, 0))$,  which means that $P_j^\circ$ is an open subset of $P$ as the inverse image of the open set $(- \infty, 0) \subset \mathbb{R}$.
        As a result, we find that $q_\varepsilon^{- 1}(q_n)$ is a Borel set as the intersection of the closed set $P_n$ and the intersection of the open sets $P_j^\circ$, $j \in \overline{1, n - 1}$.
        The proof is complete.

\textbf{\large Statements and Declarations.}

\textbf{Funding.}
The research of E.~Bayraktar was supported in part by NSF Grant DMS-2507940.
The research of C.~Keller was supported in part by NSF Grant DMS-2106077.

 \textbf{Competing Interests.} The authors declare that they do not have any competing interests.

\end{document}